\documentclass[12pt]{amsart}

\usepackage{amsfonts,amssymb,latexsym,amsmath,amscd,tikz}
\usepackage{youngtab}
\usepackage{mathrsfs}
\usepackage{times}
\usepackage{anysize}

\marginsize{1in}{1in}{1in}{1in}

\newcommand{\CC}{\mathbb{C}}

\newcommand{\eee}{\mathbf{e}}

\newcommand{\NN}{\mathbb{N}}

\newcommand{\PP}{\mathbb{P}}

\newcommand{\ZZ}{\mathbb{Z}}

\newcommand{\Hh}{\mathbf{H}}
\newcommand{\Hhu}{\overline{\mathbf{H}}}

\newcommand{\Pp}{\mathbf{P}}

\newcommand{\aA}{\mathcal{A}}
\newcommand{\bB}{\mathcal{B}}

\newcommand{\dD}{\mathcal{D}}

\newcommand{\fF}{\mathcal{F}}
\newcommand{\gG}{\mathcal{G}}
\newcommand{\hH}{\mathcal{H}}

\newcommand{\lL}{\mathcal{L}}
\newcommand{\mM}{\mathcal{M}}
\newcommand{\nN}{\mathcal{N}}
\newcommand{\oO}{\mathcal{O}}
\newcommand{\pP}{\mathcal{P}}

\newcommand{\sS}{\mathcal{S}}

\newcommand{\hh}{\mathfrak{h}}
\renewcommand{\gg}{\mathfrak{g}}
\renewcommand{\aa}{\mathfrak{a}}

\newcommand{\wtp}{\mathfrak{w}}
\newcommand{\htp}{\mathfrak{h}}
\newcommand{\uhtp}{\overline{\mathfrak{h}}}
\newcommand{\Hmu}{\widetilde{H}_{\mu}}

\newtheorem{lemma}{Lemma}[section]
\newtheorem{conjecture}[lemma]{Conjecture} 
\newtheorem{corollary}[lemma]{Corollary} 
\newtheorem{theorem}[lemma]{Theorem} 
\newtheorem{definition}[lemma]{Definition} 
 
\newtheorem{proposition}[lemma]{Proposition}

\theoremstyle{remark} 
\newtheorem*{remark}{Remark}
\newtheorem*{example}{Example}

\DeclareMathOperator{\ecH}{H}

\DeclareMathOperator{\gr}{gr}
\DeclareMathOperator{\Hilb}{Hilb}
\DeclareMathOperator{\Spec}{Spec}
\DeclareMathOperator{\Coef}{Coef}

\DeclareMathOperator{\Hom}{Hom}

\DeclareMathOperator{\ch}{ch}
\DeclareMathOperator{\Tr}{Tr}
\DeclareMathOperator{\pf}{PF}
\DeclareMathOperator{\tr}{tr}
\DeclareMathOperator{\Sp}{Sp}

\newcommand{\para}{\widetilde}

\newcommand{\nild}{\underline{D}}

\newcommand{\CAns}{\mathbf{H}}
\newcommand{\CA}{\mathbf{H}_{m/n}}
\newcommand{\CAc}{\mathbf{H}_{c}}
\newcommand{\uCAc}{\overline{\mathbf{H}}_{c}}

\newcommand{\Hbar}{\hH}
\newcommand{\Mbar}{\overline{M}}
\newcommand{\Hmn}{{H}_{m/n}}
\newcommand{\Hmnf}{{H}_{m/n}^{\fF}}
\newcommand{\Hin}{{H}_{\infty/n}}
\newcommand{\Ru}{\overline{R}}
\newcommand{\lmn}{L_{m/n}}
\newcommand{\lnm}{L_{n/m}}
\newcommand{\lmnprev}{L_{(m-n)/n}}
\newcommand{\fmn}{F_{m/n}}
\newcommand{\lin}{L_{\infty/n}}
\newcommand{\hhu}{\overline{\mathfrak{h}}}
\newcommand{\ophi}{\overline{\varphi}}

\title{Torus knots and the rational DAHA}

\author[E. Gorsky, A. Oblomkov, J. Rasmussen, and V. Shende]{  Eugene Gorsky\and  Alexei Oblomkov  \and Jacob Rasmussen  \and  Vivek Shende \\ \\}

\begin{document}

\begin{abstract}
We conjecturally extract the triply graded Khovanov--Rozansky homology of the $(m,n)$ torus knot from 
the unique finite dimensional simple representation of the rational DAHA of type A, rank $n-1$, and central
character $m/n$.  
The conjectural differentials of Gukov, Dunfield and the third author receive an explicit algebraic 
expression in this picture, yielding a prescription for the doubly graded Khovanov--Rozansky homologies. 
We match our conjecture to previous
conjectures of the first author relating knot homology to $q,t$-Catalan numbers, and of the last three
authors relating knot homology to Hilbert schemes on singular curves. 
\end{abstract}

\maketitle

% \tableofcontents

\section{Introduction}

Given a knot $K$ and an integer $N$, 
Khovanov and Rozansky construct a  {\em doubly graded} vector space 
$\hH_{N,K}$ categorifying the quantum $\mathfrak{sl}_N$ 
invariant of $K$ (in the fundamental representation) \cite{KR1}.  
In particular, its Poincar\'e polynomial $\pP_{N,K}(q,t)$ 
specializes at $t = -1$ to the 
$\mathfrak{sl}_N$ knot polynomial $P_{N,K}(q)$. 
They also construct a {\em triply graded} vector space 
$\hH_K$ whose Poincar\'e polynomial $\pP_K(a,q,t)$ 
specializes at $t = -1$ to the HOMFLY polynomial $P_K(a,q)$ \cite{KR2}.  
While the HOMFLY polynomial collects the data of the 
$\mathfrak{sl}_N$ polynomials, $P_K(q^N,q) = P_{N,K}(q)$,
this {\em does not hold} for the Poincar\'e series
of the categorifications: $\pP_K(q^N, q, t) \ne \pP_{N,K}(q,t)$. 
However, there is a spectral sequence 
$\hH_K \to \hH_{N,K}$ constructed by the third author \cite{R}, which
establishes part of a richer conjectural pictured advanced in \cite{DGR}.  
In all known cases, the spectral sequence converges after the first differential. 

Let $K$ be a $(n,m)$ torus knot.  While these
are among the simplest of all knots, no rigorous calculation of the invariants 
$\pP_K$ or $\pP_{N,K}$ has been carried out even when $n=3$ (except the case $N=2$, where the answer was computed in 
\cite{Tu}). 
By contrast, an explicit 
formula for $P_K$ was given by Jones \cite{Jon}.  
Some recent works \cite{mm, as, ch, ORS} give conjectural formulas\footnote{
None of these are given in closed form.  In \cite{ORS} the answer takes the form of a sum over
diagrams whose evaluation requires times measured in minutes for e.g. a $(13,20)$ torus knot.}  for 
$\pP_K$ with various differing (often physical) motivations; in all tested cases
the formulas agree both with each other and the actual values of $\pP_K$.  
However thus far there has been no 
prescription for
either $\pP_{N,K}$ or the differential recovering $\hH_{N,K}$ from $\hH_K$. 

We describe here conjectures which recover the homologies of the torus knots -- both
$\hH_K$ and $\hH_{N,K}$ -- from the simple representations of the rational Cherednik algebras
of type A \cite{EG}.  Taking $\hh$ to be the $n-1$ dimensional reflection representation of $S_n$, recall
that for $c \in \CC$ the rational Cherednik algebra $\Hh_c$ contains as subalgebras $\CC[\hh]$,
$\CC[\hh^*]$, and $\CC[S_n]$, and these together generate it.  Permutations $\sigma \in S_n$ interact with
$\CC[\hh]$ and $\CC[\hh^*]$ via $\sigma h = \sigma(h) \sigma$.  The elements of $\hh$
and $\hh^*$ are subject to the following commutation relation: 

\[ [y,x] = \langle y, x \rangle - c \sum_{s \in \sS}  \langle y, \alpha_s \rangle \langle \alpha_s^\vee, x \rangle \cdot s 
\,\,\,\,\,\,\,\,\,\,\,\,\,\,\,\,\,\,\, \mbox{for $y \in \hh$ and $x \in \hh^*$.}\]

\noindent Here $\sS$ is the set of transpositions, and  $\alpha_s, \alpha_s^\vee$ are the root and coroot corresponding
to $s \in \sS$. 

One generally restricts attention to the category $\oO$ of 
finitely generated representations in which $\hh$ acts locally nilpotently.  Among these are the {\em standard modules},
constructed from an irreducible representation $\tau$ of $S_n$ by letting $P \in \CC[\hh^*]$ act on $\tau$ by the scalar $P(0)$, and then
inducing $M_c(\tau) := \Hh_c \otimes_{\CC[\hh^*] \rtimes W} \tau$.  These have a unique
simple quotient $L_c(\tau)$, and these quotients account for all simples \cite{DO}.   
For most $c$, the map $M_c(\tau) \to L_c(\tau)$ is an
isomorphism, and it is known \cite{BEG2} that $L_c(\tau)$ is finite dimensional if and only if
$c = m/n$ with $m,n$ relatively prime integers and moreover $\tau$ is the trivial representation if $m > 0$ and the sign
representation otherwise.  We denote this representation simply $L_{m/n}$. 

As the commutation relations in $\Hh_c$ preserve the difference between $\hh^*$-degree and $\hh$-degree, the algebra 
$\Hh_c$ is graded by this difference.  In fact the grading is internal: there is an element $\mathbf{h} \in \Hh_c$ 
which acts semisimply and whose eigenvalues give the grading.  It can be shown that $\mathbf{h}$ acts semisimply
and with integer eigenvalues on every module in category $\oO$, so these all acquire a grading compatible with that
of the algebra.  We can now state: 

\begin{proposition} \label{prop:eul}
	The HOMFLY polynomial of the $(n,m)$ torus knot is given by:
	\[ \Pp_{(n,m)} =  a^{(n-1)(m-1)} \sum_{i=0}^{n-1} a^{2i} \tr(q^{\mathbf{h}}; \Hom_{S_n} (\Lambda^i \hh, L_{m/n} )) \]
\end{proposition}

The left hand side has been calculated by Jones \cite{Jon} and the right hand side by Berest, Etingof, and Ginzburg \cite{BEG2}; in Section
\ref{sec:Eulerchar} we show they agree.

To recover the HOMFLY homology we need an additional grading.  The algebra
$\Hh_c$ does not admit a second grading, but it {\em does} admit a filtration in which
$S_n$ occupies filtration degree zero and $\hh, \hh^*$ both occupy filtration degree $1$.  

\begin{conjecture}
\label{conj:one}
\(L_{m/n}\) admits a filtration
$\ldots \fF_{i-1}\subset\fF_{i}\subset \fF_{i+1}\ldots$ compatible with the filtration 
of $\Hh_{m/n}$ (in particular \(gF_{i} \subset F_i\) for
\(g \in S_n\)) and the grading induced by $\mathbf{h}$ 
so that 
$$ \Hbar_{T(m,n)} \cong \Hom_{S_n}(\Lambda^*\hh ,\gr^{\fF} L_{m/n})=: \Hmn^{\fF}.$$
\end{conjecture}

The homology group $\hH_{T(m,n)}$ is triply graded, and so is $\Hmn^{\fF}$: it has one grading 
from the degree in the exterior algebra $\Lambda^* V$; one grading 
from the element $\mathbf{h}$; and one from the filtration $\fF$.  
Roughly speaking, these correspond respectively to the gradings
measured by $a$, $q$, and $t$ in the polynomial $\pP_K$; 
a precise statement appears at the end of the introduction.

For the filtration $\fF$ we have {\em three} candidates arising from
different descriptions of $L_{m/n}$.  

\begin{itemize}
 \item  $\fF^{alg}$ is the most explicit.  Inside
  the polynomial ring $\CC[x_1,\ldots,x_n]$ we consider the ideal
  $\aa$ generated by symmetric functions of positive degree, and the 
  filtration by its powers.  The space $L_{m/n}$ is a quotient of 
  $\CC[x_1,\ldots,x_n]$ on which $\Hh_{m/n}$ acts by $S_n$-twisted 
  differential operators, and so also carries a filtration by powers of 
  $\aa$.  There is a certain non degenerate pairing $(\cdot, \cdot)_{m/n}$ 
  on $L_{m/n}$.  Up to a shift by the grading, $\fF^{alg}$ is the filtration dual 
  to that arising from powers of $\aa$.  
 \item  $\fF^{ind}$ comes from relations between $L_c$ for varying values of $c$. 
  Specifically, if $m > n$, there is a way to build $L_{m/n}$ from $L_{(m-n)/n}$ 
  which induces a filtration on the former from one on the latter.  Moreover, although
  $L_{m/n}$ and $L_{n/m}$ are rather different (having in particular different dimensions),
  their ``spherical'' parts are canonically isomorphic, and the whole 
  representation may be reconstructed from its spherical part. 
  Thus using the Euclidean algorithm we may induce a filtration on any $L_{m/n}$ starting
  from the trivial filtration on the one dimensional space $L_{1/r}$.    
 \item $\fF^{geom}$ comes from the realization of $L_{m/n}$ as the cohomology of a certain
  Hitchin fibre \cite{OY}; it measures the difference between the homological degree and the grading.  
\end{itemize}

These various filtrations offer different advantages.  The filtration $\fF^{alg}$ is the most
easily computable in any given case.  
The inductive
filtration $\fF^{ind}$ is defined so as to interact well with results of 
Gordon and Stafford \cite{GS, GS2} which give formulas for 
the (triply graded) Poincar\'e series of $H_{m/n}^{\fF^{ind}}$ in terms of equivariant 
Euler characteristics of coherent sheaves on the Hilbert scheme of points 
on $\CC^2$.  This yields structural predictions and in some cases 
explicit formulas as in \cite{mm} and \cite[Conj. 8]{ORS}.  In particular,
when $m=n+1$ the representation $L_{(n+1)/n}$ is known to be 
isomorphic to the space of diagonal harmonics introduced and studied by M. Haiman \cite{haiman,haiman2} and in this case we find certain well known $q,t$-symmetric
polynomials appearing as the coefficients of $a$ in $\pP_{T(m,n)}$. These polynomials were conjectured to model knot homology in \cite{G}.

The geometric filtration $\fF^{geom}$ is defined so that, with this choice of filtration, Conjecture \ref{conj:one} can be derived from the main conjecture of \cite{ORS} relating the HOMFLY homology to Hilbert schemes of points on singular curves.  The basic point is that the representation $L_{m/n}$ can be realized geometrically \cite{Y, OY} as the cohomology of a certain parabolic Hitchin fibre corresponding to a spectral curve with singularity $x^m = y^n$; we will explain this further in Section \ref{sec:hilb}. This fibre admits an affine cell decomposition which is expected to be related to the combinatorial models for diagonal harmonics. 

The filtration $\fF^{ind}$ will be defined in Section \ref{sec:filt} as the minimal filtration satisfying certain properties.  After we verify these for $\fF^{alg}$ and $\fF^{geom}$, 
we will conclude $\fF^{ind} \subset \fF^{alg}$ and $\fF^{ind} \subset \fF^{geom}$.  Numerical evidence leads us to believe: 

\begin{conjecture}
  The three filtrations $\fF^{alg}, \fF^{ind}, \fF^{geom}$ are the same.
\end{conjecture}

\vspace{4mm}
We turn to the evidence supporting
Conjecture~\ref{conj:one}.  Proposition \ref{prop:eul} asserts the conjecture holds at the level of Euler characteristics.
When $n=2$ the HOMFLY homology $\hH_{T(2,m)}$ was calculated in 
in \cite{KhSoergl}, and for \(n=3\) a formula for  
$\pP_{T(3,m)}$ was conjectured in \cite{DGR}; we will check in Section \ref{sec:Examples} that 
these match the prediction of Conjecture \ref{conj:one}.  

At present there are no systematic calculations 
of $\hH_{T(n,m)}$ for larger $n$.  We can however check that certain structural properties which
are known or conjectured for the HOMFLY homology are manifestly present in the $H_{m/n}$. 
The first nontrivial check comes from the fact that although \(T(m,n)\) and \( T(n,m)\) are the same knots,
 the algebras $\Hh_{m/n}$ and $\Hh_{n/m}$ are not isomorphic, 
 and the representations $L_{m/n}$ and $L_{n/m}$ do not
 even have the same dimension as vector spaces.  Nevertheless, it is possible to  check that 
\begin{proposition}
\(\Hmn \cong {H}_{n/m}\) as triply-graded groups
\end{proposition}
\noindent as would be  expected from the conjecture.
The coincidence of $q$-gradings for these spaces was proved earlier in \cite{CEE} at $a=0$,
and in \cite{G1} for general $a$.  The filtration $\fF^{ind}$ is only really defined for $m/n > 1$ so in
this case we declare the above statement to hold by definition.  The other two filtrations 
make sense for all $m/n$ and the conjecture predicts a
coincidence of filtrations under the identification $\Hh_{m/n} \cong \Hh_{n/m}$.  
For $\fF^{alg}$ this follows from \cite{G1}, and for $\fF^{geom}$ it follows
from the geometric construction \cite{OY} of $L_{m/n}$. 

It was conjectured
   in \cite{DGR} that, for any knot \(K\), the triply graded homology \(\Hbar_K\) admits an
   involution generalizing the \(q \to q^{-1}\) symmetry of
   the HOMFLY polynomial.  The corresponding ``Fourier transform'' $\Phi$ 
   on $L_{m/n}$ is well known (we recall it in Section \ref{sec:daha}) 
   and, consequently, \(\Hmn\) has the desired symmetry. 

A second conjecture in \cite{DGR} states that the group \(\Hbar_K\)
should be equipped with differentials \(d_N:\Hbar_K\to
\Hbar_K\) for all \(N \in \ZZ\). 
These differentials are supposed to be mutually anticommuting.
For \(N>0\), the differentials should recover the $\mathfrak{sl}_N$ homology
via \(\ecH(\hH_K,d_N)= \hH_{N,K}\).
A weakened
form of this statement, with the single differential replaced by a
spectral sequence, was proved in \cite{R}. It is thus natural to
look for such differentials on \(\Hmn\). 

\begin{proposition}
\label{prop:diffs}
For all \(N \in \ZZ\), there are mutually anticommuting 
differentials \(d_N:\Hmn \to \Hmn\). The behaviour of these maps with
respect to the triple grading is compatible with that predicted by
\cite{DGR}. Moreover, the involution \(\Phi:\Hmn \to \Hmn \) 
exchanges \(d_N\) and \(\pm d_{-N}\). 
\end{proposition}

\begin{conjecture}
 For $N>0$, we have $\ecH(\Hmn,d_N) = \hH_{N,T(m,n)}$.
\end{conjecture}

As $m \to \infty$, the low $q$-degree terms of $\hH_{T(n,m)}$ stabilize
\cite{Stosic} ({\it cf.} \cite{DGR},
\cite{RoJW}.) More precisely,
we define 
$$ \hH_{T(n,\infty)} = \lim_{m\to \infty} 
\frac{\hH_{T(n,m)}}{(aq^{-1})^{(m-1)(n-1)}}$$
Here the powers of $a, q$ should be interpreted as shifting the 
$a, q$ gradings, and the shifts are arranged so that the unique term in
\(\hH_{T(n,m)}\) with homological grading \(0\) is in \((a,q)\)
bigrading \((0,0)\) after the shift. 

In analogy with this construction, we consider the limit
$$ \Hin = \lim_{m\to \infty} \Hmn.$$
In this limit, the representation theory of the DAHA simplifies
considerably, and we find that \(\Hin\) is the tensor product of an
exterior algebra with a polynomial algebra, as predicted by
\cite{DGR} (see also \cite{G}):
$$ \Hin \cong \Lambda^*(\xi_1,\ldots,\xi_{n-1})\otimes
\CC[u_1,\ldots,u_{n-1}].$$

\begin{proposition} The differential $d_N$ on $\Hin$ obeys the graded
Leibnitz rule, must vanish on $u_i$ by grading considerations,
and is given on $\xi_i$ by the formula
\label{prop:stabled}
$$ d_N(\xi_i) = \!\!\!\!\! \sum_{i_1+\ldots i_N=i} \!\!\!\!\! u_{i_1}u_{i_2}\ldots
u_{i_N}.$$
Equivalently, introducing the following elements of $\CC[\xi_i, u_j, t]/t^N$,
\begin{equation*}
  {\xi} = \xi_1t + \xi_2t^2 +\ldots + \xi_{n-1}t^{n-1} \quad \text{and} \quad
{u}  = u_1t + u_2t^2 +\ldots + u_{n-1}t^{n-1}, 
\end{equation*}
the formula is expressed by $d_N{\xi} = u^N$.
\end{proposition}

This gives the following conjecture, which is stripped of all explicit mention of the DAHA. 

\begin{conjecture}
\label{conj:two}
The stable reduced $\mathfrak{sl}_N$ homology of \(T(n,\infty)\) is
$$\ecH(\Lambda^*(\xi_1,\ldots,\xi_{n-1})\otimes
\CC[u_1,\ldots,u_{n-1}],d_N)$$
where \(d_N\) is defined as above. 
\end{conjecture}

When \(N=2\), the \(sl(2)\) homology is isomorphic to the original
Khovanov homology. Thanks to Bar-Natan, Morrison and Shumakovitch (\cite{DBN},\cite{katlas},\cite{KhoHo}), there are extensive
computations of Khovanov homology for torus knots.
Conjecture~\ref{conj:two} agrees with this data 
as far as we are able to compute it. For example, 
Conjecture~\ref{conj:two} correctly predicts the reduced Khovanov homology of
\(T(6,25)\) in \(q\)--degree \(\leq 50\); the total dimension of the
group in question is 793, with 267-dimensional part of \(q\)--degree \(\leq 50\). 
The Poincar\'e series of the stable homology turns out to be related to the Rogers-Ramanujan identity.
 For more details and a discussion of the unreduced Khovanov homology, see \cite{GOR}. 

\vspace{2mm}

In Section \ref{sec:hilb}, we explain the relation of Conjecture \ref{conj:one} to 
\cite[Conj. 2]{ORS}.  This previous conjecture proposes that if $L$ is the link of a 
plane curve singularity,
then the {\em unreduced} HOMFLY homology $\overline{\hH}_L$ is given as the direct sum of the cohomologies of
certain nested Hilbert schemes of points.  We show here that these cohomologies may be extracted 
from the cohomology of a parabolic Hitchin fibre where the spectral curve carries the desired singularity.
The argument involves Springer theory, both the usual kind and the global Springer theory of Z. Yun \cite{Y}, 
and various support theorems \cite{Y, MY, MS, OY} 
which build upon the work of Ng\^o \cite{Ngo-fundamental}.

Taking the singularity to be of the form $x^m = y^n$ puts us 
in the situation of the present article.  In this case, Yun and the second 
author have constructed the action of the rational Cherednik algebra 
on (a factor of) the cohomology of a parabolic Hitchin fibre 
with this spectral curve
\cite{OY}; they show that the resulting representation is $L_{m/n}$.  In the process
one takes the associated graded of the perverse filtration on the cohomology, yielding
the grading on $L_{m/n}$.  There is a filtration $\fF^{geom}$ compatible with the algebra action which
comes from the difference between the cohomological grading
and the perverse grading.  With this choice of filtration, Conjecture \ref{conj:one} 
is shown to be equivalent to the special
case of \cite[Conj. 2]{ORS} where the singularity is $x^m = y^n$.  

It should be remarked however that we have no direct construction of the $d_N$ in terms of the 
nested Hilbert schemes of \cite{ORS}.  More importantly, we have no construction 
of them at all for the links of singularities without a $\CC^*$ action; this action plays a crucial role in
the construction by \cite{OY} of the rational DAHA representation.  On the other hand the original
physical ideas \cite{OV, GSV} leading to the prediction in \cite{DGR} of the differentials 
did not depend on such an action.  Recently there has been an explicit physical derivation
of \cite[Conj. 2]{ORS} along these lines \cite{DSV, DHS}; in this construction it is the nested Hilbert schemes 
rather than the Hitchin fibres which appear.  This suggests two challenges: on the one hand,
to explain on physical grounds the appearance of the Hitchin fibres, and on the other, to construct
the differentials directly in terms of the nested Hilbert schemes.

\vspace{2mm} \noindent {\bf Grading conventions.} 

The group \(\Hbar_K\) is triply graded; that is
$$ \Hbar_K \cong \bigoplus_{{ v} \in \ZZ^3} \Hbar^{ v}_K .$$
To express this grading in the more conventional form
$$ \Hbar_K \cong \bigoplus_{i,j,k\in \ZZ} \Hbar^{i,j,k}_K $$
we must pick a basis for  \(\ZZ^3\), or equivalently, projections \(p_1,p_2,p_3:\ZZ^3\to \ZZ\) in the direction of the basis vectors. We will refer to any such projection as a {\it grading} on \(\Hbar\). 

Two of these three projections arise naturally from the requirement that the graded Euler characteristic of \(\Hbar\) is the HOMFLY-PT polynomial:
$$ P_K =  \sum_{v \in \ZZ^3} (-1)^{\pi(v)} a^{p_1(v)}q^{p_2(v)} \dim \Hbar^v_K.$$
We refer to these as the \(a\)-grading and \(q\)-grading. With our normalization of the HOMFLY-PT polynomial, these gradings are always even when \(K\) is a knot. Their (conjectured) relation to the triple grading on the groups 
\(\Hmn^\fF=\oplus_i \Hom_{S_n}(\Lambda^i \hh,\gr \lmn)\) is easily described. The \(q\)-grading is 
{\em twice} the internal grading on \(\lmn\), while the \(a\)-grading is 
$$a=\mu(K)+2i = (n-1)(m-1) + 2i.$$

There are several natural choices for the third grading on \(\Hbar_K\). Some are described below:
\begin{enumerate}
\item {\bf Homological grading:}
This is the homological grading defined by Khovanov and Rozansky in
\cite{KR2}, and is equal to one half of the grading \(gr_h\) used in \cite{R}.
 We will refer to it as the \(h\)-grading. With respect to this grading, the Poincar{\'e} polynomial of \(\Hbar_{T(2,3)}\) is \(a^2q^{-2}h^1 + a^4q^0h^{-1} + a^2q^2h^{-1}\). 
\item {\bf \(t\)-grading} This is the grading \(t=-h+a/2\). It gives the homological grading on Khovanov homology and the negative of the homological grading on $\mathfrak{sl}_N$ homology \cite{R}. If \(K\) is an algebraic knot, the conjecture of \cite{ORS} matches the \(t\)-grading with the homological grading on the cohomology of the compactified Jacobian
for those groups with minimal \(a\)-grading (\(a = \mu(K))\). 
\item {\bf \(\delta\)-grading:} This is the grading \(\delta = a+q+2h = 2a+q-2t\). It is preserved under the conjectured symmetry of \(\Hbar_K\) \cite{DGR} and is constant with value \(-\sigma(K)\) if \(K\) is a two-bridge knot \cite{R}.  (Our convention for the signature is that positive knots have negative signature.)
\item {\bf filtration grading:} If \(K\) is an algebraic knot, we define the {\it filtration grading}  to be 
$$f = (\delta-a+\mu(K))/2 = (q+\mu(K))/2 +h.$$
For torus knots, we conjecture that this grading coincides with the grading induced by the filtration $\fF^\bullet$ on 
\(\Hmn\). 
\end{enumerate}

\section*{Acknowledgements}

We are grateful to M. Aganagic, R. Bezrukavnikov, I. Cherednik, D.-E. Diaconescu, V. Ginzburg, L. G\"ottsche, I. Grojnowski, S. Gukov, T. Hausel, M. Khovanov,  A. Kirillov Jr., I. Losev, M. Mazin, L. Migliorini, A. Okounkov, L. Rozansky, S. Shakirov, Y. Soibelman, and especially 
P. Etingof and Z. Yun for useful discussions. The research of E. G. was partially supported by the grants RFBR-10-01-00678, NSh-8462.2010.1 and the Simons foundation. 
A. O. was partially supported by NSF and Sloan Foundation. J. R. would like to thank the Simons Center for Geometry and Physics for their hospitality and support while this work was conducted. The research of V. S.  is supported by the Simons foundation.

\section{The rational Cherednik algebra}
\label{sec:daha} 

In this section we review the definition and basic properties
of the rational Cherednik algebra, or rational DAHA, as defined in
\cite{EG}.  
All of the material in this section is well known, and we discuss only type \(A_n\), 
and only topics which are immediately relevant to our 
purposes. More complete introductions can be found in
 \cite{CherednikNotes} and \cite{EtingofNotes}. 

\subsection{Dunkl Operators}
Our starting point is the ring \(\Ru = \CC[x_1,\ldots,x_n]\).  It will be convenient
to write $X_i$ for the operator on $\Ru$ which multiplies a polynomial by $x_i$.  
The symmetric group \(S_n\) acts on \(\Ru\) by permuting the \(x_i\)'s.  Clearly
$$[X_i,X_j]=0 \quad \text{and} \quad s X_i s^{-1} = X_{s(i)}. $$

The same relation holds for the operators $\partial/\partial x_i$ 
of partial differentiation with respect to the $x_i$.  The {\it Dunkl operators} 
\cite{Du} \(D_i:\Ru \to \Ru \) deform the partial derivatives in such a way
that the relations continue to hold.  Fixing $c \in \CC$, the Dunkl operators
are defined by 
$$
D_i(p) = \frac{\partial p}{\partial x_i} + c \sum_{j \neq i} \frac{s_{ij}p - p}{x_i-x_j}
$$
where \(s_{ij}\) is the transposition that exchanges \(i\) and
\(j\).  A calculation shows 
$$[D_i,D_j]=0 \quad \text{and} \quad s D_i s^{-1} = D_{s(i)}. $$

The operators \(X_i\) and \(D_j\) do not commute, but their failure to
do so can be expressed in a nice form. It is convenient to
consider the vector spaces 
$$ \uhtp = \text{span}\langle D_1,D_2,\ldots, D_n \rangle \quad
\text{and} \quad \uhtp^* = \text{span}\langle X_1,X_2,\ldots, X_n
\rangle.$$ Elements of \(\uhtp\) and \(\uhtp^*\) are clearly linear
operators acting on \(\Ru\). We can identify \(\uhtp\) and \(\uhtp^*\) with the
Cartan subalgebra and dual Cartan subalgebra for \(\mathfrak{gl}_n\),
where the Weyl group \(W=S_n\) acts on operators by conjugation. In
particular, there is a pairing \(\langle \cdot, \cdot
\rangle: \uhtp \times \uhtp^* \to \CC\) defined by 
$$\langle D_i, X_j \rangle = \delta_{ij} $$

Then for \(v \in \uhtp\), \(w \in \uhtp^*\), it can be verified that
$$ [v,w] = \langle v, w \rangle - c \sum_{s\in \mathcal{S}} \langle
v,\alpha_s\rangle \langle \alpha_s^\vee, w \rangle \cdot s $$
where \(\mathcal{S} \subset S_n\) is the subset of all transpositions
(these act by reflections on \(\uhtp\) and \(\uhtp^*\)), and \(\alpha_s\) and
\(\alpha_s^\vee\) are the corresponding roots and
coroots. Explicitly, we have \(\alpha_{s_{ij}} = X_i-X_j\) and 
 \(\alpha_{s_{ij}}^\vee = D_i-D_j\).

\subsection{The algebra}
The rational Cherednik algebra \(\uCAc\) associated to 
\(\mathfrak{gl}_n\)
is defined  to be 
the associative algebra over \(\CC\) generated by $\uhtp$, $\uhtp^*$, 
and $S_n$, 
subject to relations
analogous to the ones considered in the previous subsection. 
That is, for any $x, x' \in \uhtp^*$, any $y, y' \in \uhtp$, and any $s \in S_n$, we have:

\begin{align}
\label{eq:DahaRelation1}
[x,x']& =0 \quad \qquad s x s^{-1} = s(x) \\
\label{eq:DahaRelation2}
[y,y']& =0  \quad  \qquad s y s^{-1} = s(y) \\
\label{eq:DahaRelation3}
\text{and}  \quad [y,x] & = \langle y, x \rangle - c \sum_{s\in \mathcal{S}} \langle
y,\alpha_s\rangle \langle \alpha_s^\vee, x \rangle \cdot s 
\end{align}
where \(s(x), s(y)\) indicate the action of  \(S_n\) on
\(\uhtp^*\) and \(\uhtp\).

\subsubsection{Triangular decomposition} The relations ensure that any element of $\uCAc$ can be written in the form
$\sum_\alpha P_\alpha(y) \cdot \sigma_\alpha \cdot Q_\alpha(y)$ for some polynomials
$P_\alpha, Q_\alpha$ and elements $\sigma_\alpha \in S_n$.  
In other words, the multiplication map 
$\CC[\hh] \otimes \CC[S_n] \otimes \CC[\hh^*] \to \uCAc$ is surjective.
The same is true for other orderings of the factors.  In fact, the map
above is an isomorphism; this is the ``PBW theorem'' for Cherednik algebras
\cite{EG}. 

\subsubsection{Grading} Consider the free noncommutative algebra 
$A = \CC\langle x_i, y_j, S_n \rangle$; there is a surjective
map $A \to \uCAc$ with kernel given by the above relations.  
The free algebra has a
$\ZZ$-grading in which the elements $x \in \uhtp^*$ have degree $1$, 
the elements $y \in \uhtp$ have degree $-1$, and the elements of $S_n$
have degree $0$.  From Equations \ref{eq:DahaRelation1}, 
\ref{eq:DahaRelation2}, \ref{eq:DahaRelation3} we see this grading descends
to $\uCAc$.  

\subsubsection{Internal grading and $\mathfrak{sl}_2$} 
In fact, the grading is internal, i.e., given as eigenvalues of the adjoint
action of a certain element.  Specifically,
$\overline{\mathbf{h}} = \frac{1}{2} \sum_i (x_i y_i + y_i x_i) \in \uCAc$
has the property \cite[Lem. 2.5]{BEG2} 
that $[\overline{\mathbf{h}}, P] = \deg(P) \cdot P$ for any 
homogeneous $P \in \uCAc$.  This element is included in a
$\mathfrak{sl}_2$ triple.  Let $\overline{\mathbf{x}}^2 := \sum_i x_i^2$ and
$\overline{\mathbf{y}}^2 := \sum_i y_i^2$.  Then evidently 
$[\overline{\mathbf{h}}, \overline{\mathbf{x}}^2] = 2 \overline{\mathbf{x}}^2$
and $[\overline{\mathbf{h}}, \overline{\mathbf{y}}^2] = - 2 \overline{\mathbf{y}}^2$;
one can calculate that $[\overline{\mathbf{x}}^2, \overline{\mathbf{y}}^2] =
-4\overline{\mathbf{h}}$.  So $\overline{\mathbf{h}}, \overline{\mathbf{x}}^2/2,
-\overline{\mathbf{y}}^2/2$ form an $\mathfrak{sl}_2$ triple.  Note this
$\mathfrak{sl}_2$ commutes with $S_n$.

\subsubsection{Filtration} The free algebra $A$ carries an increasing $\NN$-filtration 
$\CC[S_n] = \gG_0 \subset \gG_1 \subset \cdots$ 
in which $\gG_i A$ is spanned as a vector space by elements which can be written as 
products of at most $i$ of the $x$'s and $y$'s, and an arbitrary number of elements
of $S_n$.  Evidently this filtration descends to a filtration, again denoted $\gG_\bullet$,
on $\uCAc$.

\subsubsection{Fourier transform}
\label{fourier for algebra}
 There is a map
$\Phi:\uCAc \to \uCAc$ which acts on generators by carrying
$x_i \to y_i$, $y_i \to -x_i$, and preserving elements of $S_n$.  By inspection
of the relations it gives a well defined algebra isomorphism; by construction
it reverses the grading and preserves the filtration.  
It is shown in \cite[Sec. 3]{BEG2} that in fact the
$\mathfrak{sl}_2$ described above exponentiates to an $\mathrm{SL}_2(\CC)$ 
action by algebra automorphisms on $\uCAc$; $\Phi$ is among these.

\subsection{The \(\mathfrak{sl}_n\) version}
\label{subsec:reducedH}
We will generally work with Cherednik algebra \(\CAc\) associated to
\(\mathfrak{sl}_n\), rather than the \(\mathfrak{gl}_n\) algebra
described above; in particular, the 
representation \(\lmn\) appearing in Conjecture~\ref{conj:one} is a
representation of $\CAc$ rather than $\uCAc$.  

Let $\htp \subset \uhtp$ and $\htp^* \subset \uhtp^*$ be the Cartan 
and dual Cartan for $\mathfrak{sl}_n$; they are 
respectively generated
by the differences $y_i - y_j$ and $x_i - x_j$.  
The algebra \(\CAc\) is
generated by \(x \in \htp^*, y\in \htp\), and \(s \in S_n\), subject
to the relations in
equations~\eqref{eq:DahaRelation1}-\eqref{eq:DahaRelation3}.  Since
\(\htp\) and \(\htp^*\) are \(S_n\)--invariant subsets of \(\uhtp\) and
\(\uhtp^*\), we can view \(\CAc\) as an explicit subalgebra of
\(\uCAc\). As such, it inherits both the grading and the filtration described above;
evidently it is preserved by the Fourier transform $\Phi$. 

Note that the elements $\overline{\mathbf{h}}, \overline{\mathbf{x}}^2, \overline{\mathbf{y}}^2 \in \uCAc$ {\em do not} lie in $\CAc$.  
Denote
$\overline{x} := \frac{1}{\sqrt{n}} \sum x_i$ and $\overline{y} := \frac{1}{\sqrt{n}} 
\sum y_i$.  Then the following {\em are} elements of $\CAc$ which play
the corresponding roles: 
\[\mathbf{h} := \overline{\mathbf{h}} -
\frac{\overline{x}\overline{y} + \overline{y} \overline{x}}{2}, \,\,\,\,\,\,\,\,\,\,\,\,\,\,\,
\mathbf{x}^2 := \overline{\mathbf{x}}^2 - \overline{x}^2, \,\,\,\,\,\,\,\,\,\,\,\,\,\,\,
\mathbf{y}^2 := \overline{\mathbf{y}}^2 - \overline{y}^2.\]

In fact, there is very little difference between $\uCAc$ and $\CAc$. 
Writing $\dD \subset \uCAc$ for the subalgebra generated by 
$\overline{x}, \overline{y}$.
Then it is straightforward to see that $\dD$ is
isomorphic to the algebra of differential operators in one variable, that
$\dD$ commutes with $\CAc$, and that the multiplication map
$\dD \otimes \CAc \to \uCAc$ is an isomorphism.

\subsection{Representations}

By construction, the algebra \(\uCAc\) acts on \(\Ru\).  That is, 
the basis elements $x_i \in \uhtp^*$ act as multiplication by $x_i$, and likewise
the $y_i \in \uhtp$ act by the Dunkl operators $D_i$. This is the {\it polynomial
representation} of \(\uCAc\), and is denoted by \(\Mbar_c\). 

The subring \(R \subset \Ru\) generated by
differences \(x_i-x_j\) is a polynomial ring on \(n-1\)
variables. It is preserved by the action of \(S_n\), and by differences 
$D_i - D_j$ of Dunkl
operators. It follows that \(\CAc\)
acts on \(R\).  We write \(M_c\) for this polynomial representation of
\(\CAc\).

$\Mbar_c$ carries a grading by the eigenvalues of $\overline{\mathbf{h}}$,
which differs by from the usual polynomial grading by a shift. 
If $\eta \in \uCAc$ and $p \in \Mbar_c$ are homogenous, then  
from $\overline{\mathbf{h}} \eta p = ([\overline{\mathbf{h}}, \eta] 
+ \eta \overline{\mathbf{h}}) p$ we see
$\deg(\eta p) = \deg(\eta) + \deg(p)$.  
Writing $p(x) \in \uCAc$ for the polynomial viewed as an element of $\uCAc$ 
rather than $\Mbar_c$, we have $p = p(x) \cdot 1$ 
hence $\deg(p) = \deg(p(x)) + \deg(1)$
where $\deg(p(x))$ is just its degree as a polynomial, and we compute
$\deg(1) = \overline{\mathbf{h}} \cdot 1 = n(1 + (1-n)c)/2$.  In the case of interest we 
will have $c = m/n$ and so $\overline{\mathbf{h}} \cdot 1= (n + m - mn)/2$. 
Similarly $M_c$ acquires a grading by the eigenvalues of $\mathbf{h}$, 
which differs from the polynomial grading by 
$\mathbf{h} \cdot 1 = (n(1 + (1-n)c) -1)/2$.  When $c=m/n$ this is $-(m-1)(n-1)/2$.  

\begin{lemma}
Let $V \subset R$ be a vector subspace annihilated by all Dunkl operators,
and consider the ideal $I = \CC[\hh] \cdot \CC[S_n] \cdot V \subset R$. 
Then \(\CAc \cdot I = I\).
\end{lemma}
\begin{proof}
Recall the multiplication map 
$\CC[\hh] \otimes \CC[S_n] \otimes \CC[\hh^*] \to \CAc$
is surjective.  Since $V$ is annihilated by the Dunkl operators,
hence by all nonconstant elements of $\CC[\hh^*]$, it follows that 
$\CAc \cdot V = \CC[\hh] \cdot \CC[S_n] \cdot V = I$,
and hence $\CAc \cdot I = \CAc \cdot \CAc \cdot V = I$. 
 \end{proof}

It follows that the quotient \(M_c/I\) defines a representation of
\(H_c\). 
\begin{example}
\label{example:n/2}
Suppose \(n=2\), and write $u = x_1 - x_2$ so that
\(M_c = \CC[x_1-x_2]= \CC[u]\). \(S_2\) has a
unique nontrival element \(s\), which sends \(u\) to \(-u\). If \(k\)
is even, we compute 
$$D_1(u^k) =  k u^{k-1} +c\frac{(-u)^k - u^k}{u} = \begin{cases}
  ku^{k-1} \quad & k\  \text{even} \\  (k- 2 c) u^{k-1} \quad & k \ \text{odd} \end{cases}
$$
Since \(D_2(u^k)=-D_1(u^k)\) (in general, it's easy to show that
 \(\sum D_i = 0\) when restricted to \(R\)), 
we see that the Dunkl operators have a
 nontrivial kernel if and only if \(c = k/2\), where \(k\) is an odd
 integer. In this case \(\CAns_{k/2}\) has a finite dimensional
 representation of the form \(\CC[u]/(u^k)\). 
\end{example}

An important result of Berest, Etingof, and Ginzburg says that  for
\(c>0\), all
finite dimensional irreducible representations of \(\CAc\) arise in
this way:

\begin{theorem}
\label{thm:DahaReps}
\cite{BEG3} \(\CAc\) has finite dimensional representations if and only
if \(c= m/n\), where \(m\) and \(n\) are integers and \((m,n)\)=1. In
this case, \(\CA\) has a unique (up to isomorphism) finite dimensional
irreducible representation \(\lmn\). For \(c=m/n>0\), \(L_{c} = M_{c}/I_{c}\), where
\(I_{c}\) is an ideal generated by homogenous polynomials of degree
\(m\).  
\end{theorem}

The fact that $I_c$ must be a graded ideal follows from considering
the action of $\mathbf{h}$.  In any case $\lmn$ carries a grading
by the eigenvalues of $\mathbf{h}$, which is the image of the grading 
on \(M_{m/n}\). According to Conjecture~\ref{conj:one}, this grading
should correspond to the \(q\)-grading on the 
HOMFLY homology \(\hH_{T(m,n)}\). 

It can be shown that set of polynomials of degree \(m\) in \(I_{m/n}\) is a copy of
the standard \(n-1\) dimensional representation of \(S_n\). Explicit
formulas for these polynomials are given in \cite{Dunkl2} and \cite{ech}.

\subsection{Projectors and the Spherical Subalgebra}
Let $V_\lambda$ be the irreducible representations
of the symmetric group $S_n$; recall they
are labelled by partitions $\lambda$ of $n$.  As for any finite group, 
any representation $W$ of $S_n$ decomposes canonically 
as a direct sum into {\em isotypic} components: 
$W = \bigoplus_\lambda W_\lambda$ where
$W_\lambda$ is non-canonically
isomorphic to a direct sum of $V_\lambda$.  The decomposition
is encoded by a decomposition of the 
identity $1 \in \CC[S_n]$ into idempotents
$1 = \sum_\lambda \eee_\lambda$ 
which act by projection $\eee_\lambda W = W_\lambda$.

The algebra \(\CAc\) contains the group ring \(\CC[S_n]\). 
Letting the group ring act on \(\CAc\) by right multiplication,
we can decompose
$$ \CAc = \bigoplus_{\lambda}  \CAns_{c,\lambda}$$
Similarly $\CC[S_n]$ acts on representations of \(\CAc\), so we can decompose 
$$\lmn = \bigoplus_\lambda \eee_\lambda \lmn $$
where, as a $S_n$ representation,
\( \eee_\lambda \lmn\) is a direct sum of irreducibles of type
\(\lambda\). 

Two projectors which are of particular importance are the ones
corresponding to the trivial and alternating representations:
\begin{equation*}
\eee  = \frac{1}{|S_n|} \sum_{s\in S_n} s \qquad \qquad
\eee_-  = \frac{1}{|S_n|} \sum_{s\in S_n} \text{sign}(s) s
\end{equation*} 
The algebra \(\eee \CAc \eee\) is the {\it spherical subalgebra} of
\(\CAc\) referred to in the introduction, and \(\eee \lmn\) is the
{\it spherical representation}. Similarly, 
\(\eee_- \CAc \eee_-\) and \(\eee_- \lmn\) are known as 
the {\it antispherical} Cherednik algebra and representation. 

\subsection{Symmetries.}

\subsubsection{Fourier transform}
Since $\lmn$ is finite dimensional, the $\mathfrak{sl}_2 \subset \CAc$ exponentiates
to a $\mathrm{SL}_2(\CC)$ action.  In particular the Fourier transform 
$\Phi$ will act on $\lmn$.  Evidently for any $h \in \CAc$ and $p \in \lmn$ 
we have $\Phi(h) \Phi(p) = \Phi(hp)$. 
In particular if $h = \mathbf{h}$ is the grading element, 
$\Phi(\mathbf{h}) = -\mathbf{h}$ ensures that $\Phi$ reverses 
the grading on $\lmn$.  Since $\Phi$ commutes with $S_n$, it preserves
all isotypic components $e_\lambda \lmn$.

\subsubsection{Spherical - Antispherical}

\begin{proposition}
\label{prop:twistsym}
(\cite[Prop. 4.6]{BEG3}) There is an isomorphism 
\( \eee \CAns_{c} \eee \cong  \eee_-\CAns_{c+1}\eee_-\) which respects the natural filtrations and gradings on both algebras. 
\end{proposition}

Let $W\in \CC[\hh]$ denote the Vandermonde determinant.  Multiplication by $W$
transforms symmetric polynomials into anti-symmetric ones.

\begin{proposition}(\cite{H};\cite[Prop. 4.11]{BEG2})
\label{Vandermonde}
Identify $L_{m/n}$ and $L_{(m-n)/n}$, as above, with the appropriate quotients of $\CC[\hh]$.
Multiplication by $W$ gives an isomorphism
$$m_W: \eee \lmnprev \xrightarrow{\sim} \eee_- \lmn.$$
\end{proposition}

This isomorphism corresponds to the ``top-row / bottom row'' symmetry of
the HOMFLY polynomial described in \cite{Kalman}. An analogous
symmetry for the HOMFLY homology of algebraic knots was conjectured in
\cite{G} and \cite{ORS}.

\subsubsection{Interchanging $n$ and $m$}
\label{sec:m and n symmetry}

Write \(\hh_n \) for the standard \(n-1\)
dimensional representation of \(S_n\). Then we have 

\begin{proposition}
\label{prop:mnsym} 
For each \(1\leq k \leq \max (m,n)\), there is a grading-preserving isomorphism
$$\Hom_{S_n}(\Lambda^{k} \hh_n ,\lmn) \cong \Hom_{S_m}(\Lambda^{k} \hh_m ,\lnm).$$
\end{proposition}
\noindent In particular, the proposition implies that \(\Hom_{S_n}(\Lambda^{k} \hh_n ,\lmn)=0\) for
\(m < k\le n\).
\begin{proof}
Consider the zero-dimensional scheme 
$\mathcal{M}_{m,n}$ defined by the coefficients in $z$-expansion of the 
following equation\footnote{
This scheme appears in \cite{FGvS} and is the moduli space of genus zero
maps to a rational curve containing the singularity $x^m = y^n$.  The length
of this scheme is shown there to be equal to the Euler number of 
the compactified Jacobian of the curve.  The comparison of the grading on
$\lmn$ and the perverse grading on the cohomology of this compactified
Jacobian (see Section \ref{sec:hilb} and \cite{MS, MY, OY}) gives a one-parameter
refinement of this result.  The additional conjectural comparison of
$\fF^{geom}$ and $\fF^{alg}$, when restricted to $\eee \lmn$, is a bigraded
refinement of this result, which in some form had been conjectured by L. G\"ottsche.
}
\begin{equation}
\label{eqUV}
(1+z^2e_2-z^3e_3+\ldots+(-1)^{n}z^{n}e_n)^{m}=(1+z^2\widetilde{e_2}-z^3\widetilde{e_3}+\ldots+(-1)^{m}z^n\widetilde{e_{m}})^{n}
\end{equation}
The main result of \cite{G1} is an explicit 
isomorphism of graded vector spaces

\begin{equation}
\label{hom as omega}
\Hom_{S_n}(\Lambda^{k}\hh_n,\lmn) \cong \Omega^{k}(\mathcal{M}_{m,n}).
\end{equation}

In particular, $e_{k}$ and $\widetilde{e}_{k}$ are identified with the elementary symmetric polynomials 
on $\hh_n$ and $\hh_m$ respectively.
The right hand side of (\ref{hom as omega}) is symmetric in $m$ and $n$, which gives the result.
\end{proof}

This symmetry corresponds to the identity of torus knots \(T(n,m) =
T(m,n)\).

\begin{remark}
In \cite[Prop. 9.5]{CEE} there is a (rather more sophisticated)
construction of an isomorphism $\eee\lmn = \eee\lnm$; the techniques
there may extend to give the above result. 
\end{remark}

\begin{example}
Let us prove Proposition \ref{prop:twistsym} using (\ref{hom as omega}).
One can rewrite (\ref{eqUV}) to get the isomorphism
$$\eee \lmn=\mathbb{C}[\mathcal{M}_{m,n}]\cong \mathbb{C}[e_2,\ldots,e_n]/(f_1,\ldots,f_{n-1}),$$
where 
$$f_i=\Coef_{m+i}[(1+e_2z^2+\ldots+(-1)^{n}e_nz^{n})^{\frac{m}{n}}].$$
On the other hand, we have
$$\eee_{-} \lmn=\Omega^{n-1}(\mathcal{M}_{m,n})\cong \mathbb{C}[e_2,\ldots,e_n]/\left(\frac{\partial f_i}{\partial e_j}\right).$$
The isomorphism $\eee \lmnprev = \eee_{-} \lmn$ follows from the identity
$$\frac{\partial f_i}{\partial e_j}=(-1)^{j}\frac{m}{n}\Coef_{m+i-j}[(1+e_2z^2+\ldots+(-1)^{n}e_nz^{n})^{\frac{m-n}{n}}].$$
\end{example}

\begin{lemma}
\label{H2}
Let $p_k=\sum x_i^{k}$ denote the power sum.
The following identity holds:
$$\overline{\mathbf{y}}^2(p_k)=(1+c)k(k-1)p_{k-2}-kc\sum_{i=0}^{k-2}p_{i}p_{k-2-i}.$$
\end{lemma}

\begin{proof}
We have $y_{a}p_{k}=kx_{a}^{k-1},$ so
$$\overline{\mathbf{y}}^2(p_k)=\sum_{a}y_{a}(kx_{a}^{k-1})=k(k-1)p_{k-2}-kc\sum_{a\neq b}\frac{x_{a}^{k-1}-x_{b}^{k-1}}{x_a-x_b}=$$ $$k(k-1)p_{k-2}-kc\left(\sum_{i=0}^{k-2}p_{i}p_{k-2-i}-(k-1)p_{k-2}\right).\qedhere
$$
\end{proof}

Suppose that we have two sets of variables $x_1, \ldots, x_n$ and $\widetilde{x}_1, \ldots, \widetilde{x}_m$
and let $p_k$ and $\widetilde{p}_{k}$ denote the power sums in $x_i$ and $\widetilde{x}_i$ respectively.
Suppose furthermore that these sets of variables are related by the equation (compare with (\ref{eqUV}))

\begin{equation}
\left[\prod_{i=1}^{n}(1-zx_i)\right]^{m}=\left[\prod_{i=1}^{m}(1-z\widetilde{x}_i)\right]^{n}
\end{equation}

By taking the logarithmic derivative in $z$ one immediately gets the identity
$$mp_k=n\widetilde{p}_{k}$$ 

\begin{definition}
Let us define a bidifferential operator on symmetric functions:
$$\langle f,g\rangle_{x}=\sum_{i} \frac{\partial f}{\partial x_i}\cdot \frac{\partial g}{\partial x_i}.$$
\end{definition}

\begin{lemma}
\label{pairing}
For symmetric $f$ and $g$ we have
$$\langle f,g\rangle_{x}=\frac{m}{n}\langle f,g \rangle_{\widetilde{x}}.$$
\end{lemma}

\begin{proof}
We have   
$$\langle p_{k},p_{l}\rangle_{x} =\sum_{i} klx_{i}^{k+l-2}=klp_{k+l-2},$$
so 
$$\langle \widetilde{p}_{k},\widetilde{p}_{l}\rangle_{x}=\frac{m^2}{n^2}\cdot klp_{k+l-2}=\frac{m}{n}\cdot kl\widetilde{p}_{k+l-2}=\frac{m}{n}\langle \widetilde{p}_{k},\widetilde{p}_{l}\rangle_{\widetilde{x}}.$$
Since $\langle \cdot,\cdot \rangle$ is a first order differential operator in both arguments, we conclude that
$$\langle f,g\rangle_{x}=\frac{m}{n}\langle f,g \rangle_{\widetilde{x}}.\qedhere$$
\end{proof}

\begin{theorem}  Consider
the Hamiltonians 
$H_2^{c}:=\overline{\mathbf{y}}^2$ for $c=m/n$ and $c=n/m$ 
as operators on $\eee L_{m/n} = \eee L_{n/m}$.   Then
$$H_{2}^{m/n}=\frac{m}{n}\cdot \widetilde{H}_{2}^{n/m}.$$
\end{theorem}

\begin{proof}
By Lemma \ref{H2} we have 
$$H_{2}^{m/n}(p_k)=\frac{m+n}{n}k(k-1)p_{k-2}-k\frac{m}{n}\sum_{i=0}^{k-2}p_{i}p_{k-2-i},$$
so
$$H_{2}^{m/n}(\widetilde{p}_k)=\frac{m}{n}H_{2}^{\frac{m}{n}}(p_k)=\frac{m}{n}\widetilde{H}_{2}^{n/m}
(\widetilde{p}_{k}).$$

The following identity is true for symmetric polynomials $f$ and $g$:
\begin{equation}
\label{product rule}
H_2(fg)=\sum_{i}D_{i}(D_i(f)g+fD_{i}(g))=H_2(f)g+fH_2(g)+2\langle f,g\rangle_{x}.
\end{equation} 
Now the theorem follows from (\ref{product rule}) and Lemma \ref{pairing}, since it was already checked for the generators $p_{k}$.
\end{proof}

\begin{corollary}(cf. \cite{CEE}, Proposition 8.7)
\label{simplified CEE}
The actions of the spherical subalgebras $\eee \Hh_{m/n} \eee$ and $\eee \Hh_{n/m} \eee$ on $L_{m/n}$ and $L_{n/m}$ are proportional.
\end{corollary} 

\begin{proof}
It is well known (e.g. \cite{BEG2}, Propositions 4.10 and 4.11) that the spherical subalgebra $\eee \Hh_c \eee$ is generated by $H_2^{c}$ and $p_k$.
\end{proof}

\section{Euler Characteristic}
\label{sec:Eulerchar}

\begin{definition}
Let $L$ be a representation of $S_n$. The Frobenius character of $L$ is the following symmetric function:
$$\ch(L)=\frac{1}{n!}\sum_{\sigma\in S_n}\Tr_{L}(\sigma)p_1^{k_1(\sigma)}\ldots p_r^{k_r(\sigma)},$$
where $p_i$ are power sums, and $k_i(\sigma)$ is the number of cycles of length $i$ in a permutation $\sigma$.
\end{definition}

It is well known that the Frobenius character of an irreducible representation of $S_n$ labelled by the Young diagram $\lambda$ is
given by the Schur polynomial $s_{\lambda}$. As always we denote by $\hh$ the 
$(n-1)$-dimensional reflection representation of $S_n$.  The following lemma is well known.

\begin{lemma}
For $\sigma\in S_n$ acting in the reflection representation $\hh$,
\begin{equation}
\label{char poly}
\det (1-q\sigma)=\frac{1}{1-q}\prod_{i}(1-q^{i})^{k_i(\sigma)}
\end{equation}
\end{lemma}

\begin{proof}
Indeed, one has to prove that in the defining $n$-dimensional representation
$$\det(1-q\sigma)=\prod_{i}(1-q^{i})^{k_i(\sigma)}.$$
This is easy to check by reduction of $\sigma$ to a single cycle.
\end{proof}

\begin{proposition}
For any representation $L$ of $S_n$, 
\begin{equation}
\label{multlambda}
\sum_{k=0}^{n-1}(-u)^{k}\dim \Hom_{S_n}(\Lambda^{k} \hh, L)=\frac{1}{1-u}\ch(L; p_i=1-u^i).
\end{equation}
\end{proposition}

\begin{proof}
By orthogonality of characters, one has
$$\sum_{k=0}^{n-1}(-u)^{k}\dim \Hom_{S_n}(\Lambda^{k} \hh, L)=\sum_{k=0}^{n-1}(-u)^{k}\langle\Lambda^{k} \hh, 
L\rangle=$$
$$\frac{1}{n!}\sum_{\sigma\in S_n}\sum_{k=0}^{n-1}(-u)^{k}\Tr_{L}(\sigma)\Tr_{\Lambda^{k} \hh}(\sigma).$$
Now we can use the identity (\ref{char poly}):
$$\sum_{k=0}^{n-1}(-u)^{k}\Tr_{\Lambda^{k} \hh}(\sigma)=\det {}_{\hh}(1-u\sigma)=\frac{1}{1-u}\prod_{i} (1-u^{i})^{k_i}.$$
\end{proof}

\begin{definition}
Let $\fmn(q;p_i)$ denote the graded Frobenius character of $\lmn$.
Let $P_{m,n}(a,q)$ denote the reduced HOMFLY polynomial for the the $(m,n)$ torus knot.
\end{definition}

\begin{proposition}
The generating function for $\fmn(q;p_i)$ for given $m$ is:
\begin{equation}
\label{t=1}
F_{m}(q;p_i):=\sum_{n=0}^{\infty}z^{n}\fmn(q;p_i)  = \frac{1}{[m]_{q}}\prod_{i=0}^{m-1}\prod_{j}\frac{1}{1-q^{i+\frac{1-m}{2}}zx_{j}},
\end{equation}
where $[m]_{q}=\frac{q^{m/2}-q^{-m/2}}{q^{1/2}-q^{-1/2}}$ is a {\em normalized} $q$-integer.
\end{proposition}

\begin{remark}
While we have been considering $\lmn$ only for coprime $m$ and $n$, we will see in the proof that the formula
for the Frobenius character $\fmn$ makes sense for any natural number $n$.
\end{remark}

\begin{proof}
Let $\delta_{m,n}=\frac{(m-1)(n-1)}{2}.$
The $q$-graded character of $\lmn$ was computed in \cite{BEG3} (Theorem 1.6):

$$\Tr_{\lmn}(\sigma\cdot q^{\mathbf{h}})=q^{-\delta_{m,n}}\frac{\det_{\hh} (1-q^{m}\sigma)}{\det_{\hh} (1-q\sigma)}\stackrel{(\ref{char poly})}{=}q^{-\delta_{m,n}}\frac{1-q}{1-q^{m}}\prod_{i}\left(\frac{1-q^{mi}}{1-q^i}\right)^{k_i(\sigma)}.$$

Therefore

$$\fmn(q;p_i)=\frac{1}{n!}\sum_{\sigma \in S_n}\frac{q^{\frac{n(1-m)}{2}}}{[m]_{q}}\prod_{i}\left(\frac{1-q^{mi}}{1-q^i}p_i\right)^{k_i(\sigma)}.$$

Now
$$F_{m}(q;p_i)=\frac{1}{[m]_{q}}\exp\left(\sum_{k=1}^{\infty} \frac{(1-q^{mk})p_{k}z^{k}q^{\frac{k(1-m)}{2}}}{(1-q^{k})k}\right)=$$ $$\frac{1}{[m]_{q}}\prod_{i=0}^{m-1}\exp \left(\sum_{k=1}^{\infty} \frac{p_k q^{ik}z^{k}q^{\frac{k(1-m)}{2}}}{k}\right)=
\frac{1}{[m]_{q}}\prod_{i=0}^{m-1}\prod_{j}\frac{1}{1-q^{i+\frac{1-m}{2}}zx_{j}}.\qedhere$$
\end{proof}

\begin{theorem}
The representation $\lmn$ has a following decomposition into irreducible representations:
\begin{equation}
\label{SW duality}
\lmn=\frac{1}{[m]_{q}}\bigoplus_{|\lambda|=n}s_{\lambda}(q^{\frac{1-m}{2}},q^{\frac{3-m}{2}},\ldots ,q^{\frac{m-1}{2}})V_{\lambda}
\end{equation}
where $V_{\lambda}$ is an irreducible representation of $S_n$ labelled by the Young diagram $\lambda$.
\end{theorem}

\begin{proof}

We have to prove the identity
$$\fmn(q;p_i)=\frac{1}{[m]_{q}}\sum_{|\lambda|=n}s_{\lambda}(x_i)s_{\lambda}(q^{\frac{1-m}{2}},q^{\frac{3-m}{2}},\ldots ,q^{\frac{m-1}{2}}).$$
 
It follows from (\ref{t=1}) and the Cauchy identity
$$\prod_{i,j}\frac{1}{1-x_iy_j}=\sum_{\lambda} s_{\lambda}(x_i)s_{\lambda}(y_j)$$
with 
$$y_{j}=\begin{cases} zq^{j+\frac{1-m}{2}}&,0\le j<m\\ 0&, j\ge m\end{cases}.\qedhere$$
\end{proof}

\begin{remark}
The polynomial $s_{\lambda}(q^{\frac{1-m}{2}},q^{\frac{3-m}{2}},\ldots ,q^{\frac{m-1}{2}})$ is the $q$-character of the representation of Lie algebra $\mathfrak{sl}(m)$ labelled by the
diagram $\lambda$. In particular, it vanishes if $\lambda$ has more than $m$ rows. It is useful to use the hook formula (e.g. \cite{aiston},\cite{Resh})
\begin{equation}
\label{hook formula}
s_{\lambda}(q^{\frac{1-m}{2}},\ldots ,q^{\frac{m-1}{2}})=\prod_{(i,j)\in\lambda}\frac{[m+i-j]_{q}}{[h_{ij}]_{q}},
\end{equation}
where $h_{ij}$ is the hook length for a box $(i,j)$ in $\lambda$.
\end{remark}

The identity (\ref{SW duality}) can be also deduced from the Propositions 2.5.2 and 2.5.3 of \cite{haiman}.

\begin{example}
Let $\lambda$ be a hook-shaped diagram with $k+1$ rows. The corresponding representation of $S_n$ is $\Lambda^{k}\hh$, and the corresponding 
representation of $\mathfrak{sl}(m)$ by hook formula (\ref{hook formula}) has a character 
$$s_{\lambda}(q^{\frac{1-m}{2}},\ldots ,q^{\frac{m-1}{2}})=\frac{[m+n-k-1]_{q}!}{ [n]_{q}\cdot [k]_{q}!\cdot [n-k-1]_{q}! [m-k-1]_{q}!}.$$
Therefore the multiplicity of $\Lambda^{k} \hh$ in $\lmn$ is
$$\frac{[m+n-k-1]_{q}!} {[m]_{q}\cdot [n]_{q}\cdot [k]_{q}!\cdot [n-k-1]_{q}! [m-k-1]_{q}!}.$$
This matches the formula for the coefficient of $(-1)^k a^{\mu+2k}$ for the HOMFLY polynomial obtained in \cite{G} (Theorem 3.1) and \cite{BEM} (Equation 3.46).
In particular, it is symmetric in $m$ and $n$.
\end{example}

We obtain the following result.

\begin{theorem}
\label{euler characteristic}
The $q$-graded character of $\Hmn$ coincides with the HOMFLY polynomial of the $(m,n)$-torus knot:
$$a^{(m-1)(n-1)} \sum_{i=0}^{n-1}(-a^2)^i \dim_{q} \Hom_{S_n}(\Lambda^{i}\hh,\lmn)=P_{m,n}(a,q).$$
\end{theorem}

\section{Filtrations on \(\lmn\)}
\label{sec:filt}

In this section, we construct two filtrations on \(\lmn\), which we call the {\it algebraic} and {\it inductive} filtrations. We conjecture that these two filtrations are identical.

\vspace{2mm}

We begin with some remarks regarding filtrations.  Let $V, V'$ be vector spaces supporting
increasing filtrations $\fF_i$ and $\fF'_i$.  Then $V \oplus V'$ is filtered by
$(\fF \oplus \fF')_i(V \oplus V') = \fF_i(V) \oplus \fF'_i(W)$, and 
$V \otimes V'$ carries the tensor product filtration
\[(\fF \otimes \fF')_k(V \otimes V') = \sum_{i+j = k} \fF_i V \otimes \fF'_j V'\]
A map $f: V \to V'$ is compatible with the filtration if $f(\fF_i V) \subset \fF'_i V'$. 
If $A$ is a filtered algebra, then a filtration on an $A$-module $V$ is compatible with
that on $A$ if the multiplication
map $A \otimes V \to V$ is compatible with the filtrations.  If $V, V'$ are respectively
right and left filtered $A$-modules, then $V \otimes_A V'$ carries a tensor product filtration. 
If $V$ is a graded vector space, we require each step $\fF_i V$ of a filtration
to be spanned by homogeneous elements; in this case $\mathrm{gr}^{\fF} V$ carries
a bigrading. 

Specifically, for a filtration $\fF$ on $\lmn$ to be compatible with the algebra 
filtration $\gG$ on $\CAc$, 
we must have $\sigma \fF_i \subset \fF_i$ for any $\sigma \in S_n$ and 
$x_\alpha \fF_i, y_\alpha \fF_i \subset \fF_{i+1}$.  For it to be compatible with
the grading, we must have $\mathbf{h} \fF_i = \fF_i$.  

Since $\CC[S_n] = \gG_0 \CAc$, such a filtration will satisfy $\fF_i = \bigoplus \eee_\lambda \fF_i$.
Note in particular $\eee_- \lmn$ will be a filtered module over $\eee_- \CAc \eee_-$.   For $c > 1$, the
natural map $\CAc \eee_- \otimes_{\eee_- \CAc \eee_-} \eee_- \lmn \to \lmn$ is known to be
\cite{GS, BE} an isomorphism.  The map is evidently filtered, i.e. if $\fF^{-}$ is the 
filtration on the left hand side then $\fF^{-, i} \subset \fF^i$.
We may ask this to be an equality.\footnote{We could, but will not, ask the same to hold for the restriction
to $\eee \lmn$.  Indeed, both these conditions cannot hold simultaneously.}

Proposition \ref{Vandermonde} asserts that, 
for $c > 0$, multiplication by the Vandermonde
determinant fixes
compatible isomorphisms $\eee \CAc \eee = \eee_- \mathbf{H}_{c+1} \eee_-$ and 
$\eee \lmn = \eee_- L_{(m+n)/n}$. 
We may ask that the filtrations on both sides agree, fixing the isomorphism defined in Proposition \ref{Vandermonde}.
In \cite{CEE}, it is shown that there is a canonical isomorphism 
$\eee \lmn \cong \eee \lnm$; that the images of the algebras $\eee \mathbf{H}_{m/n}$ and 
$\eee \mathbf{H}_{n/m}$ in the endomorphisms of this space agree as filtered algebras; and
that the $\mathfrak{sl}_2$ factors have the same image (see section \ref{sec:m and n symmetry} for more details).
Thus it is natural to ask that 
the two filtrations on this space agree as well.  From the point of view of our Conjecture 
\ref{conj:one} at least the dimensions of the associated bigraded spaces must be the same.

\subsection{The inductive filtration}  The above desiderata uniquely specify a filtration.  The idea of the following construction is due to Etingof and Ginzburg.  When \(m \equiv 1 \pmod n\), this construction appears explicitly
 in section 7 of \cite{BEG3} ({\it c.f.}
\cite{GS}).

\begin{theorem} \label{thm:ind}  
   Consider $\lmn$ for $m > n$ and $\eee \lmn$ for $m > 0$.  There is a unique system of filtrations
   on these modules compatible with the algebra filtrations on $\mathbf{H}_{m/n}$ such that: 
   \begin{enumerate}
   	\item For the one dimensional modules 
		$\eee L_{1/n}$, the filtration is $0 = \fF_{-1} \subset \fF_0 = \eee L_{1/n}$.
	\item The canonical isomorphisms $\eee \lmn = \eee \lnm$ and 
		$\eee \lmn = \eee_- L_{(m+n)/n}$ preserve filtrations.
	\item For $m > n$ the isomorphism 
		$\lmn = \Hh_{m/n}\eee_- \otimes_{\eee_- \Hh_{m/n}\eee_-} \eee_- \lmn$
		preserves filtrations.
   \end{enumerate}
   We denote the resulting filtration $\fF^{ind}$.  It is preserved by the $\mathfrak{sl}_2$ action. 
   
   If $\fF'$ is a family of filtrations satisfying (1) and (2) and compatible with the action of $\Hh_c$,
   then $\fF^{ind}_i \lmn \subset \fF'_i \lmn$ for all $m,n,i$. 
\end{theorem}
\begin{proof}
   Consider the partial order on non-integer positive fractions in lowest terms $n/m$ 
   generated by the relations $m/n \prec (m+n)/n$ and $m/n \succ n/m$ if $n < m$.
   The Euclidean algorithm implies this is a strict partial order with minimal elements
   $1/n$ to which every other pair is connected by a finite chain of the above relations. 
   We construct $\fF^{ind}$ by induction.  Its value on the minimal elements $1/n$ is prescribed
   by condition (1) above.  If $m < n$ then we need only provide a filtration on 
   $\eee \lmn = \eee \lnm$; the RHS is defined by induction.  Otherwise we have by 
   (2) and (3) the following isomorphisms of filtered modules.  As the RHS is defined by 
   induction, the LHS is determined as well.  
   \[L_{m/n} = \Hh_{m/n} \eee_-  \otimes_{\eee_- \Hh_{m/n} \eee_-} 
   	\eee_- L_{m/n} = \Hh_{m/n} \eee_-  \otimes_{\eee_- \Hh_{m/n}\eee_-} 
	\eee L_{(m-n)/n}\]
  The filtration
   thus constructed is preserved by the $\mathfrak{sl}_2$ action since the same holds for the filtrations
   on all objects used in its construction.

   A similar induction establishes the final statement.  Indeed, suppose 
   $\fF^{ind}_\bullet \lmn \subset \fF'_\bullet \lmn$.  Let $\fF''_\bullet L_{(n+m)/m}$
   be the filtration induced from $\fF'_\bullet \lmn$; evidently $\fF^{ind}_\bullet L_{(n+m)/m}
   \subset \fF''_\bullet L_{(n+m)/m}$.     
   The second equality in the equation is still an equality of
   filtrations by condition (2), and from the 
   first we see $\fF''_\bullet L_{(n+m)/m} \subset \fF'_\bullet L_{(n+m)/m}$
   since $\fF'$ was assumed to be compatible with the algebra structure.  
\end{proof}

\begin{definition}
We define $\fF^{ind}$ on $H_{m/n}=\Hom_{S_n}(\Lambda^{*}V,\lmn)$
in the following way:
\begin{enumerate}
\item For $m>n$ it is obtained by the restriction of $\fF^{ind}$ onto the corresponding isotypic components
\item For $m<n$ it is obtained by the isomorphism $H_{m/n}\simeq H_{n/m}$.
\end{enumerate}
\end{definition}

\begin{proposition}
	If $f \in \lmn$ is a homogeneous element of degree $d$, then $f \notin \fF^{ind}_{-d-1} \lmn$. 
\end{proposition}
\begin{proof}
	Follows from the inductive construction.
\end{proof}

\begin{proposition} \label{prop:indvand}
	Let $W$ be the Vandermonde determinant and $m/n > 1$ so that the image 
	of $W$ is nonzero in $\lmn$.  Let $h$ be an element of degree $-d$.  Then
	$h \in \CC[\hh].W$ if and only if $h \in \fF^{ind}_d \lmn \setminus \fF^{ind}_{d-1} \lmn$. 
\end{proposition}
\begin{proof}
	We first check the ``only if'' direction. 
	For degree reasons the given elements cannot lie in the smaller step of the filtration.
	It suffices to establish the result for $1$ and $W$, since it will follow for
	$\CC[\hh].W$.  By Proposition \ref{Vandermonde}, $W$ is the image of $1$ under the 
	map identifying $\eee \lmn = \eee_- L_{\frac{n+m}{n}}$, and the 
	identification $\eee \lmn = \eee \lnm$ takes $1$ to $1$.  To deduce the result 
	from the inductive construction
	it remains to note that for $c>1$ there is a homogeneous polynomial $p$ such that $p(y).W = 1$. 

	In the ``if'' direction, assume $h \in \fF^{ind}_d \lmn$.  By definition we may write
	 $h = q(y) p(x) v$ with $v \in \fF^{ind}_{d- \deg q - \deg p} \eee_- \lmn$
	 We have
	 $\deg v = \deg h + \deg q - \deg p = - d + \deg q  - \deg p$.  By the proposition
	 we must have $d - \deg q - \deg p \ge d - \deg q + \deg p$, so certainly $p$ is constant. 
	 It remains to show $v$ is proportional to $W$.   But we have an equality of filtered graded
	 spaces $\eee_- \lmn = \eee L_{(m-n)/n}$;
	 if $m - n < n$ we pass to $\eee L_{n/(m-n)}$.   Let us write $v'$ for the image of $v$ in this space;
	 and $W'$ for the Vandermonde in this space; then by induction we have
	 $v' = r(y) .W'$ for some $r(y)$.  As the only such elements are scalars, we see $v$ is proportional to $W$. 
\end{proof}

\subsection{The algebraic filtration}

Let $\aa \subset \CC[\hh^*] = M_c$ be the ideal generated by symmetric polynomials of positive degree; its
powers give a {\em decreasing} filtration on $M_c$ and hence on $L_c$.

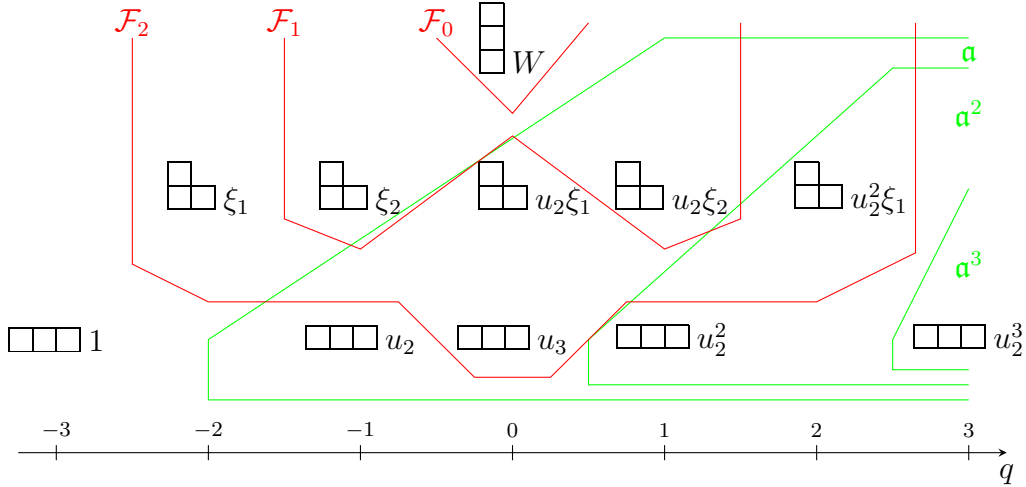
\begin{figure}[ht]
\begin{tikzpicture}
\draw (0,0) node {{\tiny $\yng(3)$} $1$};
\draw (4,0) node {{\tiny $\yng(3)$} $u_2$};
\draw (6, 0) node {{\tiny $\yng(3)$} $u_3$};
\draw (8.1,0) node {{\tiny $\yng(3)$} $u_2^2$};
\draw (12,0) node {{\tiny $\yng(3)$} $u_2^3$};
\draw (2,2) node {{\tiny $\yng(1,2)$} $\xi_1$};
\draw (4,2) node {{\tiny $\yng(1,2)$} $\xi_2$};
\draw (6.3,2) node {{\tiny $\yng(1,2)$} $u_2\xi_1$};
\draw (8.1,2) node {{\tiny $\yng(1,2)$} $u_2\xi_2$};
\draw (10.45,2) node {{\tiny $\yng(1,2)$} $u_2^2\xi_1$};
\draw (6,4) node {{\tiny $\yng(1,1,1)$} $W$};

\draw [green] (12,-0.8)--(2,-0.8)--(2,0)--(8,4)--(12,4);
\draw [green] (12,-0.6)--(7,-0.6)--(7,0)--(11,3.6)--(12,3.6);
\draw [green] (12,-0.4)--(11,-0.4)--(11,0)--(12,2);
\draw [green] (12,1) node {$\mathfrak{a}^3$};
\draw [green] (12,3) node {$\mathfrak{a}^2$};
\draw [green] (12,3.8) node {$\mathfrak{a}$};

\draw [red] (5,4)--(6,3)--(7,4.2);
\draw [red] (3,4)--(3,1.6)--(4,1.2)--(6,2.7)--(8,1.2)--(9,1.6)--(9,4.2);
\draw [red] (1,4)--(1,1)--(2,0.5)--(4.5,0.5)--(5.5,-0.5)--(6.5,-0.5)--(7.5,0.5)--(10,0.5)--(11.3,1.15)--(11.3,4.2);
\draw [red] (5,4.2) node {$\mathcal{F}_0$};
\draw [red] (3,4.2) node {$\mathcal{F}_1$};
\draw [red] (1,4.2) node {$\mathcal{F}_2$};

\draw [->,>=stealth] (-0.5,-1.5)--(12.5,-1.5) node [below] {$q$};
\draw (0,-1.6)--(0,-1.4);
\draw (2,-1.6)--(2,-1.4);
\draw (4,-1.6)--(4,-1.4);
\draw (6,-1.6)--(6,-1.4);
\draw (8,-1.6)--(8,-1.4);
\draw (10,-1.6)--(10,-1.4);
\draw (12,-1.6)--(12,-1.4);
\draw (0,-1.2) node {\tiny $-3$};
\draw (2,-1.2) node {\tiny $-2$};
\draw (4,-1.2) node {\tiny $-1$};
\draw (6,-1.2) node {\tiny $0$};
\draw (8,-1.2) node {\tiny $1$};
\draw (10,-1.2) node {\tiny $2$};
\draw (12,-1.2) node {\tiny $3$};
\end{tikzpicture}
\caption{Algebraic filtration (red) and filtration by powers of $\aa$ (green) on $L_{\frac{4}{3}}$.}
\end{figure}

The filtration $\fF^{ind}$ constructed above evidently differs from 
the filtration by powers of $\aa$: the former is increasing and preserved by
the $\mathfrak{sl}_2$ action; the latter is is decreasing and not preserved by
$\mathfrak{sl}_2$.  Nonetheless we will conjecture that they carry the same information.

Denote by $\lmn(i)$ the component of grading $i$.  We introduce a modified decreasing filtration 

\[\fF^i \lmn = \sum_j \left( \aa^j \cap \bigoplus_{k < 2j - i} \lmn(k) \right)\]

\begin{proposition} \label{prop:algdec}
	The (decreasing) filtration $\fF^i$ is 
	compatible with the filtration on $\CAc$ in the sense that 
	$S_n \cdot \fF^i \lmn = \fF^i \lmn$ and $(\hh \oplus \hh^*) \cdot \fF^i \subset \fF^{i-1}$. 
	Moreover, $\fF^i \lmn$ is preserved by $\mathfrak{sl}_2$. 
\end{proposition}
\begin{proof}
	The fact that $\hh^* \fF^i \subset \fF^{i-1}$ is obvious from the definition.  
	Recall from \cite{Du} that if $g$ is a symmetric
	polynomial, and $f$ is any polynomial, we have 
	\begin{equation}
	\label{leibnitz for dunkl}
	D_i(fg) = D_i(f) g+ D_i(g) f.
	\end{equation} 
	From this it follows that $D_i \aa^n \subset \aa^{n-1}$, and hence that 
	$\hh \fF^i \subset \fF^{i-1}$. 

	The filtration is evidently preserved by $\mathbf{x}^2$ and $\mathbf{h}$.  Let us
	see it is preserved by $\mathbf{y}^2$.  Since $\mathbf{y}^2$ decreases degree by $2$, we must 
	show that $(\sum D_r^2) \aa^k \subset \aa^{k-1}$.  The result is vacuous for $k = 0,1$,
	so we assume $k \ge 2$. 
	Consider $f = g \sigma_1 \ldots \sigma_k$, where $\sigma_i$ are symmetric polynomials of positive degree. By (\ref{leibnitz for dunkl}) we   
	have 
	\[(\sum_r D_r^2) (g \sigma_1 \ldots \sigma_k) = \sum_r D_r
	\left( D_r(g)\sigma_1 \ldots \sigma_k + g D_r(\sigma_1) \sigma_2 \ldots \sigma_k + \cdots + g \sigma_1\ldots 
	D_r(\sigma_k)\right)
	\]
	Evidently $D_r(D_r(g)\sigma_1 \ldots \sigma_k) \in \aa^{k-1}$.  The remaining
	terms are all of the same sort; we expand the second term to obtain
	\[
	\sum_r D_r(g D_r(\sigma_1)) \sigma_2 \ldots \sigma_k + 
	g D_r(\sigma_1) D_r(\sigma_2) \sigma_3 \ldots \sigma_k +
	g D_r(\sigma_1) \sigma_2 D_r(\sigma_3) \sigma_4 \ldots \sigma_k + \cdots 
	\]
	The first term inside the sum is evidently in $\aa^{k-1}$.  The remaining terms
	are the products of something in $\aa^{k-2}$ times something which, upon
	summing over $r$, becomes symmetric. 
\end{proof}

To obtain the {\em increasing} filtration conjecturally equivalent to $\fF^{ind}$, we 
take the orthogonal complement with respect to Dunkl's bilinear form \cite{D4}.  
Recall this is defined on the polynomial representation of $\Hh_{c}$ 
by the formula 
$$(f,g)_{c}=\left[\Phi(f)\cdot g\right]|_{x=0},$$
where $\Phi$, as above, denotes the Fourier transform on $\Hh_{c}$ (see section \ref{fourier for algebra}).
When $c$ is clear from context we omit it from the notation.  
One can show \cite{DO, DJO} that the form is symmetric and $(y f,g)=(f,\Phi(y) g)$.
Its kernel on $\mathbb{C}[\mathfrak{h}]$ coincides with the defining ideal of $L_{c}$, hence it 
defines a nondegenerate symmetric bilinear form on $\lmn$ \cite{DO, BEG3}.
It is preserved up to a scalar by the isomorphism $\eee L_c = \eee_- L_{c+1}$ 
due to the equality
$(Wf,Wg)_{c+1}=(W,W)_{c+1}(f,g)_{c}$, where $f,g$ are symmetric and $W$ is
the Vandermonde determinant \cite[Cor. 4.5]{DJO} (see also \cite{H}).  
It follows from Corollary \ref{simplified CEE} that the isomorphism $\eee L_{m/n}=\eee L_{n/m}$
is isometric with respect to the Dunkl forms. For any 
subspace $V \subset \lmn$ we write $V^\perp$ for its orthogonal complement
with respect to the Dunkl form.

\begin{definition}
   The algebraic filtration on $\lmn$ is $\fF^{alg}_i \lmn := (\fF^i \lmn)^\perp$.
\end{definition}

\begin{proposition} \label{prop:algcompat}
	The filtration $\fF^{alg}$ gives $\lmn$ the structure of a filtered 
	$\mathbf{H}_{m/n}$-module.  It is preserved by the 
	$\mathfrak{sl}_2$ action. 
\end{proposition} 
\begin{proof}
     Evidently $\fF^{alg}$ is preserved by the action of $S_n$.  
     Consider $f \in \fF_j^{alg}$.  Then for any $g \in \fF^{j+1}$ and $x \in \hh^*$, 
     we have $(x f, g) = (f, \Phi(x) g) = 0$ since $\Phi(x) g \in \fF^j$ by 
     Proposition \ref{prop:algdec}.  Thus $x \fF_j^{alg} \subset \fF_{j+1}^{alg}$.
     The same argument shows $\hh \fF_j^{alg} \subset \fF_{j+1}^{alg}$, and
     that the filtration is preserved by the $\mathfrak{sl}_2$ action. 
\end{proof}

\begin{theorem}
\label{thm:IndsubAleg}
	We have an inclusion of filtrations $\fF^{ind}_i \lmn \subset \fF^{alg}_i \lmn$. 
\end{theorem}
\begin{proof}
	Since we have shown that $\fF^{alg}$ is compatible with the filtration on 
	$\mathbf{H}_{\frac{m}{n}}$, it remains to verify the first two conditions
	of Theorem \ref{thm:ind}.  The first is a trivial consequence of the definition. 
	Since the isomorphism $\eee \lmn = \eee L_{\frac{m+n}{n}}$ preserves
	the grading and orthogonality under the Dunkl form, checking that
	it preserves $\fF^{alg}$ is equivalent to checking that it preserves the filtration
	by powers of $\aa$.  This follows from the fact that the map is explicitly given
	by multiplication by the Vandermonde determinant.  To check 
	that $\fF^{alg}$ is preserved by $\lmn = \lnm$, note that the equality is 
	according to \cite{G1} an isomorphism of rings; the maximal ideal in each case
	is the image of $\aa$.  It remains to observe that the isomorphism also preserves
	orthogonality under the Dunkl form. 
\end{proof}

\begin{proposition}
	If $f \in \lmn$ is a homogeneous element of degree $d$, then $f \notin \fF^{alg}_{-d-1} \lmn$. 
\end{proposition}
\begin{proof}
	Since the Dunkl pairing makes components of different degrees orthogonal, it is 
	equivalent to show that $f \in \fF^{-d-1}$.  From the definition this contains the space
	$\aa^0 \cap \lmn(d)$ of all elements of degree $d$. 
\end{proof}

\begin{corollary}
	Let $W$ be the Vandermonde determinant and $m/n > 1$ so that the image 
	of $W$ is nonzero in $\lmn$.  Let $h \in \CC[\hh].W$ be an element
	of degree $-d$.  Then
	$h \in \fF^{alg}_d \lmn \setminus \fF^{alg}_{d-1} \lmn$. 
\end{corollary}
\begin{proof}
	We have just seen the elements in question cannot lie in the smaller piece
	of the filtration.  From Proposition \ref{prop:indvand} 
	we have $h \in \fF^{ind}_d \lmn \subset \fF^{alg}_d \lmn$.
\end{proof}

\begin{corollary}
\label{kostant} 
     Let $m>n$.  Then on $\CC[\hh] \cdot W \subset \lmn$ we have $\fF^{alg} = \fF^{ind}$.  Equivalently 
	$\CC[\hh] \cdot W = \aa^\perp$.
\end{corollary} 
\begin{proof}
	The first statement follows by comparing Proposition \ref{prop:indvand} to the previous
	corollary.  The second follows by unwinding the definitions. 
\end{proof}

\begin{remark}
	Since both filtrations are preserved by $\mathfrak{sl}_2$ they must agree on 
	$\CC[\mathbf{x}^2] \cdot \CC[\hh] \cdot W$. 
\end{remark}

Motivated by Corollary \ref{kostant} and numerical evidence, we formulate the following

\begin{conjecture} 
For $c>1$ we have the identity of filtrations
$\fF^{alg} = \fF^{ind}$.  \end{conjecture}

\section{Examples}
\label{sec:Examples}

In this section we calculate the groups \(\Hmn\) in a few examples, and see that they agree either with the known behavior of the HOMFLY-PT homology, or the conjectured behavior as described in \cite{DGR}. The examples we consider are the \((2,n)\) and \((3,n)\) torus knots, and the \((n,m)\) torus knot in the limit as \(m\to \infty\). We also discuss a variant of conjecture~\ref{conj:one} which describes the unreduced homology. 

\subsection{The \((2,n)\) torus knot} As described in example~\ref{example:n/2}, we have \(L_{n/2} \cong \CC[u]/(u^n)\). Then \(L_{n/2}=\eee L_{n/2}\oplus \eee_- L_{n/2}\), with \(\eee L_{n/2}\) and \( \eee_- L_{n/2}\) are spanned by even and odd powers of \(u\) respectively. 
\begin{proposition}
\(H_{n/2}\cong L_{n/2}\). The \(q\) grading on \(L_{n/2}\) is given by  \(q(u^k) = 2k- (n-1)\).
 Elements of \(\eee L_{n/2}\) have \(a\)-grading \(n-1\) and filtration grading \((n-1)/2\), while elements of 
\(\eee_- L_{n/2}\) have \(a\)-grading \(n+1\) and filtration grading \((n-3)/2\). 
\end{proposition}
\begin{proof}
\(S_2\) has only two representations: the trivial representation \(\Lambda^0\hh\)  and the alternating representation \(\Lambda^1\hh\). Both are one-dimensional, which implies \(\Hom_{S_n}(\Lambda^*\hh, L_{n/2}) \cong L_{n/2}\). 
The action of the nontrivial element of $S_2$ takes $u \to -u$, so the even powers of $u$ are in $\eee L_{n/2}$ and
the odd powers in $\eee_- L_{n/2}$.   By definition \(\eee L_{n/2}\) has \(a\)-grading \(0 + \mu(K) = n-1\), and
 \(\eee_- L_{n/2}\) has \(a\)-grading \(2+ \mu(K) = n+1\).  The normalized \(q\)-grading is 
 \(q(u^k) = 2k- (n-1)\). 
 These gradings are illustrated in Figure~\ref{fig:T2n} and give the HOMFLY polynomial of \(T(2,n)\). 
 
 \begin{figure}
 \label{fig:T2n}
\begin{tikzpicture}
\draw (0,0) node {$\bullet$};
\draw (1,1) node {$\bullet$};
\draw (2,0) node {$\bullet$};
\draw (3,1) node {$\bullet$};
\draw (4,0) node {$\bullet$};
\draw (5,0.5) node {$\dots$};
\draw (6,0) node {$\bullet$};
\draw (7,1) node {$\bullet$};
\draw (8,0) node {$\bullet$};
\draw (9,1) node {$\bullet$};
\draw (10,0) node {$\bullet$};
\draw [blue] [->,>=stealth] (0.9,0.9)--(0.1,0.1);
\draw [blue] [->,>=stealth] (2.9,0.9)--(2.1,0.1);
\draw [blue] [->,>=stealth] (6.9,0.9)--(6.1,0.1);
\draw [blue] [->,>=stealth] (8.9,0.9)--(8.1,0.1);
\draw [red] [->,>=stealth] (1.1,0.9)--(1.9,0.1);
\draw [red] [->,>=stealth] (3.1,0.9)--(3.9,0.1);
\draw [red] [->,>=stealth] (7.1,0.9)--(7.9,0.1);
\draw [red] [->,>=stealth] (9.1,0.9)--(9.9,0.1);
\draw (0,-0.3) node {\tiny $-(n-1)$};
\draw (2,-0.3) node {\tiny $-(n-5)$};
\draw (4,-0.3) node {\tiny $-(n-9)$};
\draw (6,-0.3) node {\tiny $n-9$};
\draw (8,-0.3) node {\tiny $n-5$};
\draw (10,-0.3) node {\tiny $n-1$};
\draw (-0.5,0) node {\tiny $n-1$};
\draw (-0.5,1) node {\tiny $n+1$};
\end{tikzpicture} 
 \caption {The generators of \(L_{n/2}\) arranged according to \(a\) (vertical) and \(q\) (horizontal)
 gradings. The action of multiplication by \(u\) and the Dunkl operator on the top row are indicated by red and blue arrows, respectively. }
 \end{figure}
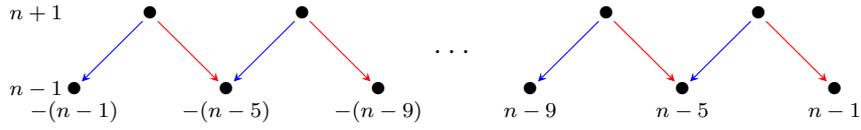

We compute the filtration $\fF^{ind}L_{n/2}$, starting with the one-dimensional representation \(L_{1/2}\), which has filtration \(0\). From the figure, it is clear that any generator in the bottom row (spherical DAHA) can be reached from an element of the top row (antispherical DAHA)  either by multiplying by \(u\) or by applying the Dunkl operator. The filtration grading of elements of the top row of \(L_{n/2}\) is the same as the filtration grading of elements in the bottom row of \(L_{(n-2)/2}\), so by induction we conclude that the filtration grading of generators in the top row of \(L_{n/2}\) is \((n-3)/2\), and that all generators in the bottom row have filtration grading \((n-1)/2\). This is equivalent to all generators having \(\delta\)-grading \(n-1 = - \sigma(T(2,n))\). Finally, it is easy to see that \(\fF^{ind}=\fF^{alg}\) in this case. 
 \end{proof}
 
 Observe that the action of \(u\) and the Dunkl operators on the top row matches the action of the differentials \(d_1\) and \(d_{-1}\) on HOMFLY-PT homology, as described in \cite{DGR}. This is not a coincidence; in section~\ref{sec:Diffs} we use the action of the DAHA to construct differentials on \(\Hmn\).

 \subsection{The \((3,n)\) torus knot}

In the description of $H_{n/3}$ we will use the identification (\ref{hom as omega}): 
$\Hom_{S_n}(\Lambda^* \hh, \lmn)\cong \Omega^{\bullet}(\mathcal{M}_{3,n})$, where 
$$\mathcal{M}_{3,n}=\Spec \mathbb{C}[u_2,u_3]/(f_{n+1},f_{n+2});\ f_{i}=\Coef_{i}[(1+u_2z^2+u_3z^3)^{\frac{n}{3}}].$$

The ring of functions $\mathbb{C}[\mathcal{M}_{3,n}]$ was computed in \cite{GM1}.
\begin{lemma}[\cite{GM1}]
\label{3n basis}
The ring $\mathbb{C}[\mathcal{M}_{3,n}]$ has a monomial basis consisting of monomials
\begin{equation}
\label{mbasis}
u_{2}^{a}u_3^{b}, a+3b\le n-1.
\end{equation}
\end{lemma}

\begin{proof}
Let us consider the case $n=3k+1$, the case $n=3k+2$ is analogous.
The defining equations $p_{3k+2}$ and $p_{3k+3}$ has degrees $3k+2$ and $3k+3$ respectively.
One can check that $p_{3k+3}$ has non-zero coefficient at $u_3^{k+1}$ and $p_{3k+2}$ has non-zero coefficient at
$u_3^{k}u_2$. The syzygy between leading monomials shows that the leading monomial in
$$p_{3k+5}:=u_2p_{3k+3}-\lambda u_3p_{3k+2}=u_2(c_1u_3^{k+1}+c_2u_3^{k-1}u_2^3+\ldots)-\frac{c_1}{ d_1}u_3(d_1u_3^{k}u_2+d_2u_3^{k-2}u_2^4+\ldots)$$
is $u_3^{k-1}u_2^{4}$. The syzygy between leading monomials in $p_{3k+2}$ and $p_{3k+5}$ has a form
$$p_{3k+8}:=u_3p_{3k+5}-\lambda u_2^{3}p_{3k+2},$$
and its leading monomial is $u_3^{k-2}u_2^{7}$ etc.

Using this process, we can eliminate the monomials $$u_3^{k+1}, u_3^{k}u_2, u_3^{k-1}u_2^{4}, u_3^{k-2}u_2^7,\ldots u_2^{3k+1}.$$
In the quotient we get the monomials (\ref{mbasis}). 
\end{proof}

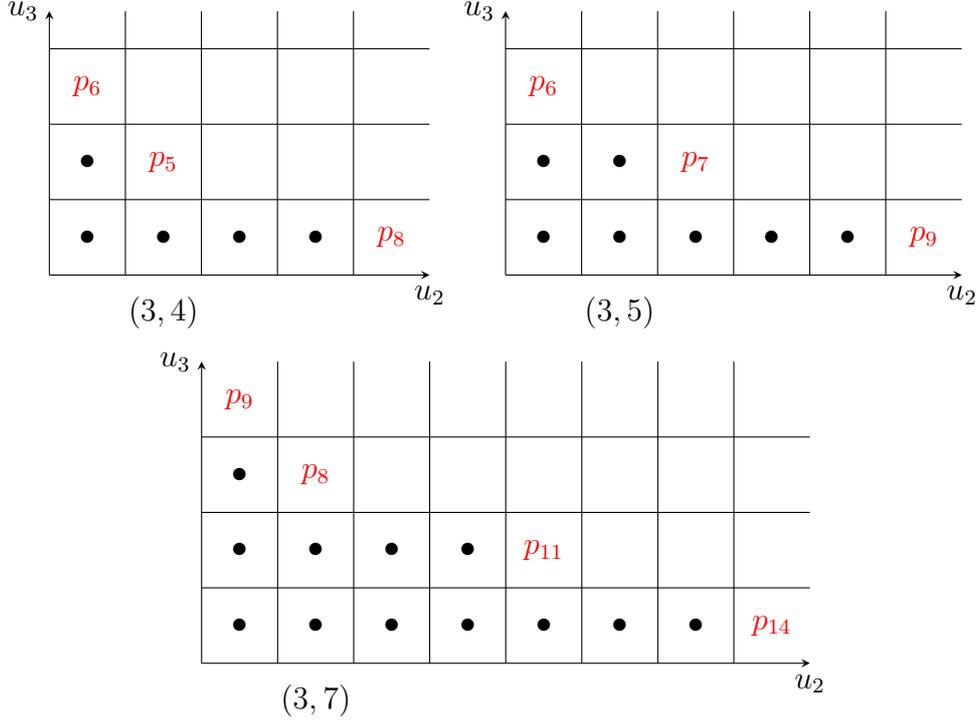
\begin{figure}
\begin{tikzpicture}
\draw[->,>=stealth] (0,0)--(5,0) node[below] {$u_2$};
\draw[->,>=stealth] (0,0)--(0,3.5) node[left] {$u_3$};
\draw (0,1)--(5,1);
\draw (0,2)--(5,2);
\draw (0,3)--(5,3);
\draw (1,0)--(1,3.5);
\draw (2,0)--(2,3.5);
\draw (3,0)--(3,3.5);
\draw (4,0)--(4,3.5);
\draw (0.5,0.5) node {$\bullet$};
\draw (0.5,1.5) node {$\bullet$};
\draw (1.5,0.5) node {$\bullet$};
\draw (2.5,0.5) node {$\bullet$};
\draw (3.5,0.5) node {$\bullet$};
\draw [red] (0.5,2.5) node {$p_{6}$};
\draw [red] (1.5,1.5) node {$p_{5}$};
\draw [red] (4.5,0.5) node {$p_{8}$};
\draw (1.5,-0.5) node {$(3,4)$};

\draw[->,>=stealth] (6,0)--(12,0) node[below] {$u_2$};
\draw[->,>=stealth] (6,0)--(6,3.5) node[left] {$u_3$};
\draw (6,1)--(12,1);
\draw (6,2)--(12,2);
\draw (6,3)--(12,3);
\draw (7,0)--(7,3.5);
\draw (8,0)--(8,3.5);
\draw (9,0)--(9,3.5);
\draw (10,0)--(10,3.5);
\draw (11,0)--(11,3.5);
\draw (6.5,0.5) node {$\bullet$};
\draw (6.5,1.5) node {$\bullet$};
\draw (7.5,0.5) node {$\bullet$};
\draw (8.5,0.5) node {$\bullet$};
\draw (9.5,0.5) node {$\bullet$};
\draw (10.5,0.5) node {$\bullet$};
\draw (7.5,1.5) node {$\bullet$};
\draw [red] (6.5,2.5) node {$p_{6}$};
\draw [red] (8.5,1.5) node {$p_{7}$};
\draw [red] (11.5,0.5) node {$p_{9}$};
\draw (7.5,-0.5) node {$(3,5)$};
\end{tikzpicture}

\begin{tikzpicture}
\draw[->,>=stealth] (0,0)--(8,0) node[below] {$u_2$};
\draw[->,>=stealth] (0,0)--(0,4) node[left] {$u_3$};
\draw (0,1)--(8,1);
\draw (0,2)--(8,2);
\draw (0,3)--(8,3);
\draw (1,0)--(1,4);
\draw (2,0)--(2,4);
\draw (3,0)--(3,4);
\draw (4,0)--(4,4);
\draw (5,0)--(5,4);
\draw (6,0)--(6,4);
\draw (7,0)--(7,4);
\draw (0.5,0.5) node {$\bullet$};
\draw (1.5,0.5) node {$\bullet$};
\draw (2.5,0.5) node {$\bullet$};
\draw (3.5,0.5) node {$\bullet$};
\draw (4.5,0.5) node {$\bullet$};
\draw (5.5,0.5) node {$\bullet$};
\draw (6.5,0.5) node {$\bullet$};
\draw (0.5,1.5) node {$\bullet$};
\draw (1.5,1.5) node {$\bullet$};
\draw (2.5,1.5) node {$\bullet$};
\draw (3.5,1.5) node {$\bullet$};
\draw (0.5,2.5) node {$\bullet$};
\draw [red] (0.5,3.5) node {$p_{9}$};
\draw [red] (1.5,2.5) node {$p_{8}$};
\draw [red] (4.5,1.5) node {$p_{11}$};
\draw [red] (7.5,0.5) node {$p_{14}$};
\draw (1.5,-0.5) node {$(3,7)$};
\end{tikzpicture}
\caption{Monomial bases in $\mathbb{C}[\mathcal{M}_{3,n}]$ for $n=4,5$ and $7$}
\end{figure}

Proposition \ref{prop:twistsym} yields $\Omega^2(\mathcal{M}_{3,n}) \cong 
\mathbb{C}[\mathcal{M}_{3,n-3}]$.
Finally, the defining equations of $\Omega^{1}(\mathcal{M}_{3,n})$ can be obtained by applying the de Rham differential
to the equations from the proof of Lemma \ref{3n basis}. We arrive at the following result.

\begin{theorem}
The space $H_{n/3}$ has the following monomial basis:
\begin{itemize}
\item[$a=0$:] $u_{2}^{a}u_3^{b}, a+3b\le n-1.$
\item[$a=1$:] $u_{2}^{a}u_3^{b}du_2, a+3b\le n-1, a<n-1; u_{2}^{a}u_3^{b}du_3, a+3b\le n-4.$
\item[$a=2$:] $u_{2}^{a}u_3^{b}du_1\wedge du_2, a+3b\le n-4.$
\end{itemize}
\end{theorem}

To compute the dimensions of the associated graded vector space determined by the filtration \(\fF^{alg} \lmn\), we work with the dual filtration \(\fF^{i}\) given by powers of \(\aa\). One can check that the filtration level of any of the monomial basis elements above is given by 
\(2 \deg_u - q/2\), and that corresponding Hilbert polynomial 
 agrees with the Poincar\'e polynomial of $\hH_{T(3,n)}$ predicted by 
 \cite{DGR}, \cite{G} and \cite{ORS}.

  \subsection{Wedge products in \(\Hmn\)}
 
 Before going on, we recall some basic facts about  the representation theory of \(S_n\). The permutation representation \(\hhu\) of \(S_n\) has a basis  \(x_1,\ldots x_n\). Suppose \(V\) is a representation of \(S_n\). Given \(\ophi \in \Hom_{S_n}(\hhu,V)\), we write \(\ophi(x_i) = \ophi_i\). Viewing
 \(\Hom_{S_n}(\hhu, V)\) as \((\hhu^* \otimes V)^{S_n}\),  we can write  \(\ophi = \sum_i \ophi_i x_i^*\). More generally, we can view
 $$ \Hom_{S_n}(\Lambda^* \hhu, V) \cong (\Lambda^*\hhu^* \otimes V)^{S_n}. $$
 Thus given \(\ophi  \in  \Hom_{S_n}(\Lambda^i \hhu, V)\) and \(\overline{\psi}  \in  \Hom_{S_n}(\Lambda^j \hhu, W)\), we can form $$ \ophi \wedge \overline{\psi} \in   \Hom_{S_n}(\Lambda^{i+j} \hhu, V\otimes W).$$ Similarly, if 
 \(\alpha \in \Hom_{S_n}(\hhu,V)\), we can use the natural isomorphism \(\hhu \cong \hhu^*\) given by the Killing form to form \(\alpha \neg \overline{\psi} \in \Hom_{S_n}(\Lambda^{j-1} \hhu,V\otimes W)\).
 
 Next, we consider the reflection representation \(\hh\subset \hhu \) of \(S_n\). There is a natural projection 
 \(\pi_\hh:\hhu \to \hh\).
 Any \(\varphi \in \Hom_{S_n}(\hh,V)\) induces \(\overline {\varphi} = \varphi \circ \pi_\hh \in \Hom_{S_n}(\hhu,V)\). 
Conversely, it is easy to see that \(\overline{\varphi} \in \Hom_{S_n}(\hhu,V)\) is induced by some \(\varphi\) if and only if \(\sum_i \varphi_i = 0\). More generally, \(\varphi \in \Hom_{S_n}(\Lambda^k \hh,V)\) induces \(\overline{\varphi} \in \Hom_{S_n}(\Lambda^k \hhu,V)\) via the natural projection \(\pi_k:\Lambda^* \hhu \to \Lambda^* \hh\), and \(\overline{\varphi} \in \Hom_{S_n}(\Lambda^k \hhu,V)\) is induced by some \(\varphi\) if and only if \( \mathbf{1}_\hh \neg \overline{\varphi} = 0\), where
\(\mathbf{1}_\hh \in \Hom_{S_n}(\hhu,\mathbf{1})\) denotes the homomorphism to the trivial representation.

 \subsection{Stable homology of torus knots}
The {\it stable homology} of the \(n\)-stranded torus knot is defined to be
$$\Hbar_{T(n,\infty)} = \lim_{m\to \infty} (a^{-1}q)^{(n-1)(m-1)} \Hbar_{T(n,m)}.$$
The multiplications by $a$ and $q$ denote  degree shifts which act by the given multiplication
when we take  the Poincar\'e polynomial.
This limit exists by a theorem of Stosic \cite{Stosic}. In \cite{DGR}, it was conjectured that
$$\Hbar_{T(n,\infty)} \cong \Lambda^*(\xi_1,\ldots,\xi_{n-1})\otimes \CC[u_1,\ldots,u_{n-1}],$$
where \(\xi_i\) has \((a,q,t)\) grading \((2,2i,2i+1)\), and \(u_i\) has \((a,q,t) \) grading \((0,2i+2,2i)\).

Let us compute the analogous limit for \(H_{m/n}\). 
\begin{lemma} 
\label{lem:stableHOMFLY}
\(\displaystyle \lim_{m\to \infty} (a^{-1}q)^{(n-1)(m-1)}H_{m/n}  \cong 
 \Hom_{S_n}(\Lambda^*\hh,\CC[\hh^*]).\) 
\end{lemma}
\begin{proof}
\(\lmn \cong M_{m/n}/I\) where \(I\) is generated by polynomials of degree \( m\). It follows that in \(q\)--degrees \(< 2m\), \(H_{m/n} \cong  \Hom_{S_n}(\Lambda^*\hh,M_{m/n}) =  \Hom_{S_n}(\Lambda^*\hh,\CC[\hh^*])\). 
\end{proof}

To compute the Poincar\'e series of the limit, we 
recall some well-known facts about  the action of \(S_n\) 
on the polynomial representation \(M_{m/n}=\CC[\hh^*]\). First, the ring of invariants under 
this action is a polynomial ring in \(n-1\) variables:
 $$\CC[\hh^*]^{S_n} \cong \CC[u_1,\ldots, u_{n-1}]$$
 where \(u_i\) has degree \(i+1\). The ring of coinvariants \(\CC[\hh^*]_{S_n}\)
 is defined to be the quotient 
 \(\CC[\hh^*]/I\) where \(I\) is the ideal generated by all symmetric
 polynomials of positive degree. Then
\begin{equation} \label{eq:Kostant} \CC[\hh^*] \cong \CC[\hh^*]_{S_n} \otimes \CC[\hh^*]^{S_n} \end{equation}
 as representations of \(S_n\). Moreover, \(\CC[\hh^*]_{S_n}\) is isomorphic to
 the regular representation of \(S_n\). In particular, it contains
 \(\binom{n-1}{k}\) copies of \(\Lambda^k\hh\). 

More explicitly, if for  each \(u_i\) we let
\(du_i \in \Hom_{S_n}(\hh,\CC[\hh^*])\) be given by $$du_i = \sum_j \frac{\partial{u_i}}{\partial{x_j} }x_j^*,$$ then \( \Hom_{S_n}(\hh,\CC[\hh^*]_{S_n})\) is generated by the \(du_i\)'s. In addition, we can compose \(du_{i_1} \wedge \dots  \wedge du_{i_j} \), which is {\it a priori} an element of \(\Hom_{S_n}(\Lambda^* \hh,\CC[\hh^*]^{\otimes j})\) with the multiplication map 
\(\CC[\hh^*]^{\otimes j} \to \CC[\hh^*]\), to get an element of 
\(\Hom_{S_n}(\Lambda^* \hh,\CC[\hh^*])\), which we will again denote by \(du_{i_1} \wedge \dots  \wedge du_{i_j} \). \(\Hom_{S_n}(\Lambda^* \hh,\CC[\hh^*]_{S_n})\) is generated by these wedge products.

The limits of the filtrations \(\fF^{ind}\) and \(\fF^{alg}\) will be considered more carefully in section~\ref{sec:Infinite limit}. For now, we just describe the dual filtration 
\(\fF^i\) on \(\CC[\hh^*]\) given by powers of the ideal \(\aa=(u_1,\ldots,u_{n-1})\). Let \( \aa^{S_n} = (u_1,\dots, u_{n-1}) \subset \CC[\hh^*]^{S_n}\). Then  under the isomorphism  of equation~\eqref{eq:Kostant}, \(\aa^k\) corresponds to \(\CC[\hh^*]_{S_n} \otimes ( \aa^{S_n})^k\). 
Thus the filtration \(\fF^i\) is induced by a multiplicative grading \(f^*\) on \( \Hom_{S_n}(\hh,\CC[\hh^*])\), with respect to which \(u_i\) is homogenous with grading \(f^*(u_i) = (1+i)-2\) and \(du_i\) is homogenous with \(f^*(du_i) = 0\). Consulting the discussion of gradings in the introduction and recalling that \(f^* = - f\) (since \(\fF^i\) is dual to the filtration we want), we see that \(f^*\) is related to the \(t\)-grading by $$ t = f^* + (a+q)/2.$$

Putting these results together and writing \(\xi_i = du_i\), we have 
\begin{proposition}
$\displaystyle \Hom_{S_n}(\Lambda^* \hh,\gr \CC[\hh^*]) \cong \Lambda^*(\xi_1,\ldots,\xi_{n-1})\otimes \CC[u_1,\ldots,u_{n-1}]$ where \(\xi_i\) has \((a,q,t)\) bigrading \((2,2i,2i+1)\) and \(u_i\) has \((a,q,t)\) bigrading \((0,2i+2,2i)\). 
\end{proposition}

\subsection{Unreduced Homology}
\label{subsec:unreduced}

There is an unreduced version \(\overline{\hH}_K\) of the HOMFLY homology, whose graded Euler characteristic is the unnormalized HOMFLY
polynomial:
$$\chi(\overline{\hH}_K) = \frac{a-a^{-1}}{q-q^{-1} }\Pp_K.$$
The relation between \(\hH\) and \(\overline{\hH}\) which categorifies this formula is rather uninteresting:
\begin{equation*}
\overline{\hH}_K \cong H^*(\CC^*) \otimes \CC[X] \otimes \hH_K,
\end{equation*}
but there are spectral sequences relating \(\overline{\hH}_K\) to the unreduced \(\mathfrak{sl}_N\) homology \(\overline{\hH}_{N,K}\) analogous to those of \cite{R} which are not easily derived from the reduced versions. With this in mind, we formulate a version of Conjecture~\ref{conj:one} which describes the unreduced homology. 

Recall from section~\ref{subsec:reducedH} that \(\overline{\Hh}_c \cong \dD \otimes  \Hh_c \), where \(\dD\) is the algebra of differential operators in one variable.  The decomposition \(\hhu \cong \hh \oplus \mathbf{1}\) induces an isomorphism
\(\CC[\hhu^*] \cong \CC[u_1] \otimes \CC[\hh^*] \), where \(u_1 = x_1+\ldots+x_n\), and the action of 
\(\overline{\Hh}_c\)  on \(\CC[\hhu^*]\) respects this decomposition. It follows that \(\overline{\Hh}_{m/n}\) acts on \(\overline{L}_{m/n} := \CC[u_1] \otimes \lmn\). We define a \(q\)-grading on \(\CC[u_1]\) by the requirement that \(q(u_1^k) = -1+2k\), and an increasing filtration \(\fF'_i = \langle 1,u_1, \ldots, u_1^i \rangle \). Taking tensor products with the grading and filtrations on \(\lmn\) gives a \(q\)-grading on \(\overline{L}_{m/n}\), as well as filtrations \(\overline{\fF}^{ind}\) and \(\overline{\fF}^{alg}\).
 Finally, the decomposition \(\hhu \cong \hh \oplus \mathbf{1}\)  induces an isomorphism
 \(\Lambda^* \hhu \cong \Lambda(\xi_0) \otimes \Lambda^*\hh\). Combining these observations, we see that
$$ \Hom_{S_n}(\Lambda^*\hhu,\gr^{\overline{\fF}} \overline{L}_{m/n}) \cong \Lambda^*(\xi_0) \otimes \CC[u_1] \otimes \Hmn^{\fF}.$$
 Thus Conjecture~\ref{conj:one} is equivalent to 

\begin{conjecture} \( \displaystyle \overline{\hH}_{T(n,m)} \cong \Hom_{S_n}(\Lambda^*\hhu,\gr^{\overline{\fF}} \overline{L}_{m/n})\). 
\end{conjecture}

There is a natural injection $$\iota:\Hom_{S_n}(\Lambda^*\hh,\gr^{{\fF}} {L}_{m/n}) \to  \Hom_{S_n}(\Lambda^*\hhu,\gr^{\overline{\fF}} \overline{L}_{m/n})$$ given by 
$$\iota(\varphi)(x_i) = \iota_{\CC[\hh^*]}(\varphi(\pi_\hh(x_i)))$$
where \(\iota_{\CC[\hh^*]}: \CC[\hh^*] \to \CC[\hhu]\) is the inclusion, and 
\(\pi_{\hh}:\hhu \to \hh\) is the projection. Similarly, there is a projection
$$ \pi: \Hom_{S_n}(\Lambda^*\hhu,\gr^{\overline{\fF}} \overline{L}_{m/n}) \to \Hom_{S_n}(\Lambda^*\hh,\gr^{{\fF}} {L}_{m/n})$$ given by 
$$\pi(\overline{\varphi})(\widetilde{x}_i) = \pi_{\CC[\hh^*]}(\overline{\varphi}(\iota_\hh(\widetilde{x}_i))).$$

\section{Character Formulas}
\label{sec:Characters}

In this section we use the following notations:  $a(c),l(c),a'(c)$ and $l'(c)$ denote respectively arm, leg, co-arm and co-leg for a box $c$ in a diagram $\mu$, $\mu^{t}$ denotes the conjugate diagram and 
$$n(\mu)=\sum_{c\in \mu} l(c)=\sum_{c\in \mu} l'(c)
=\sum _{i}(i-1)\mu_i.$$
As above, $s_{\lambda}$ denotes the Schur polynomial corresponding to a Young diagram $\lambda$.

\subsection{Macdonald polynomials}

We will work in the ring $\Lambda$ of symmetric polynomials in the infinite number of variables. As above, let $p_k$ denote the power sums,
and it is well known that $\Lambda=\mathbb{C}[p_1,p_2,\ldots]$. Let $\langle\cdot,\cdot\rangle$ denote the standard inner product on $\lambda$
such that the Schur polynomials form the orthonormal basis in $\Lambda$.

\begin{definition}
We define the following homomorphisms of $\Lambda[q,t]$:
$$\varphi_{1-q}:p_k\to (1-q^k)p_k,\quad \varphi_{\frac{1}{1-q}}:p_k\to \frac{1}{1-q^k}p_k;$$
$$\varphi_{1-t}:p_k\to (1-t^k)p_k,\quad \varphi_{\frac{1}{1-t}}:p_k\to \frac{1}{1-t^k}p_k.$$
\end{definition}

The following theorem is a definition of the {\em modified Macdonald polynomials}.

\begin{theorem}[e.g. \cite{haimlectures}] There is a homogeneous $\mathbb{Q}(q,t)$-basis $\{\Hmu\}$ of $\Lambda$
 whose elements are uniquely characterized by the conditions:
\begin{enumerate}
\item $\varphi_{1-q}(\Hmu) \in \mathbb{Q}(q,t)\{s_{\lambda}:\lambda\geq \mu\},$
\item $\varphi_{1-t}(\Hmu) \in \mathbb{Q}(q,t)\{s_{\lambda}:\lambda\geq \mu^{t}\},$
\item $\langle\Hmu, s_{(n)}\rangle = 1.$
\end{enumerate}
Here $\mu$ is an integer partition, and $\lambda$ ranges over partitions of the same
integer $\lambda=\mu=n$, with $\geq$ denoting the dominance partial ordering on partitions.
\end{theorem}

For experts, we recall below in Proposition \ref{p to h} the relation of $\Hmu$ to the standard Macdonald polynomials $P_{\mu}$.

\begin{corollary}
The polynomials for transposed partitions are related as
\begin{equation}
\label{hmu symmetry}
\widetilde{H}_{\mu^{t}}(q,t)=\widetilde{H}_{\mu}(t,q).
\end{equation}
\end{corollary}

\begin{example} 
The polynomials $\Hmu$ for $|\mu|=3$ have the form:
$$\widetilde{H}_{3}=s_{3}+(q+q^2)s_{1,2}+q^{3}s_{111},$$
$$\widetilde{H}_{111}=s_{3}+(t+t^2)s_{1,2}+t^{3}s_{111},$$
$$\widetilde{H}_{12}=s_{3}+(q+t)s_{1,2}+qts_{111},$$
\end{example}

\begin{proposition}[\cite{GH},Corollary 2.1]
The coefficients of $\Hmu$ at Schur polynomials corresponding to the hook-shaped partitions are given by the formula
\begin{equation}
\label{hook mult}
\sum_{k=0}^{n-1} a^{k}\langle \Hmu, s_{n-k,1^{k}}\rangle=\prod_{c\neq (0,0)}(1+aq^{a'}t^{l'}).
\end{equation}
\end{proposition}

\begin{proof}
By (\ref{multlambda}) we have
$$\sum_{k=0}^{n-1} (-1)^{k} a^{k}\langle \Hmu, s_{k+1,1^{n-k-1}}\rangle=\frac{1}{1-a}\Hmu(p_i=1-a^{i}).$$
Now (\ref{hook mult}) follows from the  evaluation formula for the Macdonald polynomials (\cite{Macdonald}).
\end{proof}

\begin{corollary}
\label{multW}
$$\langle \Hmu, s_{1^{n}}\rangle=q^{n(\mu^t)}t^{n(\mu)}.$$
\end{corollary}

We will need the following result.

\begin{theorem}[\cite{GH},Theorem 2.4]
\label{en}
The elementary symmetric polynomial has the following expansion in the basis $\Hmu$:
$$e_n=\sum_{\mu}\frac{(1-t)(1-q)\Pi_{\mu}(q,t)B_{\mu}(q,t)}{\prod (q^{a}-t^{l+1})(t^{l}-q^{a+1})}\cdot H_{\mu},$$
where
$$\Pi_{\mu}=\prod_{c\neq (0,0)}(1-q^{a'}t^{l'}),\quad B_{\mu}=\sum q^{a'}t^{l'},$$
\end{theorem}

\begin{definition}
Following \cite{GH}, let us define an operator $\nabla$ by the formula
$$\nabla \Hmu=q^{n(\mu^t)}t^{n(\mu)}\Hmu.$$
\end{definition}

\begin{lemma}
The following equation holds:
$$\langle \nabla f,s_{1^{n}}\rangle=\langle f,s_{n}\rangle$$
\end{lemma}

\begin{proof}
Follows from the definition of $\nabla$ and Corollary \ref{multW}.
\end{proof} 

The ring $\Lambda$ has a canonical involution $\omega$ defined by the equation
$\omega(p_i)=(-1)^{i}p_i.$ It is well known (e.g \cite{Macdonald}) that
$$\omega(s_{\lambda})=s_{\lambda^{t}}.$$
Let us describe the action of $\omega$ on the polynomials $\Hmu$.

\begin{theorem}[\cite{GH},Theorem 2.7]
\label{duality}
The following equation holds:
$$\omega\left(\Hmu(x;q,t)\right)=q^{n(\mu^{t})}t^{n(\mu)}\Hmu(x;1/q,1/t).$$
\end{theorem}

We finish this subsection by recalling the relation of $\Hmu$ to the standard Macdonald polynomials $P_{\mu}$.
The {\em integral form} Macdonald polynomials are defined by the formula
$$J_{\mu}(x;q,t)=\prod (1-q^{a}t^{l+1})P_{\mu}(x;q,t).$$

\begin{proposition}(\cite{GH})
\label{p to h}
Let $H_{\mu}(x;q,t)=\varphi_{\frac{1}{1-t}}\left(J_{\mu}(x;q,t)\right)$, then
$$\Hmu(x; q, t) = t^{n(\mu)}H_{\mu}(x; q; 1/t).$$
\end{proposition} 
 
\subsection{Diagonal Harmonics and Procesi Bundle}

Consider the diagonal action of $S_n$ on the ring of polynomials in two groups of variables $\mathbb{C}[x_1,\ldots,x_n,y_1,\ldots,y_n]$.
The space of its coinvariants
$$DH_{n}:=\mathbb{C}[x_1,\ldots,x_n,y_1,\ldots,y_n]/(\mathbb{C}[x_1,\ldots,x_n,y_1,\ldots,y_n])^{S_n}$$
is called the space of {\em diagonal harmonics}. It carries two natural gradings ($x$-degree and $y$-degree) compatible with the action of $S_n$. The following results were conjectured in \cite{haiman} and proved in \cite{haiman2}:

\begin{theorem}
The dimension of the space $DH_n$ equals $(n+1)^{n-1}$. The dimension of its anti-invariant part is equal to the $n$-th Catalan number: 
$$\dim (DH_n)^{\varepsilon}=c_n=\frac{1}{n+1}\binom{2n}{n}.$$
\end{theorem}

The key geometric tool from \cite{haiman2} is the interpretation of $DH_n$ as a space of sections of a certain sheaf over a punctual Hilbert scheme. Let $\Hilb^{n}(\mathbb{C}^2)$ denote the Hilbert scheme of $n$ points of $\mathbb{C}^2$, and let $\pi:\Hilb^{n}(\mathbb{C}^2)\rightarrow S^{n}\mathbb{C}^2$ denote the Hilbert-Chow morphism. The {\em isospectral Hilbert scheme} $X_n$ is the
{\em reduced} fibre product appearing in the following commutative diagram.  

$$\begin{matrix}
X_n & \stackrel{\rho_n}{\longrightarrow} & \Hilb^{n}(\mathbb{C}^2)\\
\downarrow & & \downarrow\\
(\mathbb{C}^2)^{n} & \longrightarrow & S^{n}\mathbb{C}^2\\
\end{matrix}$$

The central fact in the theory is that $X_n$ is normal, Cohen-Macaulay, and
Gorenstein \cite{haimpolygraph}.

\begin{definition}
Following \cite{haiman2}, we define the following bundles on $\Hilb^{n}(\mathbb{C}^2)$:
\begin{enumerate}
\item $T$ is the tautological rank $n$ bundle: $T_{I}=\mathcal{O}_{\mathbb{C}^2}/I$;
\item $T_0=T/(1)$ is a quotient of $T$ by the trivial line subbundle;
\item $\mathcal{O}(1)=\Lambda^{n} T$  is the tautological line bundle;
\item $P=\rho_{*}\mathcal{O}_{X_{n}}$ is the {\em Procesi bundle}. Its rank is $n!$,
and its fibres carry the $S_n$ action.
\end{enumerate}
\end{definition} 

Let $Z_n\subset \Hilb^{n}(\mathbb{C}^2)$ denote the punctual Hilbert scheme.
The following theorem is a main result of \cite{haiman2}:

\begin{theorem}[\cite{haiman2}]
The following isomorphisms hold:
\begin{enumerate}
\item $H^{i}(Z_n,P\otimes T^{k})=0$ for all $k,n$ and $i>0$.
\item $DH_{n}=H^{0}(Z_n,P\otimes T)$.
\item $(DH_{n})^{\varepsilon}=H^{0}(Z_n, \Lambda^{n}T)$.
\end{enumerate}
\end{theorem}

\begin{theorem}[\cite{haimpolygraph}]
Let $I_{\mu}$ denote the monomial ideal in $\mathbb{C}[x,y]$ corresponding to a Young diagram $\mu$.
Let $\mathcal{J}_{\mu}$ denote the defining ideal of the scheme-theoretic fibre $\rho_n^{-1}(I_{\mu})$  inside $(\mathbb{C}^2)^{n}$.
Then $R_{\mu}=\mathbb{C}[x_1,\ldots,x_n,y_1,\ldots,y_n]/{\mathcal{J}_{\mu}}$ is the fibre of the Procesi bundle at a point $I_{\mu}\in \Hilb^{n}(\mathbb{C}^2)$, and the following properties hold:

\begin{enumerate}
\item  $\dim R_{\mu}=n!$
\item  The $(\mathbb{C}^{*})^2$-equivariant Frobenius character of $R_{\mu}$ is $\Hmu(q,t)$.
\end{enumerate}
\end{theorem}

\begin{lemma}
Let $V_n$ denote, as above, the reflection representation of $S_n$. Then we have the following identity in the equivariant $K$-theory $K_0^{(\mathbb{C}^{*})^2}(Hilb^{n})$:
$$\Hom_{S_n}(\Lambda^{k}V,P)=\Lambda^{k}T_{0}.$$
\end{lemma}

\begin{proof}
Let us first remark that all bundles in question are $(\mathbb{C}^{*})^2$--equivariant, and thus all computations can be localized to the fixed points $I_{\mu}$. The generating function for the graded characters of the left hand side equals
$$\sum_{k}a^{k}\dim_{q,t}\Hom_{S_n}(\Lambda^{k}V,P|_{I_{\mu}})=\sum_{k}a^{k}\langle s_{n-k,1^{k}}, \ch R_{\mu}\rangle=$$
$$\sum_{k}a^{k}\langle s_{n-k,1^{k}}, \Hmu\rangle=\prod_{c\neq (0,0)}(1+aq^{a'}t^{l'}).$$
The last equation follows from (\ref{hook mult}).

On the other hand, $\dim_{q,t}T_0|_{I_\mu}=\sum_{c\neq (0,0)} q^{a'}t^{l'},$ so
$$\sum_{k}a^{k}\dim_{q,t}\Lambda^{k}T_0|_{I_\mu}=\prod_{c\neq (0,0)}(1+aq^{a'}t^{l'}).\qedhere$$
\end{proof}

\subsection{From Cherednik algebras to Hilbert schemes}
\begin{proposition}
\label{shift}
The bigraded Frobenius characters of the finite dimensional DAHA representations satisfy the following properties;
\begin{enumerate}
\item $\ch_{q,t} L_{c+1}=\nabla \ch_{q,t} L_{c},$
\item $\ch_{q,t} L_{-c}=\omega \ch_{q,t} L_{c}.$
\end{enumerate} 
\end{proposition}

\begin{proof}
The first statement is \cite[Lem. 4.4]{GS}. The second follows from the identity
$$L_{-c}=L_{c}\otimes \varepsilon$$
proved in \cite{BEG3}.
\end{proof}

\begin{proposition}
The following bigraded Frobenius characters can be computed explicitly:
$$\ch_{q,t} L_{(kn+1)/n}=\nabla^{k} h_{n}=\sum_{\mu}\frac{q^{kn(\mu^t)}t^{kn(\mu)}(1-t)(1-q)\Pi_{\mu}(q,t)B_{\mu}(q,t)}{\prod (q^{a}-t^{l+1})(t^{l}-q^{a+1})}\cdot H_{\mu},$$
$$\ch_{q,t} L_{(kn-1)/n}=\nabla^{k} e_{n}=\sum_{\mu}\frac{q^{(k+1)n(\mu^t)}t^{(k+1)n(\mu)}(1-t)(1-q)\Pi_{\mu}(q,t)B_{\mu}(q^{-1},t^{-1})}{\prod (q^{a}-t^{l+1})(t^{l}-q^{a+1})}\cdot H_{\mu}.$$
The functions $B_{\mu}$ and $\Pi_{\mu}$ are defined in Theorem \ref{en}.
\end{proposition}

\begin{proof}
By Conjecture \ref{shift} we have 
$$\ch_{q,t} L_{(kn+1)/n}=\nabla^{k} \ch_{q,t} L_{1/n},\quad \ch_{q,t} L_{(kn-1)/n}=\nabla^{k} \omega\left(L_{1/n} \right).$$ 
The representation $L_{\frac{1}{n}}$
is a trivial 1-dimensional representation of $S_n$, and its character is $h_n$.

Now the statements follow from Theorems \ref{en} and \ref{duality}.
\end{proof}

We finish this section by mentioning the relation to the recent work on ``extended superpolynomials''
introduced by S. Shakirov, A. Morozov et al. In the series of papers  (\cite{MMS}, \cite{MMSS}, \cite{mm}) they assigned to a $(m,n)$ torus knot a polynomial in variables $q, t, p_i$ specializing to HOMFLY-PT polynomial for $q=1$ and $p_i=\frac{1-a^i}{1-t^i}.$
For general $t$, the polynomials are supposed to compute the Poincar\'e polynomials of the HOMFLY homology of the corresponding knot.

Let $P_{m,n}(a,q,t)$ denote the Poincar\'e polynomial for the {\em reduced} HOMFLY homology of the $(m,n)$ torus knot (``superpolynomial'') 
and let $H_{m,n}(q,t;p_i)$ denote the ``extended superpolynomial'' from \cite{mm} (unfortunately, the authors of \cite{mm} do not give any definition of this polynomial).
 
The following formula is proposed in \cite{mm}:
\begin{equation}
\label{reductionh}
\frac{1-a}{1-t}P_{m,n}(a,q,t)=H_{m,n}\left(q,t;p_i=\frac{1-a^i}{1-t^i}\right).
\end{equation}

It follows from (\ref{multlambda}) that 
$$P_{m,n}(a,q,t)=\frac{1}{1-a}\ch L_{m/n}(p_i={1-a^i}).$$
Therefore it is natural to define the ``extended superpolynomial'' by the formula:
$$H_{m,n}(x;q,t)=\frac{1}{1-t}\varphi_{1-t}\ch L_{m/n},\quad \ch L_{m/n}=\frac{1}{1-t}\varphi_{\frac{1}{1-t}}H_{m,n}(x;q,t).$$

This equation can be compared with (\ref{p to h}), what allows one to recast the formulas for $H_{m,n}$ in terms of 
the standard Macdonald polynomials. Such formulas were previously conjectured in \cite{mm}, where $\nabla$
was denoted by $e^{\hat{W}}$.

\section{Differentials}
\label{sec:Diffs}
In this section, we construct a family of differentials on the groups
\(H_{m/n}\). The definition of these differentials was motivated by
the main conjecture of  \cite{DGR}, which states 
 that the homology \(\Hbar(K)\) should admit a family of differentials 
\(d_N\), such that for \(N>0\), the
homology of \(\Hbar(K)\) with respect to \(d_N\) is the \(\mathfrak{sl}(N)\) homology
of \(K\).  More precisely, we have 
\begin{conjecture}\cite{DGR}
For each \(N\in \ZZ\), there are maps \(d_N: \Hbar(K) \to \Hbar(K)\)
with the following properties:
\begin{enumerate}
\item (Grading) \(d_N\) lowers \(a\)-grading by \(2\) and
  \(q\)-grading by \(2N\). For \(N \neq 0\), \(d_N\) lowers \(\delta\)
  grading by \(|N|-1\), while \(d_0\) lowers \(\delta\)-grading by 2. 
\item (Symmetry) \(\Hbar(K)\) admits an involution \(\iota\) with the
  property that \(\iota d_N = d_{-N} \iota\). 
\item (Anticommutativity) \(d_N d_M = -d_M d_N\) for \(k,l \in \ZZ\). 
In particular, \(d^2_N = 0 \).
\item (Homology) For \(N>0\), \(H(\Hbar(K), d_N) \cong \hH_N(K)\), where the
  latter group denotes the reduced \(\mathfrak{sl}(N)\) homology of \(K\) defined
  by Khovanov and Rozansky in \cite{KR1}. Moreover,  
\(H(\Hbar(K), d_0)\) is isomorphic to the knot Floer homology 
\( \widehat{HFK}(K)\). 

\end{enumerate}
\end{conjecture}  
In \cite{R}, it is proved that for \(N>0\), there are spectral
sequences with \(E_1\) term \(\Hbar(K)\) which converge to
\(H_N(K)\). For \(N>0\), the conjecture is more or less equivalent to
the statement that these spectral sequences converge at the \(E_2\) term. 

For all \(N\) in \(\ZZ\), we will construct maps  \(d_N: \Hmnf \to \Hmnf \)
satisfying the analogues of properties (1)-(3) in the conjecture. 
We are unable to prove
the analogue of property (4), but we will show that 
 \(H(\Hmnf), d_1) \cong \ZZ \cong \hH_{1,T(n,m)}\). 
 In the next section, we give an explicit description of the limit \(\lim_{m \to \infty} H(\Hmnf,
d_N)\). When \(k=2\), it is possible to compare this description with
with previously known calculations of the  Khovanov
homology and see that they agree \cite{GOR}.

\subsection{Definitions}

Recall that \(\Hmnf = \Hom_{S_n}(\Lambda^*\hh, \gr^\fF \lmn)\) where \(\fF\) is a decreasing filtration on \(\lmn\) with the property that \(\gG_i \cdot \fF_j \subset \fF_{i+j}\). Given \( \varphi
\in \Hom_{S_n}(\Lambda^k\hh, \gr^\fF \lmn)\), we want to define \(d_N(\varphi)
\in \Hom_{S_n}(\Lambda^{k-1}\hh, \gr^\fF \lmn)\). To do so, we first define maps
$$ \partial_N: \Hom_{S_n}(\Lambda^k\hh,  \lmn) \to
\Hom_{S_n}(\Lambda^{k-1}\hh,  \lmn) $$ 
with the property that for \(N \neq 0\), \(d_N(\fF_i) \subset
\fF_{i+N}\). We will then take \(d_N\) to be the induced map on
associated graded groups. 

Observe that \( \Hom_{S_n}(\Lambda^k \hh, \hh\otimes  \Lambda^{k+1} \hh)\)  is one-dimensional, since \(\hh \otimes \Lambda^{k+1} \hh\) has a unique summand isomorphic to \(\hh\). Consider the map 
\(\hh \otimes \Lambda^{k+1} \hh \to \Lambda^{k} \hh\) given by contraction. Dualizing and using the fact that \(\hh^* \cong \hh\) gives an explicit generator \(f_k\) of 
\( \Hom_{S_n}(\Lambda^k \hh, \hh \otimes  \Lambda^{k+1} \hh\)).

\begin{definition}
Let \(S_n\) act on \(\Hh_{m/n}\) by conjugation.
Given \(\alpha \in \Hom_{S_n}( \hh,\Hh_{m/n})\) and \(\phi \in
\Hom_{S_n}(\Lambda^k\hh,  \lmn)\), we define \(\partial_{\alpha} \phi \in
\Hom_{S_n}(\Lambda^{k-1}\hh , \lmn)\) to be the following composition:
$$ \Lambda^{k-1}\hh \xrightarrow{f_{k-1}} \hh \otimes \Lambda^k \hh \xrightarrow{\alpha\otimes \phi} \Hh_{m/n} \otimes \lmn \xrightarrow{\cdot} \lmn $$
where the final map is given by the action of \(\Hh_{m/n}\) on \(\lmn\). 
\end{definition}

Equivalently, if we let \(\mu: \Hh_{m/n} \otimes \lmn \to \lmn\) be the multiplication map, then 
$ \partial_\alpha(\varphi) = \mu(\alpha \neg \varphi). $ From now on, we will abuse notation and just
write \(\alpha \neg \varphi\) for the right-hand side, leaving the \(\mu\) understood.

If the image of \(\alpha\) is contained in \(\gG_i\), and the image of \(\varphi\) is contained in \(\fF_j\), then the image of \(\partial_{\alpha}(\varphi)\) will be contained in \(\fF_{i+j}\). This leads to the following
\begin{definition}
Suppose the image of \(\alpha \in \Hom_{S_n}( \hh,\Hh_{m/n})\) is contained in \(\gG_i\) but not in \(\gG_{i+1}\). Then we define
$$d_\alpha: \Hom_{S_n}(\Lambda^k \hh, \gr_*^\fF \lmn) \to  \Hom_{S_n}(\Lambda^{k-1} \hh, \gr_{*+i}^\fF \lmn)$$ to be the map on associated graded groups induced by \(\partial_\alpha\). 
\end{definition}
Observe that \(d_\alpha\) depends only on the image of \(\alpha \) in the associated graded group\\ \(\Hom_{S_n}(\hh,\gG_i/\gG_{i-1})\). 
The following result is immediate from the definition:
\begin{proposition} (Grading)
\label{Prop:Grading}
Suppose \(\alpha \in  \Hom_{S_n}( V,\gG_i)\) is homogeneous of grading \(k\). Then \(d_\alpha:\Hmnf \to \Hmnf \) raises filtration grading by \(i\), shifts the polynomial grading by \(k\), and lowers the exterior grading by \(1\). 
\end{proposition}

The symmetry \(\Phi:\CA \to \CA\) commutes with the action of \(S_n\), so given \(\alpha \in \Hom_{S_n}( \hh,\CA)\), we can define \(\Phi(\alpha) = \Phi \circ \alpha \in \Hom_{S_n}( \hh,\CA)\). We have

\begin{proposition}
(Symmetry) \(\Phi \circ d_{\Phi(\alpha)} = d_\alpha \circ \Phi\).
\end{proposition}
\begin{proof}
\(\partial_{\alpha}(\Phi(\varphi))\) is given by the following composition:
$$ \Lambda^{k-1}\hh \xrightarrow{f_{k-1}} \hh \otimes \Lambda^k \hh \xrightarrow{\alpha\otimes \Phi(\varphi)} \CA \otimes \lmn \xrightarrow{\cdot} \lmn. $$
Since \(\alpha \Phi(\varphi) = \Phi(\Phi(\alpha)\varphi)\), this is the same as 
$$ \Lambda^{k-1}\hh \xrightarrow{f_{k-1}} \hh \otimes \Lambda^k \hh \xrightarrow{\Phi(\alpha)\otimes \varphi} \CA \otimes \lmn \xrightarrow{\cdot} \lmn\xrightarrow{\Phi} \lmn. $$
Thus \(\partial _\alpha \circ \Phi = \Phi \circ \partial_{\Phi(\alpha)}\). 
\end{proof}

Now suppose \( \beta\) is another element of \(\Hom_{S_n}(\hh,\CA)\). We
consider conditions under which  \(d_\alpha\) and \(d_{\beta}\) anticommute. We say that an element of \(\CA\) is {\it pure monomial} of filtration \(i\) if it can be expressed as a product of at most
 \(i\) elements of \(\htp \cup \htp^*\). (In particular, no elements of \(S_n\) are used in the product.)
An element of \(\CA\) is {\it pure polynomial} of filtration \(i\) if it is a sum of pure monomials of filtration \(i\). 

\begin{lemma}
\label{lem:CommutatorFiltration}
If \(\alpha\) and \(\beta\) are pure polynomial of filtration \(i\) and \(j\), then the commutator \([\alpha,\beta]\) is contained in \(\gG_{i+j-2}\). 
\end{lemma}
\begin{proof}
Without loss of generality, we may assume \(\alpha\) and \(\beta\) are pure monomials. By repeatedly applying the relation
$$[x,yz] = [x,y]z + y[x,z],$$
we see that it is enough to prove the statement in the case where \(\alpha\) and \(\beta\) are pure monomial of degree 1. In this case, the claim follows immediately from 
 the defining relations for \(\CA\). 
\end{proof}

More generally, we say that \(\alpha \in \Hom_{S_n}(\hh,\CA)\) is pure polynomial of filtration \(i\) if every element in the image of \(\alpha\) is pure polynomial of filtration \(i\). 

\begin{proposition}
\label{prop:anticommute}
(Anticommutativity) Suppose \(\alpha, \beta \in  \Hom_{S_n}(\hh,\CA)\) are pure polynomial. Then 
\(d_\alpha d_\beta = - d_\beta d_\alpha\). 
\end{proposition}

\begin{proof}
If \(\alpha = \sum \alpha_i \overline{e}_i^*\), and \(\beta = \sum \beta_i \overline{e}_i^*\), 
 we define a two-form $$[\alpha,\beta] = \frac{1}{2}\sum_{i,j} [\alpha_i,\beta_j] \, \overline{e}_i^* \wedge \overline{e}_j^* $$
Then we have 
\begin{align*}
\partial_\alpha \partial_\beta (\varphi ) + \partial_\beta \partial_\alpha
(\varphi) & =  \alpha \neg (\beta \neg \varphi) + \beta \neg (\alpha \neg \varphi) \\
& = [\alpha,\beta] \neg \varphi.
\end{align*} 
Now \(\alpha\) and \(\beta\) are pure polynomial, of filtration (let us
say) \(i\) and \(j\). Then by the lemma, 
 \([\alpha_i ,\beta_j] \in \gG_{i+j-2}\). But
 \(\partial_{\alpha}\) and \( \partial_{\beta}\) raise filtration
 grading by \(i\) and \(j\) respectively, so passing to associated
 gradeds, we see that \(d_{\alpha} d_\beta + d_\beta d_\alpha = 0\). 
\end{proof}

\begin{remark}
 The construction of \cite{R} can be adapted to yield spectral sequences starting at \(\overline{\hH}(K)\) and converging to the unreduced \(\mathfrak{sl}(N)\) homology. In the latter case, the image of the reduced \(\hH(K)\) in \(\overline{\hH(K)}\) will usually not be a subcomplex with respect to the differentials in the spectral sequence. However it is true that the kernel of the natural projection \(\overline{\hH}(K) \to \hH(K)\) is preserved, and the resulting spectral sequence on the quotient agrees with the spectral sequence on reduced homology constructed in \cite{R}. 
 
 The construction above can also be adapted to the unreduced case. If we choose \(\overline{\alpha} \in \Hom_{S_n}(\hhu, \overline{\Hh}_{m/n})\), we get a differential 
\(\partial_{\overline{\alpha}}: \overline{H}_{m/n}^{\fF} \to 
: \overline{H}_{m/n}^{\fF}\), where as in section~\ref{subsec:unreduced}
\( \overline{H}_{m/n}^{\fF} = \Hom_{S_n}(\Lambda^* \hhu, \gr^{{\fF}}\overline{L}_{\frac{m}{n}})\). 
 It is easy to see that 
if \(\alpha \in \Hom_{S_n}(\hhu,\CC[\hhu^*])\), 
the kernel of the projection \(\overline{H}_{m/n}^\fF \to \Hmnf\) is preserved by \(d_{\overline{\alpha}}\), and that the induced differential on \(\Hmnf\) is \(d_{\pi(\overline{\alpha})}\), where \(\pi_{\alpha} \in \Hom_{S_n}(\hh,\CC[\hh^*])\) is obtained by composing \(\alpha\) with  the inclusion \(\hh \to \hhu\) and the projection \(\CC[\hhu^*] \to \CC[\hh^*]\).

 \end{remark}

We can now establish the proposition about differentials stated in the introduction. 
\begin{proof} (of Proposition~\ref{prop:diffs}) In order to construct the differentials \(d_N\),
we see that it is enough to choose homomorphisms \(\widetilde{\alpha}_N: \hh \to
\CA\) with the following properties:
\begin{enumerate}
\item \(\widetilde{\alpha}_N\) is homogenous of degree \(N\). 
\item \(\widetilde{\alpha}_{-N} = \Phi(\alpha_N)\).
\item \(\widetilde{\alpha}_N\) is pure polynomial of filtration \(|N|\) (for \(N
  \neq 0\)) or \(2\) (for \(N=0\)). 
\end{enumerate}
For \(N>0\), we see that the image of \(\widetilde{\alpha}_N\) must be
contained in the homogenous polynomials of degree \(N\) in \(\CC[\hh^*]\),
We choose \(\widetilde{\alpha}_N=\pi(\alpha_N)\), where \(\alpha_N = \sum_i {x_i^N} x_i^*\). 

To define \(d_0\), we must choose \(\beta_0: \hh \to \CA\) which is of degree \(1\) in the \(x_i\) and the \(y_i\) separately. Up to lower order terms in the filtration,  there is essentially a unique choice: it is given by 
$$ (\beta_0)_i = \widetilde{x}_i\widetilde{y}_i - \frac{1}{n} \sum_j \widetilde{x}_j\widetilde{y}_j,$$ 
where $$\widetilde{x_i} = x_i - \frac{1}{n} \sum_j x_j \quad \text{and} \quad \widetilde{y_i} = y_i - \frac{1}{n} \sum_j y_j. $$
Finally, to see that \(d_0 = - \Phi \circ d_0\) observe that \(\Phi(\beta_0)\) is given by the a similar formula, but with \(x\)'s and \(y\)'s reversed and an additional factor of \(-1\). Using the commutation relation, we see that \(-\Phi(\beta_0)\) and \(\beta_0\) agree modulo elements of \(\gG_0\), so they induce the same map on associated gradeds. 
\end{proof}

\begin{remark}
It is clear from the above construction that we have considerable latitude in our choice of \(\widetilde{\alpha}_N\). In section~\ref{subsec:diffCalc}, we will show that in the limit as \(m \to \infty\), {\it any} choice of \(\widetilde{\alpha}_N\) gives rise to the same differential up to multiplication by scalars. It is unclear to us whether this property continues to hold for finite \(m\).

 In contrast,  different choices of homogenous  \(\alpha\) of degree \(N\) on the unreduced homology can give genuinely different differentials. As \(m \to \infty\), there is some evidence (discussed in the next section) that \(
\alpha = \alpha_N\) is the correct choice. Consequently, we have chosen to take \(\widetilde{\alpha}_N = \pi(\alpha_N)\) in the reduced case. 
\end{remark}

\subsection{The exterior derivative}
The differentials of the previous section were constructed by contraction with \(\alpha \in \Hom_{S_n}(\hh,\CA)\). We now consider   the dual construction, in which we take wedge product with \(\alpha\). 
Let \(g_k: \Lambda^{k+1} \hhu \to \hhu \otimes \Lambda^{k} \hhu\) be dual to the map given by wedge product. 
\begin{definition}
For \(\varphi \in \Hom_{S_n}(\Lambda^k\hhu, \CC[\hhu^*])\), we define 
\(\nabla_c \varphi\) to be the following composition:
\begin{equation*}
\Lambda^{k+1}\hhu \xrightarrow{} \hhu \otimes \Lambda ^k\hhu \xrightarrow{\alpha_1 \otimes \varphi} \Hhu_c \otimes \CC[\hhu^*] \xrightarrow{} \CC[\hhu^*]
\end{equation*}
where \(\alpha_1:\hhu \to \Hhu_c\) is given by \(\alpha_1(e_i) = D_i\). 
\end{definition}

Under the identification \(\Hom_{S_n}(\Lambda^* \hhu, \CC[\hhu^*]) \cong( \Lambda^* \hhu^*\otimes \CC[\hhu^*])^{S_n}\), \(\nabla_c(\phi) = \mu(\alpha_1 \wedge \phi)\) where \(\mu\) denotes the multiplication map. As before, we will omit \(\mu\) from the notation.  Suprisingly, the value of 
\(\nabla_c(\varphi)\) does not depend on \(c\). 

\begin{lemma}
\(\nabla_c(\varphi)= \nabla_0(\varphi)\). 
\end{lemma}

\begin{proof}
Let \(p_r \in \CC[\hhu^*]\) denote the symmetric \(r\)-th power function. Then it is easy to see that 
the statement holds for \(p_r\) and \(\nabla_c p_r = r \alpha_{r-1}\). Since these elements generate
\(\Hom_{S_n}(\Lambda^* \hhu, \CC[\hhu^*])\), the claim follows from 
 Lemma 5.5 of \cite{G1}, where it is shown that if \(\varphi_1,\ldots,\varphi_k \) are in \(\Hom_{S_n}( \hhu, \CC[\hhu^*])\), then 
$$ \nabla_c (\varphi_1\wedge \ldots \wedge \varphi_k) = \sum_{i=1}^k (-1)^{i-1} \varphi_1 \wedge \ldots \wedge \nabla_c \varphi_i \wedge  \ldots \wedge \varphi_k.$$
\end{proof}

From now on we will omit the \(c\) from our notation, and simply write \(\nabla(\varphi)\). 
If we view \(\Hom_{S_n}(\Lambda^k \hhu, \CC[\hhu^*])\) as the space of \(S_n\) invariant polynomial \(k\)-forms on \(\hhu\), the lemma says that \(\nabla(\varphi)\) is the exterior derivative of \(\varphi\). 

From now on we will omit the \(c\) from our notation and simply write \(\nabla(\varphi)\). 
If we let \(c=m/n\), we see that \(\nabla\) descends to a well-defined map \(\nabla: \Hom_{S_n}(\Lambda^*\hhu, \overline{L}_{\frac{m}{n}}) \to  \Hom_{S_n}(\Lambda^*\hhu, \overline{L}_{\frac{m}{n}}) \). Moreover, it is easy to see that \(\nabla\) preserves \( \Hom_{S_n}(\Lambda^* \hh, \CC[\hh^*]) \subset  \Hom_{S_n}(\Lambda^*\hhu, \CC[\hhu^*])\), and thus descends to a well-defined map \(\nabla: \Hmn \to \Hmn\). 

Next, we consider how \(\nabla\) affects the filtration. We start with a preparatory result.

\begin{lemma}
Suppose \(\varphi \in  \Hom_{S_n}(\Lambda^k \hh, \lmn)\). Then \(\varphi \in (\aa^i)^ \perp\) if and only if 
\( \partial_{{-j_1} }\circ \ldots \circ \partial_{{-j_k}} (\varphi) \in (\aa^i)^ \perp\) for all sequences \(j_1, \ldots, j_k\). 
\end{lemma}

\begin{proof}
If \(\varphi \in (\aa^i)^\perp\), Proposition~\ref{Prop:Grading} implies that \( \partial_{{-j_1} }\circ \ldots \circ \partial_{{-j_k}} (\varphi) \in (\aa^i)^ \perp\). For the converse, let \(W = \text{im} \, \varphi \subset \lmn\) . The Dunkl pairing defines an \(S_n\) equivariant homomorphism
\(W \otimes \lmn \to \CC\). The space of possible homomorphisms is parametrized by \(\Hom_{S_n} (\Lambda^k\hh, \lmn)\), which we can view as a quotient of 
 \(\Hom_{S_n} (\Lambda^k\hh, \CC[\hh])\). The latter space is spanned by elements of the form \(u \widetilde{\alpha}_{-j_1} \wedge \ldots \wedge \widetilde{\alpha}_{-j_k}\), where \(u \in \CC[\hh]^{S_n}\). The corresponding homomorphism is given by 
 \(u \partial_{-j_1}\circ  \ldots \circ \partial_{-j_k}\). Thus if  \( \partial_{{-j_1} }\circ \ldots \circ \partial_{{-j_k}} (\varphi) \in (\aa^i)^ \perp\) for all sequences \(j_1,\ldots, j_k\), \(\varphi \in (\aa^i)^ \perp\) as well. 
\end{proof}
 
\begin{lemma}
\label{Lem:nabla filt}
 \(\nabla( \fF^{alg}_i) \subset \fF^{alg}_{i-1}.\)
\end{lemma}

\begin{proof}
We need to show that if  \(\varphi \in \Hom_{S_n}(\Lambda^k \hh, \lmn)\) is in \( (\aa^i)^\perp\), then 
\(\nabla \varphi \in (\aa^{i-1})^\perp\). 
By the previous lemma, it suffices to check that 
\( \partial_{-{j_1} }\circ \ldots \circ \partial_{{-j_{k+1}}} (\nabla \varphi) \in (\aa^{i-1})^ \perp\) for all sequences \(j_1, \ldots, j_{k+1}\). We compute
\begin{align*}
 \partial_{-j_1}\circ \ldots \circ \partial_{-j_{k+1}} (\nabla \varphi)  & = \sum_{i=1}^{k+1} (-1)^{i-1} (\alpha_{j_i} \neg \alpha_1) \cdot  \partial_{{-j_1} }\circ \ldots \partial_{-j_{i-1}} \circ \partial_{-j_{i+1}} \circ \ldots  \circ \partial_{{-j_k}} (\varphi).
 \end{align*}
 Now \(\alpha_{j_i} \neg \alpha _1 \in \CC[\hhu]^{S_n}\), so if \(x \in \aa^{i-1}\), we have
 $$ ( \partial_{-j_1}\circ \ldots \circ \partial_{-j_{k+1}} (\nabla \varphi),x) = 
 \sum_{i=1}^{k+1} (-1)^{i-1} (\partial_{{-j_1} }\circ \ldots \partial_{-j_{i-1}} \circ \partial_{-j_{i+1}} \circ \ldots  \circ \partial_{{-j_k}} (\varphi), \Phi(\alpha_{j_1} \neg \alpha_1) x).$$
 Since \(  \Phi(\alpha_{j_1} \neg \alpha_1) x \in \aa^i\),  all the terms on the right-hand side vanish. 
\end{proof}

 \begin{lemma}
\label{Lem:d1homotopy}
Suppose \(\varphi \in \Hom_{S_n}(\Lambda^k\hhu, \CC[\hhu^*])\) is homogenous of degree \(r\). Then
$$( \nabla \circ \partial_1 + \partial_1 \circ \nabla) \phi  = (r+k)\varphi. $$
\end{lemma}

\begin{proof}
Viewing \(\nabla\) as exterior differentiation, and \(\partial_1\) as contraction with \(\alpha_1\), we see that 
  \( \nabla \circ \partial_1 + \partial_1 \circ \nabla = \lL_{\alpha_1} \) is the Lie derivative with respect to \(\alpha_1\). 
Now  $$\lL_{\alpha_1}(\phi \wedge \psi) = \lL_{\alpha_1}(\phi) \wedge \psi + \phi \wedge \lL_{\alpha_1}(\psi),$$ so if the statement holds for \(\phi\) and \(\psi\), it holds for \(\phi \wedge \psi\) as well. 
When \(\phi = p_r\)  and \(\phi = \nabla p_r = r\alpha_{r-1}\), the lemma is easily verified. Since  \(\Hom_{S_n}(\Lambda^*\hh, \CC[\hh^*])\) is generated by these elements, the lemma holds in general. 
  \end{proof}
  
  Since \(\nabla\) and \(\partial_1\) both preserve \( \Hom_{S_n}(\Lambda^k\hh, \CC[\hh^*]) \subset \Hom_{S_n}(\Lambda^k\hhu, \CC[\hhu^*])\), the formula of the lemma also holds for 
  \(\phi \in \Hmn\). By Lemma~\ref{Lem:nabla filt}, \(\nabla\) reduces filtration level by 1. If we let 
  \(\widetilde{\nabla}:\Hmn^{\fF^{alg}} \to \Hmn^{\fF^{alg}}\) be the associated graded map, we see that 
  $$ \widetilde {\nabla} \circ d_1 + d_1 \circ \widetilde{\nabla} = (r+k) \text{Id}.$$
  
  \begin{corollary}
  \(H(\Hmn^{\fF^{alg}},d_1)\) is one-dimensional and is generated by \(1 \in \eee \lmn\). 
  \end{corollary}
  
The analogous result for the HOMFLY homology was conjectured in \cite{DGR} and proved in \cite{R}.

\section{Stable $\mathfrak{sl}(N)$ homology} 
\label{sec:Infinite limit}
We have already observed in section~\ref{sec:Examples}  that as \(m\to \infty\), the group \(H_{m/n}\) has a well-defined limit 
\begin{equation*}
H_{\infty/n} := \lim_{m \to \infty} (a^{-1}q)^{(n-1)(m-1)} H_{m/n} = \Hom_{S_n}(\Lambda^*\hh,\CC[\hh^*])
\end{equation*}
and that this limit is the tensor product of an exterior algebra with a symmetric algebra:
\begin{equation*}
H_{\infty/n}  \cong 
\Lambda^*(\xi_1,\ldots,\xi_
{n-1}) \otimes \CC[u_1,\ldots,u_{n-1}].
\end{equation*}
In this section, we show that the algebraic and inductive filtrations on \(\lmn\) tend to a  limiting filtration 
$$
\ldots \fF_{i-1} \lin \subset \fF_i \lin \subset \fF_{i+1}\lin  \ldots
$$
on \(\lin\). The associated filtration on  \(H_{\infty/n}\) is induced by a grading. We will show that there is a preferred choice of isomorphism
$$\rho: \Lambda^*(\xi_1,\ldots,\xi_
{n-1}) \otimes \CC[u_1,\ldots,u_{n-1}] \to H_{\infty/n}$$
which makes \(H_{\infty/n}\) into a graded ring with homogenous generators 
\(\xi_1,\ldots, \xi_{n-1},u_1,\ldots, u_{n-1}\) with respect to the filtration grading \(f\) defined by 
$$f(\xi_i) = -i \quad \quad f(u_i) = 1-i,$$
and that \(\fF_i \Hin \) is the subspace spanned by all homogenous elements with filtration grading \(\leq i\). 

The elements \(u_i\) are determined up to scalars by the relations
$$P_i(u_k) = 0 \ \ \text{for} \ \ i \neq k+1 \quad \text{and} \quad  P_{k+1}(u_k) \neq 0,$$  where 
\(P_k\) is an appropriately scaled version of the quantum Olshanetsky-Perelomov Hamiltonian \(\sum_i D_i^k\). The  \(\xi_i\) satisfy \((i+1) \xi_i = \nabla u_i\). Using the calculation of the action of the O-P Hamiltonians in \cite{G1}, we give an explicit expression for the \(u_i\) in terms of elementary symmetric functions. This, in turn, allows us to compute the action of \(d_N\) on \(\xi_i\).

\vskip0.1in
\noindent {\bf Normalizations:}
Throughout this section, we use the rescaled Dunkl operators
\begin{align*}
 D_i = \frac{1}{c} \frac{\partial}{\partial x_i} + \sum_{j\neq i}\frac{s_{ij}-1}{x_i-x_j},\quad 
 \nild_i=\sum_{j\neq i}\frac{s_{ij}-1}{x_i-x_j}.
 \end{align*}
 However, we do not rescale the \(\nabla\) operator; it is still given by 
 $$ \nabla \varphi = \left(\sum \frac{\partial}{\partial x_i} x_i^*\right) \wedge \varphi.$$

The vector space $L_{\infty/n}$ is naturally a module over algebra $\Hh_{\infty/n}$ first studied in 
\cite{EG}. It has the same generators as $\Hh_c$ with relations:
$$ [v,w] = \sum_{s\in \mathcal{S}} \langle
v,\alpha_s\rangle \langle \alpha_s^\vee, w \rangle \cdot s.$$

As it is shown in \cite{EG} (or \cite{O} for more elementary proofs) the spherical algebra 
$\mathbf{e}\Hh_{\infty/n}\mathbf{e}$ is the algebra of regular functions on so called Calogero-Moser
space $CM_n$ and $\Hh_{\infty/n}=End_{\mathbf{e}\Hh_{\infty/n}\mathbf{e}}(\Hh_{\infty/n}\mathbf{e})$ 
with $\Hh_{\infty/n}\mathbf{e}$ forming vector bundle of rank $n!$ over $CM_n$. 

The algebra $\mathbf{e}\Hh_c\mathbf{e}$ is a quantization of $\mathbf{e}\Hh_{\infty/n}\mathbf{e}$ induced 
by the Poisson bracket $\{\cdot,\cdot\}$ defined by:
$$\{ f,g\}:=\lim_{c\to \infty} c(\tilde{f}\tilde{g}-\tilde{g}\tilde{f}),$$
where $\tilde{f},\tilde{g}$ are lifts of elements $f,g\in \mathbf{e}\Hh_{\infty/n}\mathbf{e}$
to $\mathbf{e}\Hh_c\mathbf{e}$.

It is shown in \cite{EG} (or \cite{O} for more elementary proofs) the spherical algebra 
$\mathbf{e}\Hh_{\infty/n}\mathbf{e}$ is the algebra of regular functions on the Calogero-Moser
space $CM_n$.   Moreover, 
$\Hh_{\infty/n}=\mathrm{End}_{\mathbf{e}\Hh_{\infty/n}\mathbf{e}}(\Hh_{\infty/n}\mathbf{e})$ 
and $\Hh_{\infty/n}\mathbf{e}$ forms a vector bundle of rank $n!$ over $CM_n$.
The algebras $\eee \Hh_c \eee$ give a deformation quantization of $\eee \Hh_\infty/n \eee$, and
this structure is recalled by a Poisson bracket on the latter.

It has the same generators as $\Hh_c$ with relations:
$$ [v,w] = \sum_{s\in \mathcal{S}} \langle
v,\alpha_s\rangle \langle \alpha_s^\vee, w \rangle \cdot s.$$

We will also use rescaled versions of the \(q\)-grading, \(a\)-grading, and filtration. As indicated above, we shift the \(q\)-grading on \(\Hmn \) up by a factor of \(\mu = (m-1)(n-1)\), and the \(a\)-grading down by the same amount. In addition, we shift the filtration on \(\Hmn\) down by a factor of \(\mu/2\). The net result is that the image of the  element \(1 \in \CC[\hh^*] \) in \(\lmn\) has rescaled \(q\)-grading, \(a\)-grading, and filtration level all equal to \(0\). 
\subsection{More about \(\lin\)}

In this section, we discuss the action of Dunkl operators on the representation
\(\lin = \CC[\hh^*]\) and describe a natural filtration on \(\Hin\). In the next section, we will show that this filtration coincides with the limit of both \(\fF^{alg} \lmn\) and \(\fF^{ind} \lmn\) as \(m\to \infty\). 

Our starting point is the following lemma, whose proof is immediate from the definition: 
\begin{lemma}
\label{lem:symDunkl}
If \(f \in \eee \lin = \CC[\hh^*]^{S_n}\), and \(g \in \lin\), then \(\nild_i(fg) = f \nild_i(g)\). 
\end{lemma}

\begin{corollary}
\label{cor:symdiff}
If \(f \in \eee \lin\) and \(\xi \in \Hin\), then \(\partial_{\alpha} (f \xi)= f \partial_\alpha \xi\). 
\end{corollary}

\begin{corollary}
Let \(W:\eee \lin \to \eee_- \lin\) and \(\widetilde{W}:\eee_- \lin \to \eee \lin\) be multiplication by \(W(x)\) and \(W(y)\) respectively, where \(W\) denotes the Vandermonde determinant. Then \(W\circ \widetilde{W}\) and \( \widetilde{W}\circ W\) are both multiplication by a nonzero constant \(\nu\).
\end{corollary}
\begin{remark}
The constant is $\nu = (-1)^{n \choose 2} n! (n-1)! \ldots 1!$, although we will not need this fact here. 
\end{remark}
\begin{proof}
Every element of \(\eee_- \lin\) can be written as \(fW(x)\), where \(f \in \eee \lin\). Thus the claim follows from the lemma together with the fact that  
\(W(y) \cdot W(x)\) is the stated constant, for which see \cite[Thm. 2.6]{Dunkl-Hanlon} or \cite[Eq. 0.4]{Garsia-Wallach}.
\end{proof}

\begin{lemma} (Kostant decomposition)
The map \(\rho: \CC[\hh]_{S_n} \otimes \CC[\hh^*]^{S_n} \to \lin\) given by 
$$\phi(f(y)\otimes g(x)) = g(x) (f(y) \cdot W(x))$$  is a linear isomorphism. 
\end{lemma}

\begin{proof}
The graded dimensions of the two spaces agree, so it suffices to show \(\rho\) is injective. 
Consider the pairing \(\CC[\hh]_{S_n}\otimes \CC[\hh]_{S_n} \to \CC\) given by
$$(f_1,f_2) = f_1(y)f_2(y)\cdot W(x).$$
Up to a factor of \(\nu\), this is the same as taking the projection of \(f_1 f_2\) onto the 1--dimensional subspace of top degree in \(\CC[\hh]_{S_n}\). It is well-known that this is a nonsingular pairing, so given a basis \(f_i\) of \(\CC[\hh]_{S_n}\), we can find a dual basis \(f^i\) with the property that \((f_i,f^j) = \delta_i^j\). It follows that if 
 $$ h= \sum g_i(x)(f_i(y) \cdot W(x))=0,$$
then  \(f^i(y) \cdot h = g_i(x) = 0\). Thus \(\rho\) is injective.
\end{proof}

Next, we consider an analog of the Kostant decomposition for \(\Hom_{S_n}(\Lambda^* \hh, \lin)\). The following 
lemma follows from duality $\Phi(\partial_i)=\partial_{-i}$ and the Vandermonde formula:

\begin{lemma}
\(\displaystyle \widetilde{W} = \partial_{-1} \circ \partial_{-2} \circ \ldots \circ \partial_{-n}.  \) 
\end{lemma}

We define \(\partial_{\, \hh}  W(x) \subset \Hom_{S_n}(\Lambda^*\hh, \lmn)\) to be the space spanned by the images of \(W(x)\) under repeated applications of operators of the form \(\partial_{\alpha}\), where \(\alpha \in \Hom_{S_n}(\hh, \CC[\hh])\) is  polynomial in the Dunkl operators.   \(\partial_{\, \hh}  W(x) \) is analogous to the space \(\CC[\hh] \cdot W(x)\subset \lin \) appearing in the Kostant decomposition. 
If \(I=\{i_1,\ldots,i_k\} \subset \{1,\ldots,{n-1}\}=\overline{I}\), we write 
 \( \partial_{-I} = \partial_{-i_1}\circ \ldots \circ \partial_{-i_k} \). 
\begin{lemma}
 The set \(\{\partial_{-I}(W(x)) \, | \, I \subset \overline{I} \}\) is a basis for \(\partial_{\, \hh} W(x)\). \end{lemma}
 
 \begin{proof} Any \(\alpha \in  \Hom_{S_n}(\hh, \CC[\hh])\) can be written as \(\alpha = \sum_i f_i \alpha_i\), where  \(f_i \in \CC[\hh]^{S_n}\). Note that if \(f_i\) is an invariant polynomial with vanishing constant term, then \(f_i(y) \cdot W(x) = 0\), since it is  an alternating polynomial in \(\CC[\hh^*]\) with degree strictly  than that of \(W(x)\). Thus 
 $$\partial_{\alpha}W( x) = \sum_i \partial_i (f_i \cdot W(x)) = \sum_i f_i(0) \partial_{i}(W(x)).$$
 It follows that elements of the form \(\partial_{-I}(W(x))\) span \(\partial_{\, \hh} W(x)\). To see that they are linearly independent, observe that \(\partial_{-I}(\partial_{-J}W(x))\) is nonzero for \(J= \overline{I}/I \) and is zero otherwise. 
 \end{proof}
 
 \begin{corollary}
 \label{cor:diffKostant}
 \(\partial_{\, \hh} W(x)= \Hom_{S_n} (\Lambda^*\hh,\CC[\hh]\cdot W(x))\).
 \end{corollary}
 
  \begin{lemma}
  \label{lem:nabla0}
  \(\nabla(\partial_{\, \hh} W(x)) = 0 \). 
   \end{lemma}

\begin{proof} We show that \(\nabla(\partial_{-I} W(x)) = 0 \) by induction on the size of \(I\). When \(I\) is empty, the claim is trivial. In general, we use the identity
\(e \neg (\alpha \wedge \beta) = (e \neg \alpha) \wedge \beta + (-1)^{|\alpha|}\alpha \wedge (e \neg \beta) \) to write 
\begin{align*}
\nabla(\partial_{-I} W) & = \alpha_1 \wedge (\alpha_{i_1} \neg (\partial_{-I'} W)) \\
& = (\alpha_{i_1} \neg \alpha_1) \wedge (\partial_{-I'} W) - \alpha_{i_1} \neg ( \alpha_1 \wedge \partial_{-I}(W)) \\
& = p_{i_1+1}(y) \cdot (\partial_{-I'} W) - \partial_{i_1}(\nabla (\partial_{-I'} W)) \\ & =
\partial_{-I'} ( p_{i_1+1}(y)\cdot W) - \partial_{i_1}(\nabla (\partial_{-I'} W))\\  & = 0.
\end{align*}
The first term in the next-to-last line vanishes because \(p_{i_1+1}(y) \cdot W\) is an antisymmetric function of degree strictly less than that of \(W\), and the second term vanishes by the induction hypothesis. 
\end{proof}

\begin{definition}
Let \(\xi_i = (-1)^{i}{\nu}^{-1}\partial_{-\overline{I}/\{i\}} W(x) \in \Hom _{S_n}(\hh^*, \lin)\).
\end{definition}

\begin{proposition}
\label{prop:xis}
Among elements of \(\Hom _{S_n}(\hh^*, \lin)\), 
\(\xi_i\) is uniquely characterized by the properties that 1) \(\partial_\alpha \xi_i = 0 \) for all polynomials \(\alpha \in \CC[\hh]\) with degree less than \(i\); and 2) \(\partial_{-i}(\xi_i) = -1\). 
\end{proposition}

\begin{proof}
It is easy to see that \(\xi_i\) satisfies the given properties. For the converse, recall that as a module over \(\CC[\hh^*]^{S_n}\), \(\Hom _{S_n}(\hh^*, \lin)\) is freely generated  by elements of \(q\)-grading \(2,4,\ldots,2(n-1)\). Observe that if \(f_i \in \CC[\hh^*]^{S_n}\), then 
$$\partial_{-i}\left (\sum_{j<i} f_j(x) \xi_j \right) = 0,$$ so \(\xi_i\) is not contained in the span of \(\xi_j\) for \(j<i\). By inducting on \(i\), we see that \(\xi_1,\ldots,\xi_{n-1}\) generate \(\Hom _{S_n}(\hh, \lin)\).
Thus we can write any other element of \(q\)-grading \(2i\) as 
$$\xi = \sum_{j\leq i}  f_j(x) \xi_j.$$
Then \(\partial_{-j} \xi =  f_j(x)\), so condition 1) implies that \(f_j = 0 \) for \(j<i\). 
\end{proof}

\begin{corollary}
\label{cor:xigen}
The set \(\{\xi_1,\ldots,\xi_{n-1}\}\) is a basis for \(\Hom_{S_n}(\hh, \CC[\hh^*])\) over 
\(\CC[\hh^*]^{S_n}\). 
\end{corollary}

\begin{definition}
Let \(u_i = \partial_1 (\xi_i) \in \eee \lin\). 
\end{definition}

To provide a characterization of the \(u_i\) analogous to that of Proposition~\ref{prop:xis}, we consider the operators \(H_k:\eee \lmn \to \eee \lmn\) given by \(H_k(x) = p_k(y) \cdot x\).
The \(H_k\) are known as the quantum Olshanetsky-Perelomov Hamiltonians (\cite{EtingofNotes},\cite{OP}). 
From the expression 
$D_i = \frac{1}{c} \frac{\partial}{\partial x_i} + \nild_i $
and the fact that \(\nild_i\) vanishes on any symmetric polynomial, it is clear
 see that
  \(H_k \to 0\)  as \(c \to \infty\). We consider a normalized version of these operators which captures their leading order behavior in \(c\).

\begin{definition}
The normalized Olshanetsky-Perelomov Hamiltonian \(P_k: \eee \lin \to \eee \lin \) is given by 
 $P_k = \lim_{c\to \infty} c \, H_k$.
\end{definition}

Expanding \(D_i^k\) to first order in \(1/c\), we see that 
$$P_k(f) = \sum_i \nild_i^k\left( \frac{\partial f}{\partial x_i} \right) = \partial_{-k+1}(\nabla f).$$

Combining the formula above with Corollary~\ref{cor:symdiff}, we see that \(P_k\) satisfies the Leibniz rule:
\begin{lemma}
If \(u, v \in \CC[\hh^*]^{S_n}\), then \(P_k(uv) = P_k(u)v + u P_k(v)\). 
\end{lemma}
\begin{proposition}
Among elements of \(\eee \lin\), \(u_i\) is uniquely characterized by the properties that
\begin{enumerate}
 \item \(P_k(u_i)=0\) for \(k\neq i+1\)
 \item \(P_{i+1}(u_i) = -(i+1)\).
\end{enumerate} 
\end{proposition}

\begin{proof}
By Lemmas~\ref{Lem:d1homotopy} and  \ref{lem:nabla0}, \(\nabla(u_i) = (i+1) \xi_i\). The fact that \(u_i\) satisfies properties 1) and 2) follows from Proposition~\ref{prop:xis}. Conversely, if \(u\) is homogenous of degree \(i+1\) and  satisfies 1) and 2), then \(\nabla u /(i+1)\) satisfies properties 1) and 2) of Proposition~\ref{prop:xis}, so \(\nabla u /(i+1) = \xi_i\). 
\end{proof}

\begin{remark}
From the discussion at the beginning of the section we see that operators could be constructed in terms of
the Poisson bracket:
$$P_k(f)=\{p_k(y),f\} \mod (p_2(y),\dots,p_n(y)).$$
In particular, the previous proposition states that coordinates $u_i$, $1\le i\le n-1$ are (up to some 
factors) canonically dual to the coordinates $p_i(y)$, $2\le i\le n$ along the locus of the Calogero-Moser 
space defined by equations $p_i(y)=0$.
\end{remark}

\begin{corollary}
\label{cor:ugen}
\(\CC[\hh^*]^{S_n}\) is generated by \(u_1,\ldots, u_{n-1}\). 
\end{corollary}

\begin{proof}
It suffices to show that \(u_i\) is not contained in the ring generated by \(u_1,\ldots,u_{i-1}\). For  \(u \in \CC[u_1,\ldots, u_{i-1}]\), we can use the Leibnitz rule to see that \(P_{i+1}(u) = 0\).  Since \(P_{i+1}(u_i) \neq 0\), the claim is proved. 
\end{proof}

Combining Corollaries~\ref{cor:xigen} and \ref{cor:ugen}, we see that 
$$\Hom_{S_n}(\Lambda^*\hh,\lin) \cong \Lambda^*(\xi_1,\ldots,\xi_n)\otimes 
\CC[u_1,\ldots, u_n].$$ 

\begin{remark}
Note that the operators \(P_k\) can be viewed as acting on \(\CC[\hhu^*]\). In this case, \(P_1\) acts nontrivially. It is easy to see that \(\CC[\hh^*] \subset  \CC[\hhu^*]^{S_n}\) is the kernel of \(P_1\), so we could equally well view the \(u_i\) as being elements of \(\CC[\hhu^*]^{S_n}\)  characterized by the condition that \(P_k(u_i) = 0\) for \(k \neq i+1\), (\(k = 1, \ldots, n\)). We have
$$\Hom_{S_n}(\Lambda^*\hhu,\lin) \cong \Lambda^*(\xi_0,\xi_1,\ldots,\xi_n)\otimes 
\CC[u_0,u_1,\ldots, u_n].$$ 
where \(u_0 = x_1+\ldots+x_n\) and \(\xi_0 = \nabla u_0 = \sum \overline{e}_i^*\). 
\end{remark}

We now turn our attention to the filtration. 
 \begin{definition}
The {\it filtration grading} \(f\) on \(\eee \lin \)  is the multiplicative grading
determined by the condition  and $f(u_i) = 1-i$. Equivalently, if \(a\) is a monomial in the \(u_i\) , then 
$$ f(a) = 2 \deg_u(a) - \frac{1}{2}q(a).$$
\end{definition}

Let \(\fF^{sph}_i \subset \eee \lin \) be the subspace generated by homogenous elements with filtration grading \(\leq i\). The \(\fF^{sph}_i\) define an increasing filtration on 
\( \eee \lin\). In the next subsection, we will show that this filtration agrees with the limits of both the algebraic and the inductive filtrations on \(\eee \lin\). 

Suppose  \(\gG^{sph} \) is a filtration on \(\eee \lin\). By Proposition~\ref{prop:twistsym}, its push-forward to the antispherical representation \(\eee_- \lin\) is given by \(m_W(\gG)_i  = W(x) \cdot \gG_{i+n(n-1)/2}\). We say that  \(\gG^{sph} \)  is {\it stable} if when we  form the filtration induced on \(\lin\) by \(\psi(\gG^{sph})\), its restriction to \(\eee \lin \) will again be \(\gG^{sph}\). 
If the limit of the inductive filtrations on \( \eee \lmn\) exists, its restriction to \(\eee \lin\) is clearly stable. 

Next, we consider the limit of the algebraic filtration. Let \(P^{alg}(\eee \lmn )\) be the 
 Hilbert polynomial of the filtration \(\fF^{alg}\) restricted to \( \eee \lmn\). 
 If \(\gG^{sph}\)  is the limit of these filtrations, the Hilbert series of  \(\gG^{sph}\) should agree with the limit of the  \(P^{alg}(\eee \lmn )\). This limit is easily computed:
 
\begin{lemma}
\label{lem:P^alg}
\(\displaystyle P^{alg}(\eee \lin) := \lim_{m\to \infty} P^{alg}(\eee \lmn) = \prod_{i=1}^{n-1} \left(1-q^{2i+2}t^{1-i}\right)^{-1}\).
\end{lemma}

\begin{proof} The filtration  
\(\fF^i \CC[\hh^*]^{S_n}  = \sum_j \left( \aa^j \cap \bigoplus_{k < 2j - i} \CC[\hh^*] ^{S_n}(k) \right)\)
of the ring of symmetric functions by powers of its maximal ideal has Hilbert series
\[ \prod_{i=1}^{n-1} \left(1-q^{2i+2}t^{i-1}\right)^{-1}.\] Thus the claim  follows  from the fact that 
\begin{enumerate}
 \item \(\eee \lmn (k) = \CC[\hh^*]^{S_n} (k)\) for all \(m>k\)
 \item  the filtration \(\fF^{alg}_i \eee \lmn(k)\) is dual to \(\fF^i \eee \lmn(k)\) under the Dunkl pairing. 
\end{enumerate}
\end{proof}

In fact, these two conditions are enough to characterize the filtration \(\gG^{sph}\). 

\begin{proposition}
\label{prop:limitfiltration}
Suppose \(\gG^{sph}\) is an increasing filtration on \(\eee \lin \) which is stable and 
whose Hilbert series is given by  \( \gG^{sph}\). Then \(\gG^{sph} = \fF^{sph} \). 
\end{proposition}

\begin{proof}
It is easy to see that the Hilbert series  of \(\fF^{sph}\) is given by 
\(P^{alg}(\eee \lin)\). Thus it is enough to show that \(\fF^{sph}_i \subset \gG^{sph}_i\). 

We claim that \(u_i\), which is an element of  \(\fF^{sph}_{1-i}\), is contained in \(\gG^{sph}_{1-i}\) as well. To see this, write  \(u_i = \partial_1(\partial_{\overline{I}/i}(W(x))\). It is clear from \(P^{alg}(\eee \lin)\) that \(1 \in \gG^{sph}_0\), so by stability of \(\gG\), \(W(x) \in \gG_{-n(n-1)/2}\). It follows that in the induced filtration, the level  of \(u_i\) is \(\leq -n(n-1)/2 + 1 + |I/i| = 1-i\). Since \(\gG\) is stable, \(u_i \in \gG_{1-i}\).

 We will show by induction on \(k\) that 
\(\fF^{sph}_i (2k) \subset \gG^{sph}_i(2k)\). When \(k=0\), this is clear. In general, if \(u \in \eee \lin\) is a monomial in the \(u_i\) with \(q(u) = 2k\) we can write \(u = u_i u'\) for some \(u_i\). By Corollary~\ref{cor:symdiff}, \(u =\partial_1(\partial_{\overline{I}/i}(W(x)u'))\). Then by the induction hypothesis, \(u' \subset \gG^{sph}_{f(u')}\), and the same argument as above shows that \(u \subset \gG^{sph}_{f(u')+1-i} = \gG^{sph}_{f(u)} \). 
\end{proof}

To describe the corresponding filtration on all of \(\lin\), we use the Kostant decomposition:
\begin{definition}
\label{Def:linFilt}
We define a grading \(f\) on \(\lin\) by declaring that if \(a=  g(x) h(y) \cdot W (x)\), where \(g \in \CC[\hh^*]^{S_n}\) is a monomial in the \(u_i\), and \(h \in \CC[\hh]_{S_n}\) is \(q\)-homogenous, then \(a\) is \(f\)-homogenous with 
$$f (a) = f(g(x)) - q(h(y) \cdot W(x))/2 = 2 \deg_u(g) - q(a)/2. $$
\end{definition}
Let \(\fF_i \subset \lin\) be the subspace generated by homogenous elements with grading \(\leq i\). Clearly \(\fF_i \cap \eee \lin = \fF_i^{sph}\).
We say that a filtration \(\gG\) of \(\lin\) is {\it stable} if the filtration induced on \(\lin\) by \(m_W(\gG^{sph})\) agrees with \(\gG\). 

\begin{proposition}
\label{Prop:F_on_lin}
Suppose \(\gG\) is an increasing filtration on \(\lin\) which is stable and whose restriction to \(\eee \lin\) agrees with \(\fF^{sph}\). Then \(\gG=\fF\). 
\end{proposition}

\begin{proof}
Let us show \(\fF_i \subset \gG_i\). 
Suppose \(a = g(x) h(y) \cdot W(x) \), where \(h(y) \in \CC[\hh]_{S_n}\) is \(q\)-homogenous of degree \(2k\) and \(g(x) \in \CC[\hh^*]\) is a monomial in the \(u_i\). Since \(\gG^{sph}\) agrees with \(\fF^{sph}\), \(g(x) W(x) \in \gG_{f(g) - n(n-1)/2}\). This implies  \(a \in  \gG_{f(g) - n(n-1)/2+k}= \gG_{f(a)}\). 

Conversely, suppose \(a \in \gG_i\). Write \(a = \sum_j g_j(x) h_j(y) \cdot W(x)\), where \(g_j\) is \(q\)-homogenous and \(h_j\) is a monomial in the \(u_i\),  and choose \(j\) for which \(f(g_j) -q(h_j(y)\cdot W)/2\) is maximal. Let \(q(h_j(y) \cdot W) = 2k\).  We can find \(\overline{h}\) \(q\)-homogenous of degree \(2k\) so that \(h_j \overline{h} \equiv W(y) \mod \CC[\hh]^{S_n}_+\). Then \(\eee \overline{h} \cdot a \in \gG_{i+k}^{sph} =  \fF_{i+k}^{sph}\)   has a nonzero coefficient of \(g_j\), so \(f(g_j) \leq i+k\). It follows that 
$$f( g_j(x) h_j(y) \cdot W(x)) = f(g_j) -q(h_j(y)\cdot W)/2 \leq i $$
so \(a \in \fF_i\). 
\end{proof}

\begin{definition}
\label{Def:HinFilt}
We define a grading \(f\) on \(\Hin\) by declaring that if \(\varphi=  g(x) \partial_{-I} W(x)\), where \(g \in \CC[\hh^*]^{S_n}\) is a monomial in the \(u_i\), then \(a\) is \(f\)-homogenous with 
$$f (\varphi) = f(g(x)) - q(\partial_{-I}W(x))/2 = 2 \deg_u(g) - q(\varphi)/2. $$
\end{definition}

From  Corollary~\ref{cor:diffKostant}, we deduce
\begin{corollary} 
\label{Cor:HinFilt}
The filtration induced by  \(\fF \lin \) on \(\Hin\)  is the same as the increasing filtration defined using \(f\). 
\end{corollary}
\subsection{The limiting filtration} 

Our goal in this section is to prove the following

\begin{theorem}
\label{thm:filtrationlimit}
The limits \(\lim_{m \to \infty} \fF^{alg} \lmn \) and \(\lim_{m \to \infty} \fF^{ind} \lmn \) both exist  and are equal to \(\fF \lin \). 
\end{theorem}

In fact, we will prove slightly more:

\begin{proposition}
\label{Prop:Stableind=Alg}
For any fixed \(k\), there exists some \(M\) so that \(\fF^{ind}_i \lmn (k)= \fF^{alg}_i \lmn(k)\) whenever \(m \geq M\). 
\end{proposition}

Note that in both cases, it suffices to prove the result under the additional hypothesis that \(m \equiv r \mod n\) for an arbitrary value of \(r\). Our starting point is the following 

\begin{lemma}
\label{lem:alglimit}
\(\lim_{m \to \infty} \fF^{alg}\lmn \) exists. Its restriction to \(\eee \lin\) has Hilbert series \(P^{alg}(\eee \lin)\). 
\end{lemma}

\begin{proof}
Since \(\fF^{alg}\) is compatible with the \(q\)-grading, it suffices to show that \(\lim_{m \to \infty} \fF^{alg}(k)\) exists for any fixed value of \(k\). Taking \(m>k\), we may identify 
\(\lmn(k) = \lin (k) = \CC[\hh^*](k)\), so the filtrations in question all have the same underlying space \(V_k\). Moreover, the filtration \(\fF^d(k)\) defined by powers of the maximal ideal is independent of \(m\). 

For each value of \(c>k/n\), the Dunkl pairing defines a nonsingular pairing on \(V_k\). Alternately, we may think of  it as defining a nonsingular pairing on \(V_k \otimes \CC(c)\), which we view as a vector space over the ring of rational functions in \(c\). Let \(\{v_i\}\) be a basis for \(V_k\). Then by row-reduction, we can find a basis 
$ \alpha_j= \sum a_{ij} v_i $ over \(\CC(c)\) for \((\fF^d(k) \otimes \CC(c))^\perp \).  
This defines a rational map from \(\PP^1\) to the Grassmanian of \(r\)-planes in \(V_k\), where \(r= \dim \fF^d(k) \). Any rational map of \(\PP^1\) to a projective variety extends to all of \(\PP^1\), so the limit as \(c\to \infty\)  exists. 
 The last statement follows immediately from Lemma~\ref{lem:P^alg}.
\end{proof}

We are now ready to start the proof of Theorem~\ref{thm:filtrationlimit}. As usual, we begin with the case of the spherical representation. 

\begin{proposition}
\label{Prop:SphericalLimit}
For each \(k\) and \(r\), the following statements hold:
\begin{enumerate}
\item There exists  \(M\) such that for all \(m >M\) such that  \(m \equiv r \mod n\), 
$\fF^{ind} \eee \lmn (\leq k) = \fF^{alg} \eee \lmn (\leq k)$. 
\item \(\lim_{m \to \infty} \fF^{alg} \eee \lmn (\leq k) = \fF^{sph} (\leq k) \). 
\end{enumerate}
\end{proposition}

\begin{proof}
By induction on \(k\). When  \(k=0\), the statement follows immediately from Proposition~\ref{prop:indvand}. Now suppose the statement holds for \(k-1\). We have already shown in Theorem~\ref{thm:IndsubAleg} that $\fF^{ind}_i \eee \lmn \subset \fF^{alg}_i \eee \lmn$. Thus to prove statement (1), it suffices to show that \(\dim \fF^{ind}_i \eee \lmn (k) =  \dim \fF^{alg}_i \eee \lmn (k)\) when \(m\) is large. 

Recall that the filtration on  \(\eee \lmn\) is induced by the filtration on \(\eee \lmnprev\) via the isomorphism \(m_W:\eee \lmnprev \to \eee_- \lmn\) and the relation 
$$ \beta \in \fF^{m/n}_i \ \text{if and only if} \ \beta  = a \phi(\alpha) \ \text{where} \ a \in \gG_j, \alpha \in 
\fF^{(m-n)/n}_k,\ \text{and} \  i+j = k.$$

We say that \(a \in  \eee \lmn\) is {\it reachable at filtration level} \(i\) from \( \eee \lmnprev(\leq k)\) if we can find \(b \in \fF^{ind}_{i_1} \eee \lmnprev\) and \(\beta \in \Hh_{\frac{m}{n}}\) of filtration level \(i_2\) with \(i_1+ i_2 - n(n-1)/2  \leq i\) and \(a = \beta \cdot m_W(b)\). Let \(\fF^{\leq k}_i \eee \lmn\) be the set of elements which are reachable at level \(i\) from  \( \eee \lmnprev(\leq k)\). By construction, 
\(\fF^{\leq k}_i \eee \lmn \subset \fF^{ind}_i \eee \lmn\). 

The proof of Proposition~\ref{prop:limitfiltration} amounts to showing that the dimension of the subspace of \(\eee \lin (k)\) which is reachable at level \(i\) from \(\eee \lin (\leq k-1)\) is equal to the dimension of \(\fF^{alg}_i \eee \lmn (k)\) for \(m\) large. Now by the induction hypothesis, we know that as \(m \to \infty\), \(\fF^{ind}_i \eee \lmn (\leq k-1) \to \fF_i \eee \lin(\leq k-1) \). Moreover, as \(m \to \infty\), the action of \(\Hh_{\frac{m}{n}}\) on \(\CC[\hh^*]\) tends to the action of 
\(\Hh_{\frac{\infty}{n}}\) on  \(\CC[\hh^*]\). Since the property that that the image of a subspace under a linear map has dimension at least \(j\) is an open condition, it follows that when \(m\) is large, 
$$\dim \fF^{ind}_i \eee \lmn (k) \geq \dim \fF^{\leq k-1}_i \eee \lmn \geq 
 \dim \fF^{\leq k-1}_i \eee \lin =  \dim \fF^{alg}_i \eee \lmn (k). $$
 We conclude that statement (1) holds for \(k\). 
 
 To show that statement (2) holds, let \(\gG^{sph}(\leq k) = \lim_{m\to \infty} \fF^{alg} \eee \lmn (\leq k) \). (The limit exists by Lemma~\ref{lem:alglimit}.) By statement (1) \(\gG^{sph}\) is also the limit of \(\fF^{ind} \eee \lmn\), so it is stable. Finally, by proposition~\ref{prop:limitfiltration}, we conclude that  \(\gG^{sph}(\leq k) = \fF^{sph} (\leq k)\). 
\end{proof}

We can now prove the analogous results for all of \(\lmn\). 

\begin{proof} (of Proposition~\ref{Prop:Stableind=Alg})
As in the proof of the last proposition, it suffices to show that \(\dim \fF^{ind}_i \lmn (k) \geq \dim \fF^{alg}_i \lmn(k)\) for large \(m\).  
To compute the right-hand side of this inequality, we use the isomorphism
\(\lin \cong \CC[\hh]_{S_n} \otimes \CC[\hh^*]^{S_n}\) provided by the Kostant decomposition. It is easy to see that 
under this isomorphism, \({\aa^k}\) maps to \( \CC[\hh]_{S_n} \otimes (\CC[\hh^*]^{S_n}_+)^k\).  Thus the Hilbert series of the decreasing filtration \(\fF^i \lin\) is given by 
\begin{equation*}
 Q(q^2t) P^{\aa^k}(\eee \lmn) =\prod_{i=2}^{n} \frac{1-q^{2i}t^{i}}{(1-q^2t)(1-q^{2i}t^{i-2})}
\end{equation*}
where \(Q(q^2) = [n!]_{q^2}\) is the Hilbert polynomial of 
\(  \CC[\hh]_{S_n}\). The Hilbert polynomial of the dual filtration \(\fF^{alg} \lin\) is obtained by replacing \(t\) with \(t^{-1}\) in this series. It is clear from Definition~\ref{Def:linFilt} that this is also the Hilbert polynomial of the filtration \(\fF\lin \). In summary, we have seen that when \(m\) is large relative to \(k\),  \(\dim \fF_i^{alg} \lmn  (k) = \dim \fF_i \lin  (k)\). 

The proof of Proposition~\ref{Prop:F_on_lin} shows that if \(a \in \fF_i\lin (k)\), then \(a\) is reachable at filtration level \(i\) from \(\eee \lin (\leq k+n(n-1)/2)\). Since we already know that \(\fF^{ind} \eee \lmn (\leq k) \to \fF \eee \lin (\leq k) \) as \(m \to \infty\), the same argument as in the proof of Proposition~\ref{Prop:SphericalLimit} shows that when \(m\) is large,
 $$\dim \fF_i^{ind} \lmn (k) \geq \dim \fF_i \lin (k) = \dim \fF^{alg}_i \lmn (k) $$
 which is what we wanted to prove. 
\end{proof}

\begin{proof} (of Theorem~\ref{thm:filtrationlimit})
By Lemma~\ref{lem:alglimit}, \(\gG:= \lim_{m\to \infty} \fF^{alg} \lmn\) exists. Proposition~\ref{Prop:Stableind=Alg} implies that \(\gG= \lim_{m\to \infty} \fF^{ind} \lmn\), so \(\gG\) is stable. Then by Proposition~\ref{Prop:F_on_lin}, \(\gG= \fF\). 
\end{proof}

If \(I = \{i_1,\ldots, i_k\} \subset \overline{I}\), we write 
\(\xi_I = \xi_{i_1}\wedge\ldots \xi_{i_k}\). 

\begin{proposition}
The set \(\{ \xi_I \, | \, I \subset \overline{I}\}\) is a basis for \(\partial_{\hh} W(x)\). 
\end{proposition}

\begin{proof} The set of \(\xi_I\)'s has the correct cardinality, and they are linearly independent, since \(\xi_I \wedge \xi_J = 0 \) unless \(J = \overline{I}/I\), and is a nonzero multiple of \(W\) in that case. Thus it suffices to prove that \(\xi_I\) is contained in \(\partial_{\hh} W(x)\). 
From Corollary~\ref{Cor:HinFilt}, we see that \(\partial_\hh W(x) (2k) = \fF_{-k} \Hin(2k)\). We prove by induction on \(|I|\) that \(\xi_I \in \fF_{-k} (2k) \Hin\), where \(2k=q(\xi_I)\) . When \(|I|=1\), this is clear. In general, write 
\(\xi_I\) = \(\xi_{i_1} \wedge \xi_{I'}\). Then by the induction hypothesis \(\xi_{I'}\in \fF_{-k'} (2k')\), where \(q(\xi_{I'}) = 2k\). Then \(u_{i_1} \xi_{I'} \subset \fF_{-k'+1-i_1}\). We proved in lemma~\ref{Lem:nabla filt} that \(\nabla(\fF_i^{alg} \Hmn) \subset \nabla(\fF_i^{alg} \Hmn)\). Since \(\fF_i \Hin\) is the limit of the \(\fF_i^{alg}\),  \(\nabla(\fF_i^{alg} \Hin) \subset \nabla(\fF_i^{alg} \Hin)\). Thus
\(\nabla(u_{i_1} \xi_{I'}) = \xi_{i_1} \wedge \xi_{I'} = \xi_I\) is in \( \fF_{-k'-i_1}= \fF_{-k}\). 
\end{proof}

\begin{remark}
It seems natural to conjecture that \(\xi_I\) is a scalar multiple of \(\partial_{-\overline{I}/I}.\) 
\end{remark}

It follows that the filtration \(\fF \Hin\) is induced by a multiplicative grading \(f\) on \(\Hin \cong \Lambda^*(\xi_1,\ldots, \xi_{n-1}) \otimes \CC[u_1,\ldots, u_{n-1}]\), where 
\(f(u_i) = 1-i\) and \(f(\xi_i)=-i\).

\subsection{Determining \(u_k\)} 

Recall that \(u_k \in \eee \lin\) is characterized by the fact that 
\(P_i(u_k) =0\) for \(i\neq k+1\) and \(P_{k+1}(u_k) = -(k+1)\), where \(P_k: \lin \to \lin \) is the normalized Olshanetsky-Perelomov Hamiltonian. The symmetric polynomials \(u_k\) can be written explicitly as a sum of elementary symmetric polynomials.
If  \(\lambda\) is a partition of \(k\), and \(\lambda'\) is its dual partition, we define
$$F_{n}(\lambda) = \frac{1}{n^{\lambda_1'}(n-1)^{\lambda_2'}\ldots (n-k+1)^{\lambda_k'}}=\prod_{i}\frac{(n-\lambda_i)!}{n!}.$$
Moreover, we define constants \(a_k(\lambda)\) to be the coefficients of the expansion of the  \(k\)th symmetric power function in terms of elementary symmetric functions:
$$ p_k = \sum_{\lambda \vdash k} a_k(\lambda) e_\lambda.$$
Then we have 
\begin{proposition}
\label{Prop:uInTermsOfe}
\( \displaystyle
u_k =  \sum_{\lambda \vdash (k+1)} a_{k+1}(\lambda) F_{n} (\lambda) e_\lambda. \)
\end{proposition}

\begin{proof}
The action of \(H_k\) on \(e_l \in \eee \lmn\) was calculated in Theorem 5.10 of \cite{G1}. It is given by
$$ H_k(e_l) = -\left(-\frac{m}{n}\right)^{k-1}(n+1-l)\dots(n+k-l)e_{l-k}.$$
Passing to the limit \(m \to \infty\), we see that the action of the rescaled Hamiltonian \(P_k\) on \(e_l \in \eee \lin\)  is given by 
$$P_k(e_l) = (-1)^{k} (n+1-l)\dots(n+k-l)e_{l-k}.$$
Therefore if $$\overline{e}_l=(-1)^{l}e_{l}\frac{(n-l)!}{n!},$$ then
\begin{equation}
\label{pk on rescaled el}
P_{k}(\overline{e}_l)=\overline{e}_{l-k}.
\end{equation}
Consider the generating function
$$\overline{E}(z)=\sum_{k=0}^{n}\overline{e}_{k}z^{k}=\sum_{k=0}^{n}(-1)^{k}\frac{(n-k)!}{n!}e_{k}z^{k},$$
then (\ref{pk on rescaled el}) implies $P_k(\overline{E}(z))=z^{k}\overline{E}(z).$
Since \(P_k\) obeys the Leibniz rule, we get
\begin{equation}
\label{pk on ln ebar}
P_{k}(\ln \overline{E}(z))=\frac{P_{k}(\overline{E}(z))}{\overline{E}(z)}=\frac{z^{k}\overline{E}(z)}{\overline{E}(z)}=z^{k}.
\end{equation}
On the other hand, it is well known (e.g. \cite{Macdonald}) that
$$\sum_{k=1}^{\infty}\sum_{\lambda \vdash k}\frac{z^{k}}{k}a_{k}(\lambda)e_{\lambda}=\sum_{k=1}^{\infty}\frac{z^{k}}{k}p_{k}=-\ln\left(\sum_{k=0}^{n} (-1)^{k}e_{k}z_{k}\right),$$
therefore 

\begin{equation}
\label{u as ln ebar}
U(z)=\sum_{k=0}^{\infty}u_{k}\frac{z^{k+1}}{k+1}=\sum_{k=1}^{\infty}\sum_{\lambda \vdash k}\frac{z^{k}}{k}a_{k}(\lambda)\prod_{i}e_{\lambda_i}\frac{(n-\lambda_i)!}{n!}=
-\ln \overline{E}(z).
\end{equation}

From (\ref{pk on ln ebar}) and (\ref{u as ln ebar}) we conclude $P_{k}(u_l)=-(k+1)\cdot \delta_{k}^{l}$.
\end{proof}

 \begin{corollary}
 \(\displaystyle \lim_{n\to \infty} n^{k-1} u_k(n) = p_k\). 
 \end{corollary}

\subsection{Computing \(d_N\)}
\label{subsec:diffCalc}
In this section, we determine the action of \(d_\alpha\) on \(\Hin^\fF\), where 
\(\alpha \in \Hom_{S_n}(\hh,\CC[\hh^*]) \subset \Hom_{S_n}(\hh, \Hh_{\frac{\infty}{n}})\). 
Before we go on, we pause to consider the dependence of \(d_{\alpha}\) on the choice of \(\alpha\). Suppose that \(\alpha \in \Hom_{S_n}(\hh,\CC[\hh^*])\) is \(q\)-homogenous of degree \(N\). The set of all such polynomials has a basis consisting of monomials of the form \(u_\lambda \xi_i\), where \(|\lambda|+i = N\). It is immediate from the definition of \(\partial_\alpha\) that 
  \(\partial_{u _\lambda \xi_i} = u_\lambda d_{\xi_i}\). Multiplication by \(u_{\lambda}\) shifts the filtration level by \(2 l(\lambda) -l\lambda|\), so the net shift in filtration level is 
  \(i + 2l(\lambda)-|\lambda| = N-2(|\lambda|-l(\lambda))\). Unless  \(u_\lambda \) is a power of \(u_0\), this is strictly less than \(N\), so it will not contribute to the map on associated graded groups. 
  Thus when we work with reduced homology, there is an essentially unique choice of \(d_N\) up to scale. 
  
  We now return to the calculation. Our first step is to observe that \(d_{\alpha}\) satisfies a Leibniz rule:
  \begin{lemma}
  Consider \(d_\alpha: \Hin^\fF \to \Hin^\fF\), where \(\alpha \in  \Hom_{S_n}(\hh,\CC[\hh^*])\). We have
  $$d_{\alpha} \varphi \wedge \psi = (d_\alpha \varphi) \wedge \psi + (-1)^{|\varphi|} \varphi \wedge (d_\alpha \psi).$$ 
  \end{lemma}
  
  \begin{proof}
  Since \(\alpha \in  \Hom_{S_n}(\hh,\CC[\hh^*])\), we have 
  $$\alpha \neg (\varphi \wedge \psi) = (\alpha \neg \varphi) \wedge \psi + (-1)^{|\varphi|} \varphi \wedge (\alpha \neg \psi)$$
  from which it follows that 
   $$\partial_{\alpha} \varphi \wedge \psi = (\partial_\alpha \varphi) \wedge \psi + (-1)^{|\varphi|} \varphi \wedge (\partial_\alpha \psi).$$ 
   Since the filtration \(\fF\) is induced by a multiplicative grading, the analogous statement for associated gradeds also holds. 
  \end{proof}
 Thus it suffices to compute \(d_{\alpha} \xi_i\). 
For the moment, we work with unreduced homology. Suppose that \(\alpha, \beta \in \Hom_{S_n}(\hhu,\CC[\hhu^*])\) are given by \(x_i \mapsto \alpha_i\), \(x_i \mapsto \beta_i\). We define \(\alpha\cdot \beta \in \CC[\hhu^*]^{S_n} \) by 
$ \alpha \cdot \beta = \sum_i \alpha_i \beta_i$ and 
   \(\alpha*\beta \in \CC[\hhu^*]\) by 
   $ x_i \to \alpha_i \beta_i $
   respectively. We clearly have
   \begin{lemma}
   \label{Lem:DotStar}
   $ \alpha \cdot (\beta * \gamma) = (\alpha * \beta) \cdot \gamma.$
   \end{lemma}
   
 If  \(\alpha_j \in \Hom_{S_n}{\hhu, \CC[\hhu ^*]} \) is  given by 
   \(x_i \mapsto x_i^j\), then we can compute

 \begin{proposition}
 \label{prop:a1*}
 \(\alpha_1 * \xi_k = \sum_{i=0}^k u_i \xi_{k-i} + (n-k-1)\xi_{k+1}\).
 \end{proposition}
 
 \begin{proof}
 Since $\frac{\partial e_k}{\partial x_i}=e_{i-1}(x_1,\ldots,\widehat{x_i},\ldots, x_n)$, one has
$$x_i\frac{\partial e_k}{\partial x_i}=e_{k}-\frac{\partial e_{k+1}}{\partial x_i},$$
therefore
$$x_i\frac{\partial \overline{e}_k}{\partial x_i}=\overline{e}_{k}+(n-k)\frac{\partial \overline{e}_{k+1}}{\partial x_i},$$
and
\begin{equation}
\label{DE for E}
zx_i\frac{\partial \overline{E}(z)}{\partial x_i}=z\overline{E}(z)+(n+1)\frac{\partial \overline{E}(z)}{\partial x_i}-\frac{\partial^2 \overline{E}(z)}{\partial z\partial x_i}.
\end{equation}
By (\ref{u as ln ebar}) $\overline{E}(z)=\exp( -U(z))$, so we can rewrite (\ref{DE for E}) as
$$zx_i\frac{\partial U(z)}{\partial x_i}=-z+(n+1)\frac{\partial U(z)}{\partial x_i}-\frac{\partial^2 U(z)}{\partial z\partial x_i}+\frac{\partial U(z)}{\partial x_i}\cdot \frac{\partial U(z)}{\partial z}.$$
By expanding in $z$, we get the equation
$$\frac{x_{i}}{k+1}\frac{\partial u_k}{\partial x_i}=-\delta_{0}^{k}+\frac{(n-k-1)}{k+2}\frac{\partial u_{k+1}}{\partial x_i}+\sum_{j=0}^{k}\frac{u_{j}}{k+1-j}\cdot \frac{\partial u_{k-j}}{\partial x_i},$$
thus we can use the equation $\nabla u_{i}=(i+1)\xi_{i}$:
$$x_i(\xi_{k})_{i}=-\delta_{0}^{k}+(n-k-1)(\xi_{k+1})_{i}+\sum_{j=0}^{k}u_{j}\cdot (\xi_{k-j})_{i}.$$
 \end{proof}
 
We can now prove the main result of this section.

\begin{theorem}
\(\displaystyle d_{\alpha_N}(\xi_k) = \sum_{ j_1+\ldots+j_N=k } u_{j_1}\ldots u_{j_N},\) where the sum runs over \(j_i \geq 0\). 
\end{theorem}

\begin{proof}
We have 
$$ d_{\alpha_N}(\xi_k) = [\alpha_{N} \cdot \xi_k],$$
where \([\cdot]\) denotes the equivalence class in the associated graded group. 
By Lemma~\ref{Lem:DotStar}
$$\alpha_{N} \cdot \xi_k = ( \alpha_1*\alpha_{N-1}) \cdot \xi_k = \alpha_1\cdot (\alpha_{N-1} * \xi_k).$$
By repeatedly applying Proposition~\ref{prop:a1*}, we find that
$$\alpha_{N-1} * \xi_k  = \sum_{j_1+\ldots+j_N=k} u_{j_1}\ldots u_{j_{N-1}} \xi_{j_N} + \ldots$$
where the terms which are omitted all have lower filtration grading. To finish the proof, we need only recall that \(\alpha_1 \cdot \xi_j = \partial_1 \xi_j = u_j\). 
\end{proof}

\begin{corollary}
\(\displaystyle d_{\widetilde{\alpha}_N}(\xi_k) = \sum_{ j_1+\ldots+j_N=k } u_{j_1}\ldots u_{j_N},\) where the sum runs over \(j_i \geq 1\). 
\end{corollary}

\begin{proof}
This is an immediate consequence of the remarks following Proposition~\ref{prop:anticommute}.
\end{proof}

\section{The geometric filtration} 
\label{sec:hilb}

Given a point $p$ on a planar curve $C$, 
let $C^{[d]}_m$ be the space of colength $d$ ideals which require $m$ generators
at $p$.  Let $\wtp_t(X)$ denote the weight polynomial (aka the virtual Poincar\'e polynomial). 
Consider the following series:

\begin{equation} 
Z(C) = \sum_{d, m} q^{2d} \wtp_t(C^{[d]}_m) \sum_{k=0}^m \binom{m}{k}_{t^2} 
a^{2k} t^{k^2} 
\end{equation}

In \cite{ORS} we proposed the following

\begin{conjecture} \label{conj:ORS}
Let $C$ be a locally planar curve, smooth away from a single point $p$.  Then
the unreduced HOMFLY homology is given by:
$$\overline{\pP}_{\mathrm{link}(C,p)} = 
(aq^{-1})^{\mu - 1} (1-q^2)^{1-b} (1-q^2 t^2) (1+q^2 t)^{-2\tilde{g}} Z(C)$$ where
$\tilde{g}$ is the geometric genus of $C$,  $b$ is the number of analytic local branches 
at $p$, and $\mu$ is the Milnor invariant of the singularity.  
\end{conjecture}

In this section we use
Springer theory to relate the above conjecture to Conjecture \ref{conj:one}.  

\subsection{From Hilbert schemes to Hitchin fibres}

For this subsection we fix $\gg = \mathfrak{gl}_n$. 
Let $\bB$ be the space of complete flags, and $\widetilde{\gg} \subset 
\gg \times \bB$ be the locus where the matrix preserves the flag.  The map 
$s: \widetilde{\gg} \to \gg$ is called the Grothendieck simultaneous resolution,
and its restriction to the nilpotent matrices, $s|_{\nN}:\widetilde{\nN} \to \nN$ 
is called the {\em Springer resolution}. 
As Lusztig
observed \cite{Lusztig-greenpoly}, $s: \widetilde{\gg} \to \gg$ is small, so 
$Rs_* \CC_{\widetilde{\gg}}$ is the IC extension of its
restriction to the open dense locus $\gg^{rs}$ of regular semisimple elements.  
The fibre over such elements is naturally a $S_n$-torsor, and so 
the local system $Rs_* \CC_{\widetilde{\gg}}|_{\gg^{rs}}$ carries an $S_n$ action.
Then functoriality of IC gives an action of $S_n$ on the 
cohomology of every fibre.  These are the 
{\em Springer representations},
which have many other equivalent constructions \cite{Springer-1, Springer-2, 
Slodowy, Kazhdan-Lusztig-springtop,
Borho-MacPherson, Hotta}.

The flag manifold $\bB$ carries tautological bundles
$L_i$ for $i = 1, \ldots, n$; the fibre of $L_i$ at a given flag $F$ being
$F_i/F_{i-1}$.  The cohomology $\ecH^*(\bB)$ is
the quotient of the algebra generated by the Chern classes $x_i = c_1(L_i)$ by the ideal
generated by the elementary symmetric functions in the $x_i$; 
the action of $S_n$ is by permuting the $c_i$.  The fibres $\bB_\gamma:=s^{-1}(\gamma)$ 
are similarly well understood; we restrict attention to the case of nilpotent
$\gamma$.  As for $\bB$, the $\bB_\gamma$ are stratified by linear spaces 
\cite{Spaltenstein}. 
The inclusion
$\bB_\gamma \to \bB$ induces $\ecH^*(\bB) \to \ecH^*(\bB_\gamma)$; this map is known
to be surjective, and moreover can be described as the quotient by a certain 
explicit ideal of functions in the $x_i$ {\em depending only on the conjugacy class
of $\gamma$} \cite{deConcini-Procesi}.  
In particular if we stratify 
$\nN$ by conjugacy classes, the homology sheaves of $R\pi_* \CC_{\widetilde{\nN}}$ are 
{\em trivial} local systems on each stratum.  

To a nilpotent matrix $\gamma$, 
assign the partition $j(\gamma) \vdash n$
whose entries are the sizes of the Jordan blocks; e.g., 
$j(\mathbf{0})=(1^n)$.  Since all $\ecH^*(\bB_\gamma)$ with $j(\gamma) = \pi$
are canonically isomorphic, we will sometimes denote this space as
$\ecH^*(\bB_\pi)$. 
Our convention for partitions is as in
\cite{Macdonald}: a partition $\mu$ has parts 
$\mu_1 \ge \mu_2 \ge \cdots$; we write $\ell(\mu)$ for the number of parts,
$n(\mu):=\sum (i-1) \mu_i$, and $\mu'$ for the transposed partition. 
Recall that irreducible representations of $S_n$ are indexed by partitions of $n$; 
in our convention $V_{(n)}$ is the trivial representation, 
and more generally $V_{(1^k, n-k)}$
is the k-th exterior power of the standard representation. 
Writing $S_\pi := S_{\pi_1} \times \cdots \times S_{\pi_k}$, 
it was shown (apparently first in unpublished work of Macdonald) that 
$\ecH^*(\bB_\gamma) = \mathrm{Ind}_{S_{j(\gamma)}}^{S_n} 1$.  The 
decomposition into irreducible representations of $S_n$ is given by the Kostka numbers
\cite[Cor. 4.39]{Fulton-Harris}: 
$K_{\lambda \mu} = \dim \Hom_{S_n}(V_\lambda, \mathrm{Ind}_{S_\mu}^{S_n})$. 
The $K_{\lambda \mu}$ are combinatorially the number of semistandard 
tableaux on $\lambda$ with $\mu_k$ k's.  It is easy to see 
$K_{(1^k, n-k), \mu} = \binom{\ell(\mu) - 1}{k}$, where 
$\ell(\mu)$ is the number of parts of $\mu$.  

The Kostka numbers admit a refinement $\widetilde{K}_{\lambda\mu}(x)$ such that
$\widetilde{K}_{\lambda\mu}(1) = K_{\lambda\mu}$.  These have many interpretations,
see \cite{Macdonald} and the references thereof; in particular
\cite{Lusztig-greenpoly, Hotta-Springer}:
\[\widetilde{K}_{\lambda \mu}(t^2) = \sum t^i \dim 
\Hom_{S_n}(V_\lambda, \ecH^i(\bB_\mu))\]  
The $\widetilde{K}_{\lambda\mu}(x)$
also admit a combinatorial expression 
$\widetilde{K}_{\lambda \mu}(x) = \sum_{T \in SST(\lambda, \mu)} x^{n(\mu)-c(T)}$,
where $SST(\lambda, \mu)$ are the same tableaux counted by $K_{\lambda \mu}$ 
and $c(T)$ is a certain complicated statistic called the {\em charge} 
\cite{Lascoux-Schutzenberger} (or see \cite[III.6]{Macdonald}).

\begin{lemma} \label{lem:covariants} Let $\nN^{(m)}$ be the locus of nilpotent matrices 
  with kernel of dimension $m$.  Let $\hh$ be the standard 
  representation of the symmetric group.  Then 
  $\Hom_{S_n}(\Lambda^k \hh, Rs_* \CC|_{\nN^{(m)}})$ is
  a direct sum of constant sheaves.  The fibre is pure with Poincar\'e series 
  $t^{k(k+1)} \binom{m - 1}{k}_{t^2}$.
\end{lemma}
\begin{proof}
  It is known \cite{Spaltenstein} that all the Springer fibers $\bB_\gamma$
  are paved by linear spaces, from which the purity of the weight
  filtration follows.  Moreover \cite{deConcini-Procesi} the cohomology of the fibers
  is given explicitly as an $S_n$-equivariant quotient of $\ecH^*(\bB)$ by an ideal 
  which (1) depends only on the partition and (2) increases as the partition
  increases in the dominance order.  In particular all $\ecH^*(\bB_\gamma)$ for 
  $\gamma \in \nN^{(m)}$ have a canonical surjective $S_n$-equivariant map onto 
  $H^*(\bB_{(n-m,1^{m})})$.  Hence 
  $\Hom_{S_n}(\Lambda^k \hh, Rs_* \CC|_{\nN^{(m)}})$ maps surjectively to 
  the constant local system with fibre 
  $\Hom_{S_n}(\Lambda^k \hh, \ecH^*(\bB_{(n-m,1^{m})}))$. 
  As all fibres have cohomology with the same total dimension 
  $K_{(n-k, 1^k), \mu} = \binom{m - 1}{k}$, 
  the map is an isomorphism.
  
  We are now reduced to computing one fibre, so it will suffice to show 
  \[ 
    \widetilde{K}_{(n-k, 1^k), (n-m,1^m)}(t^2) = t^{k(k+1)}\binom{m - 1}{k}_{t^2}
  \]
  By \cite{Lascoux-Schutzenberger}, this is a straightforward combinatorial exercise;
  or see \cite[p. 362, ex. 2]{Macdonald}.
\end{proof}

\begin{remark}
 It follows that 
  \[     
  \widetilde{K}_{(n-k, 1^k), \mu}(t^2) = t^{k(k+1)}\binom{\ell(\mu) - 1}{k}_{t^2}
  \]
 this is also easy to see directly from the combinatorial description of the 
  $\widetilde{K}(x)$. 
\end{remark}

We will need the slightly stronger statement:

\begin{lemma} \label{lem:covariants-stacky}
  Let $s/G:[\widetilde{\gg}/G] \to [\gg/G]$ be the (stacky) adjoint quotient of
  the simultaneous resolution.  Then 
  $\Hom_{S_n}(\Lambda^k \hh, R(s/G)_* \CC|_{\nN^{(m)}/G})$ is
  a direct sum of constant sheaves.  The fibre is pure with Poincar\'e series 
  $t^{k(k+1)} \binom{m - 1}{k}_{t^2}$.
\end{lemma}
\begin{proof}
  By definition, we should show that the result of Lemma \ref{lem:covariants}
  holds when $\gg$, $\nN$, $\bB$, etc., are replaced by torsors arising
  from some not-necessarily-trivial rank $n$ bundle.  But note 
  that although the bundle of flag varieties is not trivial, its cohomology
  forms a trivial local system because it is generated by Chern classes
  of tautological bundles.  Now the argument of Lemma \ref{lem:covariants}
  applies. 
\end{proof}

We recall basic properties of the Hitchin fibration \cite{Hitchin, BNR, Ngo-endoscopy}; for
a good introduction see \cite{DM}. 
Fix a smooth base curve $X$, and a line bundle $L$ of degree  
at least $2g(X)$.  Then moduli space of Higgs bundles $\mM_X = \mM_{X,L,n}$ parameterizes
pairs $(E, \phi:E \to E \otimes L)$ where $E$ is a rank $n$ vector bundle over $X$.  We write
$\mM_{X,\ell}$ for the component where $E$ has some fixed degree $\ell$.  
Let $\aA_X = \aA_{X,L,n}$ be the vector space parameterizing curves in the total space
of $L$ projecting with degree $n$ to $X$; we write $[C] \in \aA$ to denote
the point corresponding to a curve $C$.  
Then the {\em Hitchin fibration} is 
the map $h: \mM_X \to \aA_X$ which takes 
$(E,\phi)$ to the ``spectral'' 
curve cut out by the characteristic polynomial of $\phi$.  We restrict attention
to the locus $\aA_X^\heartsuit \subset \aA_X$ of integral spectral curves.  Over this
locus all Higgs bundles must be simple, since a sub-bundle would give rise to a sub-spectral
curve.  The restriction $\mM_\ell|_{\aA_X^{\heartsuit}} \to \aA_X^{\heartsuit}$
is a proper flat map between smooth varieties.  Henceforth we omit the $\heartsuit$.

Let $[C] \in \aA_X$ be a spectral curve.  The fibre $\mM_{X,\ell,[C]}:=h^{-1}([C])$  can be
identified with the moduli space $\overline{J}^{\ell'}(C)$ of 
torsion free rank one sheaves of degree $\ell' = \ell + \binom{n}{2} \deg L $ on $C$:  
such a sheaf pushes forward
to a rank $n$ bundle on $X$, which comes equipped with a Higgs field induced
from the tautological section $C \to L|_C$.

Fixing a point $x \in X$  yields an evaluation map $ev:\mM_{X,\ell} \to [\gg/G]$ 
which takes the conjugacy class of $\phi_x$.  The {\em parabolic} Hitchin system
is given by
$\para{\mM}_{X} = \mM_{X} \times_{[\gg/G]} [\widetilde{\gg}/G]$.
In other words, $\para{\mM}_{X}$ parameterizes 
triples 
$(E, \phi: E \to E \otimes L, \mbox{a complete flag in } E_x \mbox{ preserved by } 
\phi_x)$.   Evidently the fibre product is compatible with the map to $\aA_X$. 
We denote the projection
also as $s:\para{\mM}_X \to \mM_X$.  By pullback we obtain the action 
of $S_n$ on $Rs_* \CC_{\para{\mM}_{X}}$.  We 
write $\nN_{X}$ for the preimage under $ev$ of the nilpotent cone $\nN$, and 
$\para{\nN}_{X}$ for the preimage of $\para{\nN}$ under the projection to the second factor. 

Let $\mM^{(m)}_{X} \subset \mM_{X}$ be the locus where the 
$\dim E_x / \phi_x E_x = m$, and similarly 
$\nN^{(m)}_{X}$.  Consider some $(E,\phi) \in \nN^{(m)}_{X}$; note in particular 
the spectral curve $C_a$ has a single point over $x$ which moreover is on the zero section 
of $L$.  Let $z$ be a local coordinate on $X$ near $x$, and $y$ a vertical coordinate 
along the line bundle, and $f(z,y)$ the 
defining equation of the spectral curve.   If $F$ is
the torsion free sheaf on $C$ such that $\pi_* F = E$, then 
\[\dim F/(z,y)F =\dim (F/zF)/y(F/zF) = \dim E_x / \phi_x E_x = m\]

\begin{lemma}  \label{lem:covariants-hitchin}
  $\Hom_{S_n}(\Lambda^k \hh, Rs_* \CC|_{\nN^{(m)}_{{X}}})$ is
  a trivial local system whose fibre is pure with Poincar\'e series 
  $t^{k(k+1)} \binom{m - 1}{k}_{t^2}$.
\end{lemma}
\begin{proof}
 By pullback from Lemma \ref{lem:covariants-stacky}.
\end{proof}

Let $\hH_{X,\ell}$ denote the space of triples $(E, \phi, \sigma)$ where $(E, \phi) \in \mM_{X,\ell}$,
and $\sigma \in \PP \mathrm{H}^0(X,E)$.  Its fibre $\hH_{X,\ell,[C]}$ 
over a fixed (integral) spectral curve $[C] \in\aA$ parameterizes
torsion free sheaves with a section; since $C$ is locally planar and hence Gorenstein
we have $\hH_{X,\ell,[C]} \cong C^{[\ell']}$ \cite[Appendix B]{PT3} where 
as before $\ell'= \ell + \binom{n}{2} \deg L$.  Let $\para{\hH}_{X,\ell} = \hH_{X,\ell}
\times_{[\gg/G]} [\tilde{\gg}/G]$ be the
parabolic version,  and $\hH_{X,\ell}^{(m)}$ the locus 
where $\phi_x$ is nilpotent and $\dim E_x / \phi_x E_x = m$.  

Assume $C$ is nonsingular away from its fibre over $x$ and $\phi_x$ is nilpotent.  In \cite{ORS}
we used the notation $C_{m}^{[\ell']}$ to denote the locus in the Hilbert scheme of points
where the ideal sheaf requires $m$ generators; we have seen above that this may be identified 
with $\para{\hH}_{X,\ell,[C]}^{(m)}$.   

We have the binomial identity 

\[\sum_{i=0}^r \binom{r}{i}_{t^2} A^i t^{i(i-1)} = \prod_{i=0}^{r-1} (1+t^{2i} A) \] 

which we may use in conjunction with the analogue for $\widetilde{\hH}$ of Lemma \ref{lem:covariants-hitchin}
to conclude 

\[Z(C) = (1+a^2 t) \sum_{d', k} q^{2d'} (a^2 t)^{k} \wtp_t \Hom_{S_n} (\Lambda^k \hh ,
\mathrm{H}^*(\widetilde{\hH}_{X,\ell,[C]})) \]
Thus we have exchanged explicit mention of the number of generators of ideals for taking
isotypic components with respect to the Springer action on a parabolic object. 

We now recall from \cite{MY,MS} how to replace the Hilbert schemes with Hitchin fibres.
As always we work over the locus $\aA^\heartsuit$ of integral spectral curves. 
Since $\mM_{X,\ell}$ is smooth and $h$ is proper, 
the decomposition theorem of Beilinson, Bernstein, and Deligne \cite{BBD}
implies that 
$h_* \CC_{\mM_{X,\ell}}$ splits as a direct sum of shifted semisimple perverse sheaves.  
From the support theorem of Ng\^o \cite{Ngo-fundamental}, we know that moreover 
these constituents are all IC sheaves associated to the local systems of cohomologies of
fibres on the smooth locus $\overline{h}$ of $h$.  In particular, at the fibre over some 
fixed spectral curve $C$ we obtain a decomposition 
\[\mathrm{H}^i (\mM_{X,\ell,[C]})  \cong \bigoplus_j \ecH^i(
\mathrm{IC}(\mathrm{R}^j \overline{h}_* 
\CC)[-j-\dim \aA-g]|_{[C]})\]
On the RHS the restriction of the IC sheaf to a point gives some complex, and we are just 
taking its cohomology as a complex. 
We will denote by $\mathrm{H}^{i;j} (\mM_{X,\ell,\alpha})$ the $j$-th summand on the RHS.\footnote{
Because the general fibres are abelian varieties, one can produce a canonical isomorphism.
Moreover, {\em any} family of compatified Jacobians with smooth total space will induce the 
same isomorphism, so the decomposition depends only on $[C]$ and not the way in which it
appears as a spectral curve. }

Let $f:\hH_{X,\ell} \to \mM_{X,\ell}$ be the forgetful map. 
It is shown in \cite{MY,MS} that $\hH_{X,\ell}$ is nonsingular, and that the shifted perverse summands
of $h_* f_* \CC_{\hH_{X,\ell}}$ all have support equal to $\aA_{X}$.  Consequently we may 
compare $h_* f_* \CC_{\hH_{X,\ell}}$ and $f_* \CC_{\mM_{X,\ell}}$ along the common smooth locus of 
both maps (ie, the locus of smooth spectral curves), take IC sheaves to get a comparison
over all of $\aA_{X}$, and then restrict to the (singular) spectral curve of interest. 
Thus in \cite{MY,MS} it is proven that: 
\[(1-q^2)(1-q^2 t^2) \sum_{\ell'=0}^\infty q^{2\ell'} \mathrm{H}^*(\hH_{X,\ell,[C]}) = 
\sum_{i=0}^{2g} q^{2i} \mathrm{H}^{*;i}(\mM_{X,0,[C]})  \]
The equality is an equality of mixed Hodge structures, 
where $t$ indicates the inverse Tate twist (i.e., it becomes multiplication by $t$ after taking 
weight polynomial).  We have chosen degree $0$ on the right for specificity, but
the choice is arbitrary. 

The analogous equality holds $S_n$-equivariantly in the parabolic case.  Over the 
open locus
$\aA^{rs}$ where the spectral curve consists of $n$ distinct points over $x$ this is clear: 
the maps $\para{\hH} \to \hH$ and $\para{\mM}_X \to \mM_X$ are just $S_n$ torsors.  Therefore
it suffices to show that $h_* f_*  \CC_{\para{\hH}}$ and $f_* \CC_{\para{\mM}}$
are the intermediate extensions of their restrictions to $\aA^{rs}$.  This fact is
established for $f_* \CC_{\para{\mM}}$ by Yun \cite{Y}, and 
a combination of the arguments there with the appropriate $S_n$ 
enrichment of the arguments in \cite{MY}  yields the result for 
$h_* f_*  \CC_{\para{\hH}}$  (see \cite{OY} for details).  
Thus we  conclude: 

\begin{theorem}
Let the reduced curve $C$ appear as a spectral curve in the Hitchin system.  Then,
\begin{eqnarray*}
Z(C) & := & \sum_{d, m} q^{2d} \wtp_t(C^{[d]}_m) \sum_{k=0}^m \binom{m}{k}_{t^2} 
a^{2k} t^{k^2} \\ 
&  =  & \frac{1+a^2 t}{(1-q^2)(1-q^2 t^2)} \sum_{i=0}^{2g} q^{2i} (a^2 t)^k \wtp_t(\Hom_{S_n}(\Lambda^k \hh,
\mathrm{H}^{*;i}(\para{\mM}_{X,0,[C]})))
\end{eqnarray*}
\end{theorem}

Thus the conjecture of \cite{ORS} for the {\em reduced} homology {\em of a knot} reads 

\[ 
\pP_{\mathrm{link}(C,x)} = (1+q^2 t)^{-2\tilde{g}} (aq^{-1})^\mu
\sum_{i=0}^{2g} q^{2i} (a^2 t)^k \wtp_t (\Hom_{S_n}(\Lambda^k \hh,
\mathrm{H}^{*;i}(\para{\mM}_{X,0,[C]})))
\]

and when $C$ is rational, which may always be arranged,

\[ 
\pP_{\mathrm{link}(C,x)} = (aq^{-1})^{\mu}
\sum_{i=0}^{\mu} q^{2i} (a^2 t)^k \wtp_t (\Hom_{S_n}(\Lambda^k \hh,
\mathrm{H}^{*;i}(\para{\mM}_{X,0,[C]})))
\]

In other words, the cohomological degree on the Hitchin fibre should match the 
grading $\mu/2 - h$ on the HOMFLY homology.

\subsection{From Hitchin fibers to  Cherednik algebras} 

In this section we discuss the results of \cite{OY}, where the representations $L_{n/m}$ of the rational
Cherednik algebra (of type A) are constructed in the cohomology of certain Hitchin fibers associated
to $G = \mathrm{SL}_n$.  The difference between the $\mathrm{SL}_n$ case and the $\mathrm{GL}_n$ 
case discussed previously
is mild and will be explained below.  
As before fix a smooth curve $X$ and line bundle $L$; let $G$ be
any reductive group.  Then for a scheme $S$, the $S$-points of
the Hitchin moduli stack $\mM_{X,G}$ is the groupoid of Hitchin pairs 
$(\mathcal{E},\varphi)$ where:
\begin{itemize}
 \item $\mathcal{E}$ is a $G$-torsor over $S\times X$
 \item $\varphi\in \ecH^0(S\times X,Ad(\mathcal{E})\otimes_{\mathcal{O}_X} L)$, $Ad(\mathcal{E})
 =\mathcal{E}\times_{G}\mathfrak{g}$
\end{itemize}
If we choose a point $x\in X$ then by means of evaluation map at $x$ we can define the parabolic
version of the stack $\para{\mM}_{X,G}:=\mM_{X,G}\times_{\mathfrak{g}/G}\tilde{\mathfrak{g}}/G$.
(This space is often denoted $\mM_{X,G}^{par}$.) 
As in the case of moduli spaces from the previous section, the moduli stacks from above 
come equipped with the natural map $h$ to the Hitchin base $\mathcal{A}_{X,G}$ 

When $\deg(L)\ge 2g_X$ the restriction of the moduli stack  $\mM_{X,G}$ and $\para{\mM}_{X,G}$
on $\aA^{\heartsuit}$
is a smooth Deligne-Mumford stack \cite{Ngo-endoscopy,Y}. The moduli spaces 
$\mM_{X,\ell}$ and $\para{\mM}_{X,\ell}$ considered in the previous section are
open substacks of $\mM_{X,\mathrm{GL}_n}$ and $\para{\mM}_{X,\mathrm{GL}_n}$.  
There is an embedding of the Hitchin bases $\aA_{X,\mathrm{SL}_n}\to
\aA_{X,\mathrm{GL}_n}$ and we can consider the 
map  $\mM_{X,\mathrm{GL}_n}|_{\aA_{X,\mathrm{SL}_n}}\to \mathrm{Pic}(X)$ which
takes the determinant of the bundle $E$; the fiber over zero gives $\mM_{X, \mathrm{SL}_n}$. 
In particular when $X = \PP^1$ 
the fibers of $\mathrm{GL}_n$ and $\mathrm{SL}_n$ Hitchin maps are identical, so
the results on the Hitchin fiber from the previous section hold for the $\mathrm{SL}_n$ Hitchin fibers.

Now we restrict our attention to the case $X=\mathbb{P}^1$, $G=\mathrm{SL}_n$ and
we omit these objects in the notations for the Hitchin stacks. Let us fix homogeneous coordinates $[\xi:\eta]$
and $0=[0:1]$, $\infty=[1:0]$; we also fix $x=0$ in the definition of $\para{\mM}$.  Let $L=\mathcal{O}(d)$,
$d\ge m/n$, and consider the curve $C$ cut out by the section $\xi^m\eta^{dn-m}\in 
\ecH^0(\PP^1, L^{\otimes n})$.   $C$ is smooth outside of the fibers over the points $0$ and $\infty$, where
it has singularities given respectively by the equations $y^n=x^m$ and $y^n=x^{m-nd}$.
By the product formula \cite{Ngo-endoscopy} we can express (topologically) this Hitchin fibre as the
product of affine Springer fibers 
$$ \para{\mM}_{[C]} \sim \para{\Sp}_{n,m} \times \Sp_{n,nd-m}$$  Here
$\Sp_{a,b}$ is the affine Springer fiber with spectral curve $y^a = x^b$ in the 
affine Grassmannian for $\mathrm{SL}_a$, and as always the tilde denotes
the parabolic version, which sits inside the affine flag manifold \cite{KL}.  

If $C_{m,n}$ is a rational spectral curve (for a Hitchin system with some other line bundle $L$) with a unique singularity
of the type $x^m = y^n$, then according to Laumon \cite{Laum} there
is a 
homeomorphism
$\mM_{[C_{m,n}]} \sim \Sp_{n,m}$; the left hand side is just the compactified Jacobian of $C_{m,n}$.  
Introducing parabolic
structure at the point on the base curve beneath the singularity we have similarly 
$\para{\mM}_{[C_{m,n}]} \sim \para{\Sp}_{n,m}$.

As explained in the previous section the cohomology of the Hitchin fiber carries a natural
filtration $P$ coming from the perverse $t$-structure on $D^b_c(\aA)$.   This filtration is
centered at zero, but in the notation
of the previous section we had put a shift so that 
$$\ecH^{i;j} (\mM_{X,[C]}) = \mathrm{Gr}_{j-g(C)}^P \ecH^i(\mM_{X,[C]}).$$

In \cite{Y} it is shown that the perverse filtration is compatible
with the product formula in the sense that one can induce a local version of this filtration
on the cohomology of the local factors $\Sp_{n,m}$. Thus we can construct the associate graded space
$$Gr^P \ecH^*(\para{\Sp}_{n,m})=\bigoplus_i P_{\le i} \ecH^*(\para{\Sp}_{n,m})/P_{\le i-1}
\ecH^*(\para{\Sp}_{n,m}).$$ 

In \cite{Y} the action of the trigonometric DAHA is constructed on the equivariant cohomology of the 
Hitchin system $\mathrm{R}h_* \CC_{\para{\mM}}$ with respect to the scaling action on the Hitchin base,
which is 
induced by scaling the line bundle $L$.  
However, taking a non-equivariant fibre (i.e. the fibre over any spectral curve other 
than a multiple of the zero section) comes at
the cost of losing the interesting non-commutative structure.  The virtue of taking the base $\PP^1$ and 
the fibre over curve $C$ introduced above is there is a second $\CC^*$ action on $\mM_{\PP^1}$
coming from scaling $\PP^1$, and in the fibre $\mM_{\PP^1, [C]}$ there remains a $\CC^*$ action 
coming from balancing the scaling on $\PP^1$ and on $L$ to preserve $C$.  
The basic idea of \cite{OY} is that the representations of the {\em rational} DAHA can 
be constructed by using this $\CC^*$ action in the constructions of \cite{Y}, 
specializing the equivariant parameter to $1$, and passing 
to the associated graded with respect to the perverse filtration.

\begin{theorem}\cite{OY} The vector space $Gr^P \ecH^*(\para{\Sp}_{n,m})$ carries geometrically
constructed endomorphisms endowing it with the structure of a module over 
$\mathbf{H}_{m/n}$; it is the irreducible representation $L_{m/n}$.  The perverse degree
is the internal grading.
\end{theorem}

Specializing
the equivariant cohomology $\ecH^*_{\CC^*}(\para{\Sp}_{n,m}) = \ecH^*(\para{\Sp}_{n,m})[\omega]$ at
$\omega = 1$ means 
$$\ecH^k_{\CC^*}(\para{\Sp}_{n,m}) = \bigoplus_i \ecH^{k-2i}(\para{\Sp}_{n,m}) \omega^i \to 
\bigoplus_i \ecH^{k-2i}(\para{\Sp}_{n,m})$$ 
so operations which respect the degree in equivariant cohomology will now only respect the filtration
indicated on the RHS.  
We saw in the previous section that the cohomological grading on 
$\para{\Sp}_{n,m}$ should correspond to the grading $\mu/2 - h$ on the HOMFLY homology.
The relation between $h$ and the filtration grading explained in the introduction
suggests the following definition:
 $$ \mathcal{F}^{geom}_i L_{m/n}= \!\!\!\!\!\!\!\!\!\!\!\! \bigoplus_{k-j + (m-1)(n-1) \le i}  
 \!\!\!\!\!\!\!\!\!\!\!\!\! Gr^P_{k} \ecH^{j}(\para{\Sp}_{n,m}).$$
The cohomological degree is always bounded below by the 
dimension $(m-1)(n-1)/2$ plus the perverse degree, so this is an increasing filtration 
which must stabilize at $i = (m-1)(n-1)/2$.  

\begin{proposition}\cite{OY} The filtration $\fF^{geom}$ is compatible with
the filtration on $\mathbf{H}_{m/n}$.
\end{proposition}

We conclude: 

\begin{theorem}
	Conjecture \ref{conj:one} with $\fF = \fF^{geom}$ is implied by
	Conjecture 2 of \cite{ORS}. 
\end{theorem}

The affine Springer fibres enjoy certain relations.  Denoting by 
$\Sp_{a,b}^{(c)}$ the sublocus where the endomorphism has kernel of rank 
$c$ at the central point $x$. 
Then there is an identification $\Sp_{n,n+m}^{(n)} \to \Sp_{n,m}$ given by 
replacing the bundle $E$ by the kernel of the map $E \to \phi_x E_x$, the line 
bundle $L$ by $L(x)$, and the automorphism $\phi$ by $\phi/t$ where $t$ is some
local coordinate at $x$.  On the other hand we have seen in the previous section that
$\ecH^*(\Sp_{n,n+m}^{(n)}) = t^{n(n-1)} \Hom_{S_n} (\Lambda^n \hh_n, 
\ecH^*(\para{\Sp}_{n,n+m}))$.  Composing these equalities gives the geometric incarnation
of the identification $\eee L_{m/n} = \eee_- L_{(n+m)/n}$, and it will be shown
in \cite{OY} that this identification is compatible with the algebra actions.  
Similarly there is (at least topologically) an identification 
$\Sp_{n,m} \sim \Sp_{m,n}$ -- both spaces are identified with the compactified Jacobian
of the same singularity -- and it will be shown that the algebra actions are compatible
with this identification $\eee L_{m/n} = \eee L_{n/m}$.  Thus from Theorem \ref{thm:ind} 
we conclude:

\begin{proposition}
	$\fF^{ind} \subset \fF^{geom}$.  
\end{proposition}

We have conjectured these filtrations always agree.  As evidence we have
the following result:

\begin{theorem}
	$\fF^{ind} \eee L_{(mn+1)/n} = \fF^{geom} \eee L_{(mn+1)/n}$. 
\end{theorem}
\begin{proof}
	Since we have an inclusion of filtrations we need only compare the 
	associated graded dimensions of the given spaces.  For the inductive
	filtration, according to \cite{GS, GS2} this amounts to computing a certain
	specialization of the (generalized) $q,t$-Catalan numbers; the desired specialization
	is computed in \cite[Thm. 4.4]{GH} and is given by the sum over Young diagrams
	contained in the triangle of sides $mn+1, n$ weighted with the area.  
	It is shown in \cite[Cor. 1.2]{GM1}
	that the same formula gives the Poincar\'e series of the compactified Jacobian
	of the relevant singularity, i.e., the associated graded dimensions of 
	$\fF^{geom}$. 
\end{proof}

\subsection{Fixed points and parking functions}

Let $m$ and $n$ be two coprime integers.

\begin{definition}
We call a function $f:\{1,\ldots,n\}\to\{1,\ldots,m\}$
a {\em $\frac{m}{n}$--parking function} ,
if for every $k$ one has
$$|f^{-1}([1,k])| \ge \frac{kn}{m}.$$

The set of all $\frac{m}{n}$--parking functions
will be denoted as $\pf_{\frac{m}{n}}.$
\end{definition}

$\pf_{\frac{m}{n}}$ has a natural action of $S_n$ by permutation of 
elements in the source. 

\begin{proposition}
The number of $\frac{m}{n}$--parking functions equals to $m^{n-1}$.
\end{proposition}

\begin{proof}
The proof is analogous to the case $m=n+1$.
There are $m^{n}$ functions from $\{1,\ldots,n\}$ to $\{1,\ldots,m\}$,
consider the action of the cycle $(1,\ldots,m)$. One can check that in every orbit 
there is precisely one $\frac{m}{n}$--parking function.
\end{proof}

Let us draw "parking function diagrams", generalizing the similar pictures from \cite{haglund}.
Consider a $m\times n$ rectangle and lattice path below the diagonal in it.
Let us write numbers from $1$ to $n$ above this path such that 
\begin{itemize}
\item[(i)]{In every column there is exactly one number}
\item[(ii)]{In every row the numbers are decreasing from left to right}
\end{itemize} 

Parking function diagrams are in one-to-one correspondence with the parking functions, where
a function corresponding to a diagram is just $y$-coordinate.

\begin{example}
Consider a $\frac{3}{4}$--parking function
$$\left(\begin{matrix}
1& 2& 3& 4\\
1& 3& 1& 2\\
\end{matrix}\right)$$
Its diagram is shown in Figure \ref{pfdiagram}.

\begin{figure}
\begin{tikzpicture}
\draw  (0,0)--(0,3)--(4,0)--(0,0);
\draw [dashed] (1,0)--(1,9/4);
\draw [dashed] (2,0)--(2,6/4);
\draw [dashed] (3,0)--(3,3/4);
\draw [dashed] (0,1)--(8/3,1);
\draw [dashed] (0,2)--(4/3,2);
\draw [thick] (0,3)--(0,2) -- (1,2) -- (1,1)--(2,1)--(2,0)--(4,0);
\draw (0.2,2.2) node {2};
\draw (1.2,1.2) node {4};
\draw (2.2,0.2) node {3};
\draw (3.2,0.2) node {1};
\end{tikzpicture}
\caption{Example of a $\frac{3}{4}$--parking function}
\label{pfdiagram}
\end{figure}
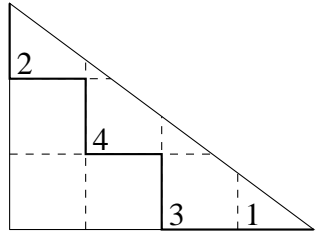
\end{example}

Let $\Gamma=\Gamma^{m,n}=\{am+bn: a,b\in \ZZ_{\ge 0}\}$ denote the integer semigroup generated by $m$ and $n$.
Recall that (\cite{LS}, \cite{Pi}) the compactified Jacobian $\Sp_{n,m}$ admits the torus action,
whose fixed points correspond to the semi-modules over $\Gamma^{m,n}$.
It was shown in \cite{GM} that such semi-modules are  in 1-to-1 correspondence with the Young diagrams in the $m\times n$ rectangle $R$ below the diagonal.

Let us label the boxes of $R$  with integers, so that the shift by $1$ up subtracts
$n,$ and the shift by $1$  to the right subtracts $m.$ We normalize these numbers so that
$mn$ is in the box $(0,0)$ (note that this box is not in the rectangle $R,$ as we start enumerating
boxes from $1$). In other words, the numbers are given by the linear function $f(x,y)=kn-kx-ny.$

One can see that the labels of the boxes below the diagonal are positive, while all other numbers in
$R$ are negative. Moreover, numbers below the diagonal are exactly the numbers from the
complement $\mathbb Z_{\ge 0}\backslash\Gamma,$ and each such number appears only once. 

\begin{definition}(\cite{GM})
For a  $0$-normalized $\Gamma$--semi-module $\Delta$, let $D(\Delta)$ denote the set of boxes with labels belonging to 
$\Delta\setminus\Gamma$.
\end{definition}

\begin{definition}(\cite{lowa})
Let $D$ be a Young diagram, $c\in D$. Let $a(c)$ and $l(c)$ denote the lengths of arm and leg for $c$.
For each real nonnegative $x$  define
$$h^{+}_{x}(D)=\sharp\left\{c\in D~~\vline~~{a(c)\over l(c)+1}\le x< {a(c)+1\over l(c)}\right\}.$$
\end{definition}

The following theorem is the main result of \cite{GM}.

\begin{theorem}
\label{piontcell}
The dimensions of cells in the compactified Jacobian can be expressed through the $h^{+}$ statistic:
$$\dim C_{\Delta}=\frac{(k-1)(n-1)}{2}-h^{+}_{\frac{n}{m}}(D(\Delta)).$$
\end{theorem}
 
One can see that under this correspondence a $\frac{m}{n}$--parking function diagram corresponds to a flag of $\Gamma^{m,n}$-semimodules
$$\Delta_{1}\supset\Delta_{2}\supset \ldots \supset \Delta_{m}\supset \Delta_{m+1}=\Delta_{1}+m$$
such that $|\Delta_{i}\setminus \Delta_{i+1}|=1.$ According to \cite{LS}, such flags parametrize the fixed points of the natural torus action in
$\widetilde{\Sp}_{n,m}$. Therefore we arrive at the following

\begin{proposition}
The variety $\widetilde{\Sp}_{n,m}$ admits a torus action with a finite number of fixed points. These fixed points are in 1-to-1 correspondence
with the $\frac{m}{n}$--parking functions.
\end{proposition}

It was shown in \cite{LS} that $\widetilde{\Sp}_{n,m}$ admits an algebraic cell decomposition with affine cells corresponding to the fixed points of the torus action.
We plan to compute the dimensions of these cells and compare them with the combinatorial statistics of \cite{HHLRU} and \cite{armstrong} in the future.

We finish with the combinatorial conjecture from \cite{ORS} describing the character of $\Hmn$.

\begin{definition}
Consider a diagram $D$ corresponding to a semigroup module $\Delta$.
Let $P_m$ denote the numbers in the SE corners, $Q_i$ denote the numbers in the $ES$ corners. Then
$$\beta(P_m)=\sum_{i}\chi(P_i>P_m)-\sum_{i}\chi(Q_i>P_m).$$
\end{definition}

\begin{example}
\label{ex3221}
Consider a semigroup generated by $5$ and $6$, and a module 
$$\Delta=\{0,1,2,5,6,\ldots\}.$$
Its diagram has a form:

%\ifhavetikz
\begin{tikzpicture}
\draw  (0,0)--(0,4);
\draw  (1,0)--(1,3);
\draw  (2,0)--(2,2);
\draw  (3,0)--(3,2);
\draw  (4,0)--(4,1);
\draw  (0,0)--(5,0);
\draw  (0,1)--(4,1);
\draw  (0,2)--(3,2);
\draw  (0,3)--(1,3);
\draw (0.5,0.5) node {$19$};
\draw (0.5,1.5) node {$14$};
\draw (0.5,2.5) node {$9$};
\draw (1.5,0.5) node {$13$};
\draw (1.5,1.5) node {$8$};
\draw (2.5,0.5) node {$7$};
\draw (3.5,0.5) node {$1$};
\draw (2.5,1.5) node {$2$};
\draw (0.5,3.5) node {$\bf{4}$};
\draw (1.5,2.5) node {$\bf{3}$};
\draw (4.5,0.5) node {$\bf{-5}$};
\draw (1.5,3.5) node {$\bf{-2}$};
\draw (3.5,2.5) node {$\bf{-9}$};
\draw (4.5,1.5) node {$\bf{-10}$};
\draw (3.5,1.5) node {$\bf{-4}$};
\end{tikzpicture}
%\fi

We have $$\{P_i\}=\{-5,-4,3,4\},\quad \{Q_j\}=\{-10,-9,-2\}.$$
Therefore $$\beta(-5)=3-1=2,\quad \beta(-4)=2-1=1,\quad \beta(3)=1,\quad \beta(4)=0.$$
\end{example}

\begin{conjecture}
The triply graded character of $\Hmn$ can be computed as a sum over Young diagrams in $m\times n$ rectangle below the diagonal:
\begin{equation}
\label{JCfull}
\mathcal{P}_{m,n}(a,q,t)=\sum_{D}q^{2|D|+2h^{+}_{\frac{n}{m}}(D)}t^{2|D|}\prod_{j=1}^{r}(1+a^2q^{-2\beta(P_j)}t).
\end{equation}
\end{conjecture}

For $a=0$ this conjecture gives a refinement of Theorem \ref{piontcell}, providing a combinatorial description of the perverse filtration on the cohomology of the compactified Jacobian. It was shown in the Appendix A.3 of \cite{ORS} that for $m=n+1$ the formula (\ref{JCfull}) agrees with the combinatorial statistics describing the $q,t$--multiplicities of hooks in the diagonal harmonics, conjectured in \cite{EHKK} and proved in \cite{hagl2}. 

The authors wrote a computer program calculating the right hand side of (\ref{JCfull}), its output is available by request. 
In all known examples the results agree with the ones of \cite{as}, \cite{ch} and \cite{mm}, obtained by completely different methods.

\end{document}